\documentclass[12pt]{article}

\usepackage{PRIMEarxiv}

\usepackage[T1]{fontenc}    % use 8-bit T1 fonts
\usepackage{hyperref}       % hyperlinks
\usepackage{url}            % simple URL typesetting
\usepackage{booktabs}       % professional-quality tables
\usepackage{amsfonts}       % blackboard math symbols
\usepackage{nicefrac}       % compact symbols for 1/2, etc.
\usepackage{natbib}
\usepackage{microtype}      % microtypography
\usepackage{lipsum}
\usepackage{amssymb}
\usepackage{fancyhdr}       % header
\usepackage{amsmath}
\usepackage{enumitem}
\usepackage{booktabs}
\usepackage{tikz}
\usetikzlibrary{arrows.meta, shapes.geometric}
\usetikzlibrary{calc, arrows.meta}
\usepackage{subcaption}   % for \begin{subfigure}

\newtheorem{definition}{Definition}
\newtheorem{proposition}{Proposition}
\newtheorem{corollary}{Corollary}

\newtheorem{lemma}{Lemma}
\newtheorem{proof}{Proof}
\usepackage{algorithm}
\usepackage{algpseudocode}
\usepackage{algcompatible}
\usepackage{tikz}
\usepackage{multirow}
\usepackage{setspace}
\usetikzlibrary{positioning,calc,fit,arrows.meta}

\usepackage{graphicx} % Required for inserting images
\setkeys{Gin}{draft=false}

\usepackage{xcolor} % Required for color support
\usepackage{todonotes}
\title{Learning to Control Stabilization in \\Column Generation
}

\author{
 Olivia Wang, Reem Khir\\
 Edwardson School of Industrial Engineering \\
Purdue University, West Lafayette, USA \\
  \texttt{\{jiew,rkhir\}@purdue.edu} \\
}

\begin{document}
\maketitle

\begin{abstract}
Column generation is a widely used decomposition technique for large-scale linear programs, but it often suffers from slow convergence due to poor initial dual estimates and dual oscillations. 
Stabilization techniques such as smoothing and penalization can mitigate these issues, but their effectiveness depends heavily on parameter selection, which requires careful tuning to avoid 
degrading performance. This paper presents a common framework for smoothing and penalization, showing that despite their different mechanisms, both are governed by two design choices: a \textit{reference point} in the dual space and \textit{stabilization parameters} that regulate how strongly 
that reference influences pricing. Within this framework, we derive parameter bounds that ensure progress, analyze predicted duals as reference points, and establish convergence guarantees for both methods. These results motivate and guide the design of RLSCG, a reinforcement learning-guided framework that adaptively selects stabilization parameters at each iteration. Computational experiments on the Cutting Stock Problem show that RLSCG substantially reduces iteration count and computation time on most synthetic and benchmark instances relative to traditional column generation, rule-based adaptive stabilization, and learning-based column selection, with the largest gains on large-scale instances.

%Column Generation is a widely used decomposition technique for large-scale linear programming, but it generally suffers from poor initial dual estimates and slow convergence due to dual oscillations. While stabilization techniques such as smoothing and penalization can mitigate these issues, their effectiveness heavily depends on the choice of parameters, and inappropriate parameter selection can even lead to adverse effects. To enhance the performance of stabilization methods, this paper first presents a unified framework that formally connects smoothing and penalization methods to better understand how stabilization brings improvements, revealing their shared underlying mechanism of controlling dual fluctuations through stabilization parameters and reference points. Building on this, we propose RLSCG, a reinforcement learning-guided stabilization framework that dynamically selects stabilization parameters during iterations. Experimental results on the Cutting Stock Problem demonstrate that RLSCG significantly reduces solution time and iteration count compared to traditional column generation, adaptive stabilization methods, and recent learning-based approach.
\end{abstract}

% keywords can be removed
\keywords{Column Generation\and Stabilization\and Reinforcement Learning\and Large-scale Optimization }

\section{Introduction}
Column generation (CG) is a decomposition framework for solving large-scale linear programs (LPs) whose full variable set is too large to enumerate explicitly. Rather than solving the full master problem directly, CG starts with a restricted master problem (RMP) containing only a subset of variables. The dual solution of the RMP is then used to construct a pricing subproblem that identifies variables with negative reduced cost, which are added to the RMP prior to reoptimization. The procedure terminates when no such variable exists, at which point the current RMP solution is optimal for the full LP master problem.
Although CG is an LP-based method, it is widely applied in integer programming via branch-and-price across domains \citep{desaulniers2005column,maher2023integer}.

Despite its successes, the standard CG procedure faces inherent drawbacks when applied to large-scale problems. First, because the RMP initially contains only a small subset of variables, its dual solution may provide a poor approximation of the dual solution of the full master problem. Pricing may therefore generate variables that improve the current RMP but do not appear in the final optimal solution, a phenomenon known as the \textit{heading-in effect}. Second, degeneracy in the master problem may produce multiple optimal dual solutions, causing the dual multipliers to oscillate across iterations. This instability can lead to variables with little effect on the RMP objective and hence to slow convergence, often referred to as the \textit{tailing-off effect} \citep{vanderbeck2005implementing}.

To address these issues, various stabilization techniques have been developed for column generation, which is the focus of this paper. 
The central idea is to limit large fluctuations in the dual variables between iterations, thereby guiding the algorithm more steadily toward the optimal dual solution \citep{desaulniers2006column}. These techniques broadly fall into two categories. The first modifies the master problem to control dual behavior through bounds or penalty terms, including the \textit{boxstep} method of \cite{marsten1975boxstep}, which confines dual multipliers to prescribed intervals, and \textit{penalization} approaches, which impose costs for deviations from a reference region using quadratic \citep{schramm1992version}, linear--quadratic \citep{pinar1994smoothing}, or piecewise-linear penalties \citep{du1999stabilized}. The second category consists of \textit{smoothing} methods \citep{wentges1997weighted,neame1999nonsmooth,pessoa2018automation}, which post-process the dual values by computing the pricing dual as a convex combination of the current dual solution and a feasible interior point. A related line of work maintains centralized dual solutions via interior-point methods \citep{elhedhli2004integration, munari2015column}.
%rousseau2007interior,

Although these techniques can accelerate convergence, their effectiveness depends strongly on the choice of parameters. Each scheme introduces its own: penalization methods require penalty coefficients or box sizes, smoothing methods require a smoothing weight and a reference point update rule, and interior-point approaches require centrality tolerances.  In practice, these parameters are typically set heuristically or held fixed across iteration \citep{du1999stabilized,wentges1997weighted,neame1999nonsmooth}. Yet this can be problematic---\cite{pessoa2018automation} show that convergence speed is highly sensitive to the smoothing weight, with effective values varying across problem instances, and similar sensitivity has been reported for penalization parameters \citep{amor2009choice}. When poorly tuned, stabilization can reduce computational savings or even underperform standard CG.

% Recent work has sought to address this through adaptive mechanisms. 
Recent work has improved stabilization through adaptive rules. For smoothing, \cite{pessoa2018automation} reduce the smoothing weight upon detecting \emph{mispricing}---when the stabilized dual fails to generate an improving column. For penalization, \cite{kraul2023machine} use regression-predicted duals as a fixed box center, while \cite{shen2024adaptive} adapt the penalty based on minimum reduced cost; both keep the reference dual fixed and update only the penalty parameter. 
Exploring other penalization types, \cite{sugishita2024use} learn initial dual values to warm-start quadratic penalization. Similarly, \cite{sarkar2025accelerating} also predict duals, but use them only to initialize the box center before reverting to the standard update based on the previous dual iterate.
% \edit{
In a parallel line of work, \cite{fang2025learning} train a reinforcement learning (RL) agent that directly outputs the perturbed RMP dual at each iteration, rather than modulating an existing smoothing or penalization mechanism; their action space is dimension-dependent, scaling with the number of constraints.
% }
% . While innovative, this approach requires a continuous action space that scales with the number of constraints and lacks a convergence fallback.

Beyond stabilization, 
ML has been used to accelerate CG through column selection: \cite{chi2022deep} applied RL to select from a candidate set, later extended to multiple columns \citep{yuan2024reinforcement,hu2024ffcg}. These methods do not modify the dual used in pricing, so the candidate pool arises from the unstabilized RMP dual and may miss columns that stabilization would surface.
This raises the question of whether learning is more effective for \emph{column selection} or for adapting the \emph{stabilization mechanism} itself; a question we investigate on the cutting stock problem.
% Beyond stabilization, machine learning has been applied to accelerate CG through column selection. \cite{chi2022deep} applied RL to select columns from a candidate set, later extended to multiple columns \citep{yuan2024reinforcement,hu2024ffcg}. While these methods influence convergence through column choice, they do not directly modify or regularize the dual solution used in pricing. Consequently, the candidate pool is still generated from the unstabilized RMP dual, potentially missing columns that would arise under stabilization.
% This raises the question of whether learning is more effective for \emph{column selection} or for adapting the \emph{stabilization mechanism} itself; a question we investigate in the context of the cutting stock problem.

This paper focuses on two widely used classes of stabilization methods, penalization and smoothing, both of which integrate naturally with the standard simplex-based CG procedure. Though these methods operate through different mechanisms, we investigate whether they admit a unified view of dual stabilization and whether that view can support a general adaptive parameter-selection approach for accelerating convergence. Our contributions are as follows.

\begin{enumerate}

    \item We present smoothing and penalization within a common structure defined by a dual-space \textit{reference point} and \textit{stabilization parameters} that control its influence on pricing. Within this framework, we derive parameter bounds ensuring progress and establish convergence guarantees.

    \item We propose RLSCG, a reinforcement learning--guided framework for stabilized column generation that replaces fixed or rule-based parameter updates with a learned policy. Although the bounds derived above are not enforced online, they inform the policy's design: its state representation combines local structural features of the current RMP with global progress and mispricing-related signals to adaptively select the stabilization level at each iteration.

    % \item We evaluate RLSCG on the cutting stock problem against traditional column generation, rule-based stabilization, and learning-based column selection. Results show substantial reductions in iteration count and computation time across diverse instances and unseen patterns. \edit{In particular, while rule-based stabilization can underperform traditional CG, RLSCG consistently improves upon it, highlighting that stabilization can accelerate convergence only when its magnitude is carefully controlled.}

    % \item \edit{We extend the use of predicted duals as reference points, previously studied for penalization, to smoothing via a hybrid strategy with provable convergence. Empirically, predicted duals substantially benefit rule-based smoothing, partially closing its gap with RLSCG.}

    \item     We evaluate RLSCG on the cutting stock problem against traditional column generation, rule-based stabilization, and learning-based column selection. Results show substantial reductions in iteration count and computation time across diverse instances and unseen patterns. Notably, while rule-based stabilization can sometimes underperform traditional CG, RLSCG consistently accelerates convergence by carefully controlling the stabilization magnitude. Furthermore, we extend the use of predicted duals as reference points to smoothing via a hybrid strategy with provable convergence, which empirically benefits rule-based smoothing and partially closes its gap with RLSCG.

\end{enumerate}
% \begin{enumerate}
 
%     \item We present smoothing and penalization within a common structure defined by a dual-space \textit{reference point} and \textit{stabilization parameters} controlling their influence on pricing. Within this framework, we derive parameter bounds ensuring progress, analyze predicted duals as reference points, and establish convergence guarantees.
    
%     \item We propose RLSCG, a reinforcement learning–guided framework for stabilized column generation that replaces fixed or rule-based parameter updates with a learned policy. Although the bounds derived above are not enforced online, they inform the policy’s design: its state representation combines local structural features of the current RMP with global progress and mispricing-related signals to adaptively select the stabilization level at each iteration.

%     \item We evaluate RLSCG on the cutting stock problem against traditional column generation, rule-based stabilization, and learning-based column selection approaches. The results demonstrate substantial reductions in computation time and iteration count across diverse instance structures, including patterns not seen during training.
% \end{enumerate}

The remainder of this paper is organized as follows. Section~\ref{sec:prelim} provides preliminary background on stabilization. Section~\ref{sec:framework} introduces the unified stabilization framework, analyzing the impact of parameters and reference points. In Section~\ref{sec:method}, we detail the RLSCG framework. Section~\ref{sec:experiment} presents our experimental results, and Section~\ref{sec:conclusion} concludes the paper.

\section{Preliminaries}\label{sec:prelim}

% Consider the master problem (MP) as a linear program in standard form:
% $$
% \min_{x \ge 0}\; c^\top x \qquad \text{s.t.}\qquad Ax=b,
% $$
% with its dual formulation:
% $$
% \max_{\pi}\; b^\top \pi \qquad \text{s.t.}\qquad A^\top \pi \le c.
% $$

% In column generation, the MP is solved by iteratively expanding a restricted master problem (RMP). 
% At iteration $t$, 
% % let $Q_t \subseteq Q$ denote the set of columns currently contained in the RMP, 
% the RMP includes only a subset of columns from $A$, denoted by the index set $R$. Solving the RMP produces a primal solution $x^t$ with objective value $c_R^\top x^t$ and an associated optimal dual solution $\pi_{\mathrm{out}}^t$.

% Given a dual vector $\pi$, the reduced cost of a candidate column $j$ is
% $$
% r_j(\pi) := c_j - A_j^\top \pi,
% $$
% where $A_j$ is the $j$th column of $A$ and $c_j$ is its cost coefficient. The pricing oracle then solves
% $$
% \sigma(\pi) := \min_j r_j(\pi),
% $$
% and returns a column $j$ attaining this minimum. If $\sigma(\pi) < 0$, the corresponding column has negative reduced cost and can improve the current RMP when added. If $\sigma(\pi) \ge 0$, then $r_j(\pi)\ge 0$ for all $j$, hence $\pi$ is dual-feasible for the full MP; consequently, the current RMP solution is optimal for the MP.

Consider the \emph{master problem} (MP): $\min\{c^\top x \mid x \in P := \{x \in \mathbb{R}^n_+ : Ax=b\}\}$, with dual $\max\{b^\top \pi \mid \pi \in D := \{\pi \in \mathbb{R}^m : A^\top \pi \le c\}\}$. 
In column generation, the MP is solved by iteratively expanding a restricted master problem (RMP). 
At iteration $t$, let $R \subseteq \{1,\dots,n\}$ index the selected columns of $A$. 
Solving the RMP yields a primal solution $x_t$ with objective value $c_R^\top x_t$ and an associated optimal dual solution $\pi_t^{\mathrm{out}}$.

Given a dual vector $\pi$, the reduced cost of a column $j$ is $r_j(\pi) := c_j - A_j^\top \pi$, where $A_j$ denotes the $j$th column of $A$ and $c_j$ its cost coefficient. 
The pricing oracle solves $\sigma(\pi) := \min_j r_j(\pi)$ and returns a column $j$ attaining this minimum. 
If $\sigma(\pi) < 0$, the corresponding column has negative reduced cost. 
If $\sigma(\pi) \ge 0$, then $r_j(\pi) \ge 0$ for all $j$, implying that $\pi$ is dual-feasible for the full MP; therefore, the current RMP solution is optimal for the MP.

\subsection{Column Generation and Stabilization} \label{subsec:prelim}

In what follows, we focus on two stabilization approaches: smoothing and penalization.

\smallskip
\noindent\textbf{Smoothing.}
At iteration $t$, let $\pi^{\mathrm{out}}_t$ be the dual solution of the current RMP, and $\pi^{\mathrm{in}}_t$ be a feasible dual \emph{in-point}.
Given $\alpha_t \in [0,1)$, we define the stabilized dual vector used in the pricing step as
\[
\pi^{\mathrm{sep}}_t
= \alpha_t \pi^{\mathrm{in}}_t + (1-\alpha_t)\pi^{\mathrm{out}}_t.
\]
Here, $\alpha_t$ controls the degree of smoothing.

\smallskip
\noindent\textbf{Penalization.}
Penalization stabilizes CG by modifying the RMP objective with a penalty that discourages abrupt changes in the associated dual multipliers.
A general stabilized RMP can be written as
\[
\min_{x \ge 0} \; c_R^\top x + P(x),
\]
where $P(x)$ is a penalty term.
We use the \emph{polyhedral} penalization of \citep{du1999stabilized} as a representative example.
The idea is to relax the RMP equalities with bounded nonnegative variables $y_-$ and $y_+$ and penalize them in the objective.
Specifically, given bounds $\varepsilon_-$ and $\varepsilon_+$ and penalty vectors $\delta_-$ and $\delta_+$, we solve the stabilized RMP
\begin{equation}\label{eq:penalization}
\begin{aligned}
\min_{x,\,y_-,y_+}\quad & c_R^\top x - \delta_{-}^\top y_- + \delta_{+}^\top y_+ \\
\text{s.t.}\quad & A_R x - y_- + y_+ = b,\quad 0 \le y_- \le \varepsilon_-,\quad 0 \le y_+ \le \varepsilon_+,\quad x \ge 0,
\end{aligned}
\end{equation}
where $A_R$ and $c_R$ are the constraint matrix and cost vector of the current RMP.
The dual of \eqref{eq:penalization} is
\begin{equation}\label{eq:penalization-dual}
\begin{aligned}
\max_{\pi,\,w_-,w_+}\quad & b^\top \pi - \varepsilon_{-}^\top w_- - \varepsilon_{+}^\top w_+ \\
\text{s.t.}\quad & A_R^\top \pi \le c_R,\quad \pi - w_+ \le \delta_+,\quad -\pi - w_- \le -\delta_-,\quad w_-, w_+ \ge 0.
\end{aligned}
\end{equation}

Throughout this paper, we use the \emph{single-point} penalty-box configuration used in \cite{kraul2023machine} and \cite{shen2024adaptive}, with
$
\delta_-=\delta_+=\pi_t^{\mathrm{center}}.
$
Thus, any deviation from $\pi_t^{\mathrm{center}}$ is penalized. Other variants, including interval-based penalty boxes and Euclidean proximal regularization, can be used in our framework as well (see Appendix~\ref{appendix:alternative_penalization}).

\subsection{Stabilization and Pricing Distortion}\label{sec:pricing_distortion}

Standard CG prices at the optimal dual solution $\pi_t^{\mathrm{out}}$ of the current RMP. Since the RMP is only a partial approximation of the full MP, $\pi_t^{\mathrm{out}}$ may be a poor proxy, especially early in the algorithm. Moreover, degeneracy can make the dual optimum nonunique, so the particular $\pi_t^{\mathrm{out}}$ returned by the solver may vary across iterations, leading to fluctuations in both the dual sequence and pricing decisions. To address this, stabilized CG replaces $\pi_t^{\mathrm{out}}$ in pricing with a perturbed dual vector, deliberately modifying the reduced-cost signals to promote more stable progress. We refer to this modification as \emph{pricing distortion}.

% \begin{definition}[Pricing Distortion]
% At iteration $t$, pricing distortion occurs when the pricing step selects a column $q \in Q$ that is \emph{not} a minimizer of the reduced cost evaluated at the current RMP dual optimum $\pi_t^{\mathrm{out}}$.
% \end{definition}
\begin{definition}[Pricing Distortion]
At iteration $t$, pricing distortion occurs if the pricing oracle returns a column $q \in Q$ with $
q \notin \arg\min_{q' \in Q} r_{q'}(\pi_t^{\mathrm{out}}).
$
\end{definition}

% \begin{definition}
% Pricing Distortion refers to the phenomenon where the pricing oracle returns a column $q \in Q$ that does not minimize the reduced cost with respect to the current RMP dual solution $\pi_t^{\mathrm{out}}$. 
% % That is,
% % \[
% % q \notin \mathop{\arg\min}_{k \in Q} r_k(\pi_t^{\text{out}}).
% % \]
% \end{definition}

% There are various sources of pricing distortion. Heuristic pricing or approximate labeling algorithms introduce distortion by not attempting to identify the exact minimum reduced cost. For example, recent learning-based method \citep{chi2022deep} introduces distortion by selecting a column from a candidate pool of negative reduced cost columns generated by $\pi_t^{\mathrm{out}}$. In this work, we focus on distortion induced by \textit{stabilization methods}, which occurs due to the modification to the dual vector itself. 
Pricing distortion may also arise from heuristic or approximate pricing. In this paper, however, we assume exact pricing and focus only on stabilization-induced distortion, where reduced costs are evaluated at a stabilized dual $\pi_t^{\mathrm{sep}} \neq \pi_t^{\mathrm{out}}$.
% There are several potential sources of pricing distortion. For example, heuristic pricing or approximate labeling algorithms may introduce distortion by failing to identify the true minimum reduced cost. In this work, however, we assume the pricing oracle is exact. We focus instead on distortion induced by stabilization methods, where the pricing oracle evaluates reduced costs using a stabilized dual vector $\pi_t^{\mathrm{sep}}$ that differs from the optimal RMP dual solution $\pi_t^{\mathrm{out}}$.

A possible consequence of pricing distortion is \emph{mispricing} \citep{pessoa2018automation}.
In the context of stabilization, this occurs when $\pi_t^{\mathrm{sep}}$ deviates so far from $\pi_t^{\mathrm{out}}$ that pricing fails to return a column that separates $\pi_t^{\mathrm{out}}$, even though one exists.
% A consequence of pricing distortion is the potential occurrence of \textit{mispricing}, a phenomenon discussed by \cite{pessoa2018automation}. Mispricing arises when the stabilized dual deviates so far from the current RMP dual solution that the pricing oracle fails to identify a column that improves the current RMP, even though such a column exists.

% \begin{definition}\label{def:mispricing}
% % Let $\sigma(\pi) := \min_{q \in Q} r_q(\pi)$ denote the optimal value of the pricing problem. 
% At iteration $t$, suppose the algorithm has not yet converged, i.e., $\sigma(\pi_t^{\mathrm{out}}) < 0$. A \emph{mispricing} event occurs at that iteration if the pricing oracle returns a column $q \in Q$ whose reduced cost at $\pi_t^{\mathrm{out}}$ is nonnegative:
% $
% r_q(\pi_t^{\mathrm{out}}) \ge 0.
% $
% \end{definition}
\begin{definition}\label{def:mispricing}
At iteration $t$, assuming $\sigma(\pi_t^{\mathrm{out}}) < 0$, a \emph{mispricing} event occurs if the pricing oracle returns a column $q \in Q$ with $
r_q(\pi_t^{\mathrm{out}}) \ge 0.
$
\end{definition}

Based on how the generated column relates to the current RMP dual solution $\pi_t^{\mathrm{out}}$, we distinguish three pricing regimes. Let $q_t$ denote the column produced at iteration $t$.

% \begin{itemize}
%     \item \textit{Standard pricing:} The oracle returns an exact minimizer, $q_t \in \arg\min_{q \in Q} r_q(\pi_t^{\text{out}})$. If $\sigma(\pi_t^{\mathrm{out}}) < 0$, then the selected column is separating for $\pi_t^{\mathrm{out}}$. 
% % In the nondegenerate case, adding this column yields a strict improvement in the current RMP objective and leads to a new dual solution $\pi_{t+1}^{\mathrm{out}}$.

%     \item \textit{Weak pricing distortion:} The oracle returns a column that is not a minimizer at $\pi_t^{\mathrm{out}}$, but still has negative reduced cost:
%     $r_{q_t}\!\left(\pi_t^{\mathrm{out}}\right) < 0.$ The resulting cut still separates $\pi_t^{\mathrm{out}}$.  
% % In the nondegenerate case, this yields a strict improvement in the current RMP objective, although typically smaller than that of the greedy choice.

%     \item \textit{Strong pricing distortion (mispricing):} The oracle returns a column whose reduced cost at $\pi_t^{\mathrm{out}}$ is nonnegative:    $r_{q_t}\!\left(\pi_t^{\mathrm{out}}\right) \ge 0.$
%     This is precisely the mispricing event in Definition~\ref{def:mispricing}. The cut does not separate $\pi_t^{\mathrm{out}}$, so $\pi_t^{\mathrm{out}}$ remains feasible for the updated dual problem. Consequently, the updated RMP has the same optimal value at this iteration. 
% \end{itemize}

\begin{itemize}
    \item \textit{Standard pricing:} $q_t \in \arg\min_{q \in Q} r_q(\pi_t^{\mathrm{out}})$. If $\sigma(\pi_t^{\mathrm{out}}) < 0$, then $q_t$ separates $\pi_t^{\mathrm{out}}$.
    \item \textit{Weak pricing distortion:} $q_t$ is not a minimizer at $\pi_t^{\mathrm{out}}$, but still satisfies
    $
    r_{q_t}(\pi_t^{\mathrm{out}}) < 0.
    $
    The resulting cut still separates $\pi_t^{\mathrm{out}}$.
    \item \textit{Strong pricing distortion (mispricing):} 
    $
    r_{q_t}(\pi_t^{\mathrm{out}}) \ge 0.
    $
    % This is precisely the mispricing event in Definition~\ref{def:mispricing}. 
The cut does not separate $\pi_t^{\mathrm{out}}$, and the RMP objective does not improve at iteration $t$.
\end{itemize}

\paragraph{Illustrative Example.}
% To illustrate how pricing distortion may affect the convergence path, we use a two-dimensional example and describe column generation from a dual perspective. The RMP can be viewed as enforcing only a subset of the dual-feasibility inequalities, with each pricing step adding one additional inequality. We start from an initial restriction that includes only the box bounds,
% \begin{align*}
%     \max_{\pi}\quad & \pi_1 + 0.1\pi_2 \\
%     \text{s.t.}\quad
%     & 0 \le \pi_1 \le 5,\quad 0 \le \pi_2 \le 5,
% \end{align*}
% and then progressively strengthen it by appending the remaining inequalities. The \emph{full} dual feasible region is obtained once the following three constraints are included:
% \begin{align*}
% \pi_1 + \pi_2 \le 6 \quad (A),
% \qquad
% \pi_1 - 2\pi_2 \le -1 \quad (B),
% \qquad
% \pi_1 - \pi_2 \le 1 \quad (C).
% \end{align*}

% At iteration $t$, solving the current restricted problem yields an optimal dual vector $\pi_t^{\mathrm{out}}$. The pricing step then selects one of the not-yet-added constraints among $(A)$--$(C)$ and adds it to form the next RMP. In this example, the optimal solution of the full problem is the vertex $\pi^\star$ defined by the intersection of constraints $(A)$ and $(C)$.

% We compare three pricing regimes: (a) standard pricing, (b) weak distortion, and (c) strong distortion. The resulting trajectories of $\{\pi_t^{\mathrm{out}}\}$ are shown in Figure~\ref{fig:toy_pricing_steps}.
% Detailed calculations for this toy example are deferred to Appendix~\ref{app:toy_details}.

To illustrate how pricing distortion can alter the convergence path, consider the two-dimensional dual problem
\[
\max_{\pi}\ \pi_1 + 0.1\pi_2
\qquad \text{s.t.} \qquad
0 \le \pi_1 \le 5,\quad 0 \le \pi_2 \le 5,
\]
and suppose the full dual feasible region is obtained after adding
\[
\pi_1 + \pi_2 \le 6 \ (A), \qquad
\pi_1 - 2\pi_2 \le -1 \ (B), \qquad
\pi_1 - \pi_2 \le 1 \ (C).
\]
Thus, the RMP can be viewed as enforcing only a subset of the full dual inequalities, with each pricing step adding one more. 
% At iteration $t$, pricing selects one of the not-yet-added inequalities among $(A)$--$(C)$, and the full optimum is the vertex $\pi^\star = A \cap C$. 
Figure~\ref{fig:toy_pricing_steps} compares the resulting trajectories under different pricing strategies. Detailed calculations are deferred to Appendix~\ref{app:toy_details}.

\begin{figure}[ht!]
\centering
\captionsetup[subfigure]{justification=centering,singlelinecheck=true,skip=2pt,font=footnotesize}

\begingroup
\tikzset{
  ax/.style={-Stealth, line width=0.55pt},
  tick/.style={line width=0.4pt},
  trueD/.style={draw=gray!55, fill=gray!12, line width=0.7pt},
  DR/.style={draw=black, line width=1.0pt},
  prevDR/.style={draw=gray!60, dashed, line width=0.75pt},
  newcut/.style={draw=red, dashed, line width=0.85pt},
  piout/.style={circle, fill=black, inner sep=1pt},
  pisep/.style={regular polygon, regular polygon sides=3, draw=black, fill=white, inner sep=0.6pt},
  lab/.style={font=\scriptsize, inner sep=1pt, text=red},
  undertitle/.style={font=\fontsize{5.8}{6.6}\selectfont, anchor=north},
  pilab/.style={font=\scriptsize, inner sep=1pt}
}

\newcommand{\PanelSetup}{%
  \draw[ax] (0,0) -- (5.55,0);
  \draw[ax] (0,0) -- (0,5.55);
  \draw[tick] (5,0.07) -- (5,-0.07);
  \draw[tick] (0.07,5) -- (-0.07,5);
  \node[font=\scriptsize, anchor=north] at (5,-0.16) {$5$};
  \node[font=\scriptsize, anchor=east]  at (-0.16,5) {$5$};
  \node[font=\scriptsize, anchor=north] at (0,-0.16) {$0$};
  \node[font=\scriptsize, anchor=west]  at (5.30,0) {$\pi_1$};
  \node[font=\scriptsize, anchor=south] at (0,5.30) {$\pi_2$};
}

\newcommand{\DrawTrueD}{%
  \path[trueD] (0,0.5) -- (3,2) -- (3.5,2.5) -- (1,5) -- (0,5) -- cycle;
}

\newcommand{\CutA}{%
  \draw[newcut] (1,5) -- (5,1);
  \node[lab] at (2.80,2.75) {$A$};
}
\newcommand{\CutB}{%
  \draw[newcut] (0,0.5) -- (5,3);
  \node[lab] at (2.25,2.00) {$B$};
}
\newcommand{\CutC}{%
  \draw[newcut] (1,0) -- (5,4);
  \node[lab] at (2.75,2.20) {$C$};
}

% =========================
% Row 1: (a) Standard pricing
% =========================
\begin{subfigure}[t]{\textwidth}
\centering
\begin{tikzpicture}[x=0.48cm,y=0.48cm,scale=0.88,transform shape]
\def\xsep{6.6}

\begin{scope}[shift={(0,0)}]
  \PanelSetup
  \DrawTrueD
  \draw[DR] (0,0) rectangle (5,5);
  \node[piout] at (5,5) {};
  \node[pilab, anchor=north east] at (5,5) {$\pi^{out}_0$};
  \node[undertitle,align=center] at (2.6,-1.12) {$t=0$: $\mathcal D_R^0$\\ ($z_R^0=5.50$)};
\end{scope}

\begin{scope}[shift={(\xsep,0)}]
  \PanelSetup
  \DrawTrueD
  \draw[prevDR] (0,0) rectangle (5,5);
  \draw[DR] (0,0) -- (5,0) -- (5,1) -- (1,5) -- (0,5) -- cycle;
  \CutA
  \node[piout] at (5,1) {};
  \node[pilab, anchor=north east] at (5,1) {$\pi^{out}_1$};
  \node[undertitle,align=center] at (2.6,-0.98) {$t=1$: add cut $A$\\ ($z_R^1=5.10$)};
\end{scope}

\begin{scope}[shift={(2*\xsep,0)}]
  \PanelSetup
  \DrawTrueD
  \draw[prevDR] (0,0) -- (5,0) -- (5,1) -- (1,5) -- (0,5) -- cycle;
  \draw[DR] (0,0.5) -- (11/3,7/3) -- (1,5) -- (0,5) -- cycle;
  \CutB
  \node[piout] at (11/3,7/3) {};
  \node[pilab, anchor=north west] at (11/3,7/3) {$\pi^{out}_2$};
  \node[undertitle,align=center] at (2.6,-0.98) {$t=2$: add cut $B$\\ ($z_R^2=3.90$)};
\end{scope}

\begin{scope}[shift={(3*\xsep,0)}]
  \PanelSetup
  \DrawTrueD
  \draw[prevDR] (0,0.5) -- (11/3,7/3) -- (1,5) -- (0,5) -- cycle;
  \draw[DR] (0,0.5) -- (3,2) -- (3.5,2.5) -- (1,5) -- (0,5) -- cycle;
  \CutC
  \node[star, star points=5, star point ratio=1.8, fill=black, draw=black, inner sep=0.7pt] at (3.5,2.5) {};
  \node[pilab, anchor=south west] at (3.5,2.5) {$\pi^\star$};
  \node[undertitle,align=center] at (2.6,-0.98) {$t=3$: add cut $C$\\ converge\\ ($z^\star=3.75$)};
\end{scope}

\end{tikzpicture}
\caption{Standard pricing}
\end{subfigure}

\vspace{0.2em}

% =========================
% Row 2: (b) and (c) side by side
% =========================
\begin{subfigure}[t]{0.49\textwidth}
\centering
\resizebox{\linewidth}{!}{%
\begin{tikzpicture}[x=0.44cm,y=0.44cm,scale=0.82,transform shape]
\def\xsep{6.25}

\begin{scope}[shift={(0,0)}]
  \PanelSetup
  \DrawTrueD
  \draw[DR] (0,0) rectangle (5,5);
  \node[piout] at (5,5) {};
  \node[pilab, anchor=north east] at (5,5) {$\pi^{out}_0$};
  \node[undertitle,align=center] at (2.6,-0.96) {$t=0$: $\mathcal D_R^0$\\ ($z_R^0=5.50$)};
\end{scope}

\begin{scope}[shift={(\xsep,0)}]
  \PanelSetup
  \DrawTrueD
  \draw[prevDR] (0,0) rectangle (5,5);
  \draw[DR] (0,0) -- (5,0) -- (5,1) -- (1,5) -- (0,5) -- cycle;
  \CutA
  \node[piout] at (5,1) {};
  \node[pilab, anchor=north east] at (5,1) {$\pi^{out}_1$};
  \node[undertitle,align=center] at (2.6,-0.96) {$t=1$: add cut $A$\\ ($z_R^1=5.10$)};
\end{scope}

\begin{scope}[shift={(2*\xsep,0)}]
  \PanelSetup
  \DrawTrueD
  \draw[prevDR] (0,0) -- (5,0) -- (5,1) -- (1,5) -- (0,5) -- cycle;
  \draw[DR] (0,0) -- (1,0) -- (3.5,2.5) -- (1,5) -- (0,5) -- cycle;
  \CutC
  \node[star, star points=5, star point ratio=1.8, fill=black, draw=black, inner sep=0.7pt] at (3.5,2.5) {};
  \node[pilab, anchor=south west] at (3.5,2.5) {$\pi^\star$};
  \node[undertitle,align=center] at (2.6,-0.96) {$t=2$: add cut $C$\\ converge\\ ($z^\star=3.75$)};
\end{scope}

\end{tikzpicture}%
}
\caption{Weak pricing distortion}
\end{subfigure}\hfill
\begin{subfigure}[t]{0.49\textwidth}
\centering
\resizebox{\linewidth}{!}{%
\begin{tikzpicture}[x=0.44cm,y=0.44cm,scale=0.82,transform shape]
\def\xsep{6.25}

\begin{scope}[shift={(0,0)}]
  \PanelSetup
  \DrawTrueD
  \draw[DR] (0,0) rectangle (5,5);
  \node[piout] at (5,5) {};
  \node[pilab, anchor=north east] at (5,5) {$\pi^{out}_0$};
  \node[undertitle,align=center] at (2.6,-0.96) {$t=0$: $\mathcal D_R^0$\\ ($z_R^0=5.50$)};
\end{scope}

\begin{scope}[shift={(\xsep,0)}]
  \PanelSetup
  \DrawTrueD
  \draw[prevDR] (0,0) rectangle (5,5);
  \draw[DR] (0,0) -- (1,0) -- (5,4) -- (5,5) -- (0,5) -- cycle;
  \CutC
  \node[piout] at (5,5) {};
  \node[pilab, anchor=north east] at (5,5) {$\pi^{out}_1$};
  \node[undertitle,align=center] at (2.6,-0.96) {$t=1$: add cut $C$\\ mispricing\\ ($z_R^1=5.50$)};
\end{scope}

\begin{scope}[shift={(2*\xsep,0)}]
  \PanelSetup
  \DrawTrueD
  \draw[prevDR] (0,0) -- (1,0) -- (5,4) -- (5,5) -- (0,5) -- cycle;
  \draw[DR] (0,0) -- (1,0) -- (3.5,2.5) -- (1,5) -- (0,5) -- cycle;
  \CutA
  \node[star, star points=5, star point ratio=1.8, fill=black, draw=black, inner sep=0.7pt] at (3.5,2.5) {};
  \node[pilab, anchor=south west] at (3.5,2.5) {$\pi^\star$};
  \node[undertitle,align=center] at (2.6,-0.96) {$t=2$: add cut $A$\\ converge\\ ($z^\star=3.75$)};
\end{scope}

\end{tikzpicture}%
}
\caption{Strong pricing distortion}
\end{subfigure}

\caption{Illustration of the CG iterative process in the dual space under standard pricing, weak distortion, and strong distortion. Thick polygons show the updated restricted-dual feasible set $\mathcal D_R^t$, gray dashed outlines show the previous set $\mathcal D_R^{t-1}$, the shaded region is the full dual feasible set $\mathcal D$, and the red dashed segment is the cut added at the current step. Filled dots denote $\pi_t^{\mathrm{out}}$ and the star denotes $\pi^\star$.}
\label{fig:toy_pricing_steps}
\endgroup
\end{figure}
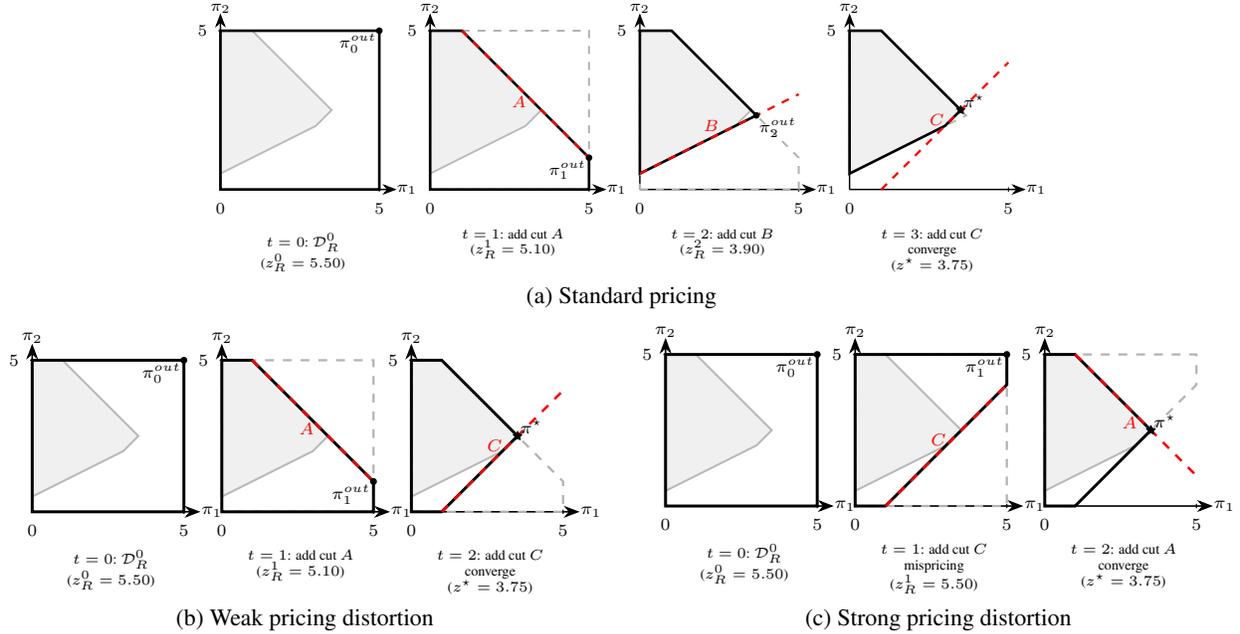

% \edit{
This example shows that distortion can change which inequalities become active later: a step with no immediate improvement may still redirect the search and avoid a detour induced by standard pricing.
The same mechanism, however, can be detrimental. When $\pi_t^{\mathrm{out}}$ is already close to $\pi^\star$, it provides a reliable pricing signal, and excessive distortion may induce repeated mispricing, delaying the columns needed to certify optimality. Existing methods therefore regulate how far pricing departs from $\pi_t^{\mathrm{out}}$: adaptive smoothing \citep{pessoa2018automation} shrinks the smoothing parameter after mispricing, and RLCG \citep{chi2022deep} restricts selection to columns with negative reduced cost under $\pi_t^{\mathrm{out}}$. Such safeguards reduce mispricing but may sacrifice useful exploration.
% }

% This example shows that distortion can change which inequalities become active later: 
% even a step with no immediate objective improvement may still redirect the search and help avoid a detour induced by standard pricing.
% However, this same mechanism can be detrimental. When the unperturbed dual $\pi_t^{\mathrm{out}}$ is already close to the true dual optimum $\pi^\star$, it provides a reliable pricing signal, and excessive distortion may induce repeated mispricing and delay the addition of the columns needed to certify optimality. Existing methods therefore regulate how far pricing departs from $\pi_t^{\mathrm{out}}$. For example, the adaptive smoothing method of \cite{pessoa2018automation} reduces the smoothing parameter after mispricing is detected, moving $\pi_t^{\mathrm{sep}}$ closer to $\pi_t^{\mathrm{out}}$ to recover separating cuts. Likewise, methods such as RLCG \citep{chi2022deep} restrict the selected column to one with negative reduced cost under $\pi_t^{\mathrm{out}}$. These safeguards reduce immediate mispricing, but may also sacrifice useful exploration.

\section{Stabilization Framework}\label{sec:framework}

In this section, we introduce a unified stabilization framework that covers both smoothing and penalization. 
The framework clarifies how stabilization parameters enter the pricing step and affect the generated columns, provides conditions under which stabilization preserves convergence, and discusses the use of predicted dual vectors as reference points.
% Using this framework, we describe how the stabilization parameters enter the pricing step and, in turn, influence the generated columns. We then provide conditions under which stabilization preserves convergence to an optimal solution. Finally, we discuss the implications of using a predicted dual vector to guide stabilization, rather than relying solely on dual information obtained from the RMP, typically using $\pi_t^{\mathrm{out}}$ or historical in-points.

\subsection{A General Stabilization Framework}\label{subsec:stabilization framework}

Although smoothing and penalization are typically presented separately, they share a common underlying structure: the dual vector used in pricing is constructed from a \textit{reference dual point} and \textit{stabilization parameters}, together with information from the current RMP. In this section, we make this structure explicit through a unified framework capturing both classes of methods. Formally, the framework consists of the following components:

%We unify smoothing and penalization by viewing both as constructing the pricing dual from a reference point, stabilization parameters, and the current RMP.
 
\begin{itemize}

    \item \textit{Reference dual point} $\pi_t^{\mathrm{ref}}$:  
    A designated point in the dual space used by the stabilization scheme when constructing the dual vector used in pricing.
    
    \item \textit{Stabilization parameters} $\theta_t$:  
    A set of method-specific parameters that control the degree of stabilization relative to the reference point $\pi_t^{\mathrm{ref}}$.
% , where $\Theta$ denotes the parameter space.
    
    \item \textit{Stabilization operator} $f_t$:  
    A mapping that combines $\pi_t^{\mathrm{ref}}$, $\theta_t$, and the current $RMP_t$ to produce the stabilized dual vector used in pricing: $  \pi_t^{\mathrm{sep}} = f_t(\pi_t^{\mathrm{ref}}, \theta_t; RMP_t).$
    % \[
    % \pi_t^{\mathrm{sep}} = f_t(\pi_t^{\mathrm{ref}}, \theta_t; RMP_t).
    % \]

\end{itemize}

We next instantiate this framework for smoothing and penalization. 
Figure~\ref{fig:stabilization_framework} illustrates the resulting 
information flow for each method.

\begin{figure}[ht!]
\centering
\captionsetup[subfigure]{justification=centering,font=small}

\begingroup
\tikzset{
    >=latex, % <-- Changed to standard lowercase latex arrow
    font=\normalsize,
    box/.style={
        draw,
        rectangle,
        align=center,
        inner sep=5pt,
        minimum height=1.00cm,
        line width=0.75pt
    },
    flow/.style={->, line width=0.72pt},
    plain/.style={line width=0.72pt}
}

% =========================
% (a) Standard CG
% =========================
\begin{subfigure}[t]{0.29\textwidth}
\centering
\begin{tikzpicture}[x=1cm,y=1cm,scale=0.62,transform shape]
\path[use as bounding box] (0,-0.75) rectangle (7.35,3.35);

\node[box, minimum width=3.25cm, minimum height=1.00cm] (pricing) at (3.55,0.00)
{\textbf{Pricing Problem}};

\node[box, minimum width=1.55cm, minimum height=1.00cm] (rmp) at (3.55,2.30)
{$RMP_t$};

% Left route: pricing -> left -> up -> right -> RMP
\coordinate (L1) at (0.80,0.00);
\coordinate (L2) at (0.80,2.30);
\coordinate (L3) at (2.28,2.30);

\draw[plain] (pricing.west) -- (L1) -- (L2) -- (L3);
\draw[flow]  (L3) -- (rmp.west);

\node[align=right, anchor=east] at (0.60,1.15) {New\\column};

% Right route: RMP -> dual -> down -> left -> pricing
\coordinate (R1) at (6.1,2.30);
\coordinate (R2) at (6.1,0.00);

\draw[plain] (rmp.east) -- (R1) -- (R2);
\draw[flow]  (R2) -- (pricing.east);

\node[above] at (4.95,2.44) {$\pi_t^{\mathrm{out}}$};

\end{tikzpicture}
\caption{Standard CG}
\end{subfigure}
\hfill
% =========================
% (b) Smoothing
% =========================
\begin{subfigure}[t]{0.38\textwidth}
\centering
\begin{tikzpicture}[x=1cm,y=1cm,scale=0.62,transform shape]
\path[use as bounding box] (0,-0.85) rectangle (9.55,3.45);

\node[box, minimum width=3.25cm, minimum height=1.00cm] (pricing) at (3.95,0.00)
{\textbf{Pricing Problem}};

\node[box, minimum width=1.55cm, minimum height=1.00cm] (rmp) at (1.70,2.30)
{$RMP_t$};

% smoothing box moved slightly right to lengthen RMP -> smooth segment
\node[box, text width=4.55cm, minimum height=1.18cm] (smooth) at (5.92,2.30)
{$\pi_t^{\mathrm{sep}}=$\\
$\theta_t \pi_t^{\mathrm{ref}} + (1-\theta_t)\pi_t^{\mathrm{out}}$};

% Left route: pricing -> longer left extension -> up -> longer right extension -> arrow into RMP
\coordinate (L1) at (0.10,0.00);
\coordinate (L2) at (0.10,2.30);
\coordinate (L3) at (0.45,2.30);

\draw[plain] (pricing.west) -- (L1) -- (L2) -- (L3);
\draw[flow]  (L3) -- (rmp.west);

\node[align=right, anchor=east] at (-0.05,1.15) {New\\column};

% Internal smoothing flow (lengthened)
\draw[flow] (rmp.east) -- (smooth.west);

\node[above] at (2.98,2.44) {$\pi_t^{\mathrm{out}}$};

% Output route: right -> down -> left -> arrow into pricing
\coordinate (R1) at (8.78,2.30);
\coordinate (R2) at (8.78,0.00);
\coordinate (R3) at (6.10,0.00);

\draw[plain] (smooth.east) -- (R1) -- (R2) -- (R3);
\draw[flow]  (R3) -- (pricing.east);

\node[above] at (8.88,2.44) {$\pi_t^{\mathrm{sep}}$};

\end{tikzpicture}
\caption{Smoothing}
\end{subfigure}
\hfill
% =========================
% (c) Penalization
% =========================
\begin{subfigure}[t]{0.29\textwidth}
\centering
\begin{tikzpicture}[x=1cm,y=1cm,scale=0.62,transform shape]
\path[use as bounding box] (0,-0.75) rectangle (7.55,3.35);

\node[box, minimum width=3.25cm, minimum height=1.00cm] (pricing) at (3.45,0.00)
{\textbf{Pricing Problem}};

\node[box, text width=3.30cm, minimum height=1.15cm] (pen) at (3.32,2.30)
{Stabilized $RMP_t$\\
$(\pi_t^{\mathrm{ref}},\theta_t)$};

% Left route: pricing -> left -> up -> arrow into stabilized RMP from outside left
\coordinate (L1) at (0.78,0.00);
\coordinate (L2) at (0.78,2.30);
\coordinate (L3) at (1.02,2.30);

\draw[plain] (pricing.west) -- (L1) -- (L2) -- (L3);
\draw[flow]  (L3) -- (pen.west);

\node[align=right, anchor=east] at (0.58,1.15) {New\\column};

% Output route with room for pi_sep
\coordinate (R1) at (6.38,2.30);
\coordinate (R2) at (6.38,0.00);
\coordinate (R3) at (5.72,0.00);

\draw[plain] (pen.east) -- (R1) -- (R2) -- (R3);
\draw[flow]  (R3) -- (pricing.east);

\node[above] at (5.78,2.46) {$\pi_t^{\mathrm{sep}}$};

\end{tikzpicture}
\caption{Penalization}
\end{subfigure}

% \caption{Comparison of the dual information flow per iteration under standard column generation and two stabilization methods. (a) Standard CG passes the optimal dual solution $\pi_t^{\mathrm{out}}$ from the RMP directly to the pricing problem. (b) Smoothing explicitly computes a convex combination of the RMP's optimal dual $\pi_t^{\mathrm{out}}$ and a reference dual point $\pi_t^{\mathrm{ref}}$ to yield the stabilized dual $\pi_t^{\mathrm{sep}}$. (c) Penalization implicitly generates the stabilized dual $\pi_t^{\mathrm{sep}}$ by directly solving a modified RMP augmented with a penalty function centered at $\pi_t^{\mathrm{ref}}$.}
% \label{fig:stabilization_framework}
% \caption{Dual information flow under standard column generation and two stabilization methods. (a) Standard CG sends $\pi_t^{\mathrm{out}}$ from the RMP directly to pricing. (b) Smoothing forms $\pi_t^{\mathrm{sep}}$ from $\pi_t^{\mathrm{out}}$ and $\pi_t^{\mathrm{ref}}$. (c) Penalization obtains $\pi_t^{\mathrm{sep}}$ by solving a stabilized $RMP_t(\pi_t^{\mathrm{ref}},\theta_t)$.}
\caption{Dual information flow under standard column generation and two stabilization methods.}
\label{fig:stabilization_framework}

\endgroup
\end{figure}
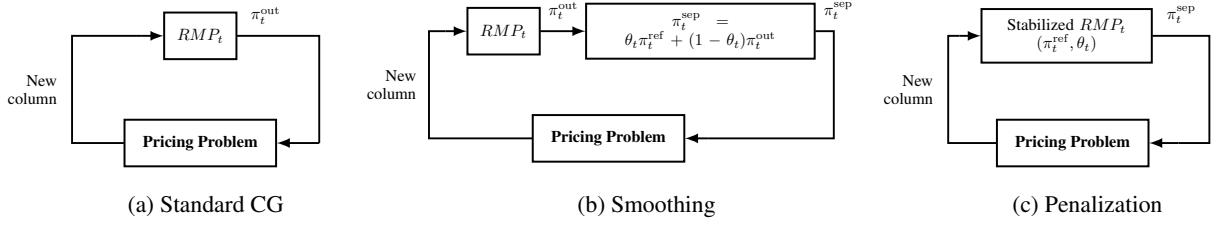

\smallskip
\noindent\textbf{Smoothing.}
Choose the reference dual point $\pi_t^{\mathrm{ref}}=\pi_t^{\mathrm{in}}$ %(e.g. historical or predicted in-points), 
and let $\theta_t := \{\alpha_t\}$ denote the smoothing weight 
with $\alpha_t\in[0,1)$. The stabilization operator is defined 
\textit{explicitly} and returns the stabilized dual as
% \vspace{-0.4em}
\begin{equation}\label{eq:f_smooth}
\pi_t^{\mathrm{sep}}
= f_t^{\mathrm{smooth}}(\pi_t^{\mathrm{ref}},\theta_t; RMP_t)
:= \alpha_t \pi_t^{\mathrm{ref}} + (1-\alpha_t)\pi_t^{\mathrm{out}}.
\end{equation}
% where $\pi_t^{\mathrm{out}}$ denotes an optimal dual solution of the current RMP.
\smallskip
\noindent\textbf{Penalization.} Assume a single-point penalty box, and
% Let $\delta_t^-=\delta_t^+=\pi_t^{\mathrm{ref}}$, so that the 
% penalty box is centered at $\pi_t^{\mathrm{ref}}$, and define 
% the stabilization parameters as 
% $\theta_t:=\{\varepsilon_t^-,\varepsilon_t^+\}$ with 
% $\varepsilon_t^-,\varepsilon_t^+\ge 0$. The penalized dual vector 
% $\pi_t^{\mathrm{pen}}$ is obtained as the $\pi$-component of an 
% optimal solution to
let $\delta_t^-=\delta_t^+=\pi_t^{\mathrm{ref}}$ and $\theta_t:=\{\varepsilon_t^-,\varepsilon_t^+\}$ with $\varepsilon_t^-,\varepsilon_t^+\ge 0$. Let $\pi_t^{\mathrm{pen}}$ denote the $\pi$-component of an optimal solution to

\begin{equation}\label{eq:f_pen1}
\begin{aligned}
\max_{\pi\in\mathbb{R}^m,\, w^-\in\mathbb{R}^m_+,\, w^+\in\mathbb{R}^m_+}\quad
& b^\top\pi - (\varepsilon_t^-)^\top w^- - (\varepsilon_t^+)^\top w^+ \\
\text{s.t.}\quad
& A_R^\top\pi \le c_R,\;
  \pi - w^+ \le \pi_t^{\mathrm{ref}},\;
  -\pi - w^- \le -\pi_t^{\mathrm{ref}}.
\end{aligned}
\end{equation}
The stabilized dual is then: $\pi_t^{\mathrm{sep}} 
= f_t^{\mathrm{pen}}(\pi_t^{\mathrm{ref}},\theta_t; RMP_t) 
:= \pi_t^{\mathrm{pen}}.$

While straightforward, for completeness we establish that both methods reduce to standard column generation under specific conditions.

% \begin{lemma}
% \label{lem:neutral}
% For both stabilization methods, there exist specific conditions on parameters or reference points that recover the standard column generation trajectory:
% \begin{enumerate}
%     \item For smoothing: when $\alpha_t = 0$, we have $f_t^{\mathrm{smooth}}(\pi_t^{\mathrm{ref}}, 0;RMP_t) = \pi_t^{\mathrm{out}}$; or when $\pi_t^{\mathrm{ref}} = \pi_t^{\mathrm{out}}$, we have $f_t^{\mathrm{smooth}}(\pi_t^{\mathrm{out}},\alpha_t;RMP_t) = \pi_t^{\mathrm{out}}$ for any $\alpha_t \in [0,1)$.
%     \item For penalization in ~\eqref{eq:f_pen1}, if either $\varepsilon_t^-=0$ and $\varepsilon_t^+=0$,
%     or $\pi_t^{\mathrm{ref}}$ is an optimal dual solution of the RMP,
%     then the stabilized dual returned by $f_t^{\mathrm{pen}}(\pi_t^{\mathrm{ref}},\theta_t;RMP_t)$ is an optimal dual solution of the RMP.
% \end{enumerate}
% \end{lemma}

\begin{lemma}
\label{lem:neutral}
For both stabilization methods, standard column generation is recovered when the stabilization strength is zero. Beyond this common case:
\begin{enumerate}
    \item For smoothing, if $\pi_t^{\mathrm{ref}}=\pi_t^{\mathrm{out}}$, then
    $
    f_t^{\mathrm{smooth}}(\pi_t^{\mathrm{out}},\alpha_t;RMP_t)=\pi_t^{\mathrm{out}}
    $
    for any $\alpha_t\in[0,1)$.
    \item For penalization in \eqref{eq:f_pen1}, if $\pi_t^{\mathrm{ref}}$ is an optimal dual solution of the RMP, then
    $
    f_t^{\mathrm{pen}}(\pi_t^{\mathrm{ref}},\theta_t;RMP_t)
    $
    returns an optimal dual solution of the RMP.
\end{enumerate}
\end{lemma}

% First, setting the stabilization \emph{strength} to a neutral value (e.g., $\alpha_t=0$ for smoothing, or $\varepsilon_t^-=\varepsilon_t^+=0$ for penalization) makes the stabilized dual coincide with the RMP dual optimum. Second, if the reference dual point is already an RMP dual optimum, then applying stabilization has no effect: the operator returns an RMP-optimal dual regardless of the remaining parameters.
% Under restrictive conditions, the two methods can even produce the same stabilized dual, though this clearly does not hold in general; we show in Appendix~\ref{sec:eq-smooth-prox} that the Euclidean-proximal penalization operator \eqref{eq:f_pen3} can be parameterized to coincide with the smoothing operator 
% \eqref{eq:f_smooth}.

Under restrictive conditions, the two methods can produce the same stabilized dual, but this equivalence clearly does not hold in general; see Appendix~\ref{sec:eq-smooth-prox}, where we show that the Euclidean-proximal penalization operator \eqref{eq:f_pen3} can be parameterized to coincide with the smoothing operator \eqref{eq:f_smooth}.

\subsection{Stabilization Dynamics and Convergence}\label{subsec:stabilization parameter}

We now examine the role of the stabilization parameters $\theta_t$ 
in controlling the deviation between $\pi_t^{\mathrm{sep}}$ and $\pi_t^{\mathrm{out}}$, and the implications of this deviation 
for pricing behavior and convergence. To isolate the effect of $\theta_t$, the analysis here treats the reference point 
$\pi_t^{\mathrm{ref}}$ as given; reference point selection is 
discussed in Section~\ref{sec:reference point}.

\subsubsection{Stabilization Dynamics}

% We begin by studying how the stabilization parameters influence the deviation between the \textit{stabilized} dual solution $\pi_t^{\mathrm{sep}}$ and the \textit{current} RMP dual solution $\pi_t^{\mathrm{out}}$, and how this deviation affects mispricing.

% To formalize the connection between dual deviations and pricing distortion,
% % let $Q_{NB} \subset Q$ denote the set of all \emph{non-basic} columns at iteration $t$. 
% we partition this set according to the reduced costs with respect to the current RMP dual solution $\pi_t^{\mathrm{out}}$:
% % \[
% % \mathcal{Q}_-(\pi_t^{\mathrm{out}}):=\{q \in Q_{NB}:\ r_q(\pi_t^{\mathrm{out}})<0\},
% % \qquad
% % \mathcal{Q}_+(\pi_t^{\mathrm{out}}):=\{q \in Q_{NB}:\ r_q(\pi_t^{\mathrm{out}})\ge 0\}.
% % \]
% \[
% \mathcal{Q}_-(\pi_t^{\mathrm{out}}):=\{q \in Q:\ r_q(\pi_t^{\mathrm{out}})<0\},
% \qquad
% \mathcal{Q}_+(\pi_t^{\mathrm{out}}):=\{q \in Q:\ r_q(\pi_t^{\mathrm{out}})\ge 0\}.
% \]
% % If the current RMP solution is not optimal for the master problem, then $\mathcal{Q}_-(\pi_t^{\mathrm{out}})\neq\varnothing$. 
% We use two quantities to characterize the reduced-cost separation around zero: the most negative reduced cost, $\gamma_t^-$, and the smallest nonnegative reduced cost, $\gamma_t^+$,
% \[
% \gamma_t^- := -\min_{q\in\mathcal Q_-(\pi_t^{\mathrm{out}})} r_q(\pi_t^{\mathrm{out}}),
% \qquad 
% \gamma_t^+ := \min_{q\in\mathcal{Q}_+(\pi_t^{\mathrm{out}})} r_q(\pi_t^{\mathrm{out}}).
% \]
% We define $\gamma_t^+=+\infty$ if $\mathcal{Q}_+(\pi_t^{\mathrm{out}})=\varnothing$.

To formalize the connection between dual deviations and pricing distortion, we partition the column set according to the reduced costs with respect to the current RMP dual solution \(\pi_t^{\mathrm{out}}\): 
\(\mathcal{Q}_-(\pi_t^{\mathrm{out}}):=\{q \in Q:\ r_q(\pi_t^{\mathrm{out}})<0\}\), and \(\mathcal{Q}_+(\pi_t^{\mathrm{out}}):=\{q \in Q:\ r_q(\pi_t^{\mathrm{out}})\ge 0\}\). Define \(\gamma_t^- := -\min_{q\in\mathcal Q_-(\pi_t^{\mathrm{out}})} r_q(\pi_t^{\mathrm{out}})\), and \(\gamma_t^+ := \min_{q\in\mathcal{Q}_+(\pi_t^{\mathrm{out}})} r_q(\pi_t^{\mathrm{out}})\), with $\gamma_t^+=+\infty$ if $\mathcal{Q}_+(\pi_t^{\mathrm{out}})=\varnothing$.
We consider iterations with $\mathcal Q_-(\pi_t^{\mathrm{out}})\neq\varnothing$ and assume that, for every dual vector $\pi$, the pricing problem admits an optimal solution.

% Under these conditions, the following proposition establishes that if the stabilized dual $\pi_t^{\mathrm{sep}}$ remains sufficiently close to $\pi_t^{\mathrm{out}}$, the pricing oracle is guaranteed to return a column that separates $\pi_t^{\mathrm{out}}$.

\begin{proposition}[Local consistency of stabilized pricing]
\label{prop:local-consistency-nos-stall}
Assume the columns are uniformly bounded, i.e., \(\bar a := \sup_{q\in Q}\|A_q\|_2 < \infty\).
% \[
% \bar a := \sup_{q\in Q}\|A_q\|_2 < \infty.
% \]
If the stabilized dual $\pi_t^{\mathrm{sep}}$ satisfies
\begin{equation}\label{eq:dist_bound}
\big\|\pi_t^{\mathrm{sep}}-\pi_t^{\mathrm{out}}\big\|_2\ \le\ 
\frac{\min\{\gamma_t^-,\gamma_t^+\}}{2\,\bar a},
\end{equation}
then 
% the column returned by the pricing oracle satisfies
$r_{q(\pi_t^{\mathrm{sep}})}(\pi_t^{\mathrm{out}})<0$,
and mispricing does not occur.
\end{proposition}

%Proposition~\ref{prop:local-consistency-nos-stall} shows that the room for dual perturbation is characterized by the separation margin $\min\{\gamma_t^-,\gamma_t^+\}$, which measures how far the reduced costs are from zero on both sides --- the margin by which the best column improves the RMP, and the margin by which the nearest non-improving column fails to. Although the condition is only sufficient, it offers useful intuition: under degeneracy, both sides shrink, and controlling the distortion induced by stabilization becomes substantially harder. We next map this condition into explicit bounds for each stabilization method.

% Proposition~\ref{prop:local-consistency-nos-stall} shows that if $\pi_t^{\mathrm{sep}}$ remains sufficiently close to $\pi_t^{\mathrm{out}}$, pricing still returns a column that separates $\pi_t^{\mathrm{out}}$. identifies the 
% separation margin $\min\{\gamma_t^-,\gamma_t^+\}$ as the key 
% quantity governing the safe perturbation radius.
Proposition~\ref{prop:local-consistency-nos-stall} shows that if $\pi_t^{\mathrm{sep}}$ remains sufficiently close to $\pi_t^{\mathrm{out}}$, pricing still returns a column that separates $\pi_t^{\mathrm{out}}$. The governing quantity is the separation margin $\min{\gamma_t^-,\gamma_t^+}$: within this radius, pricing may differ from standard pricing but remains in the weak-distortion regime; beyond it, the selected column may enter the strong-distortion regime.

Applying Proposition~\ref{prop:local-consistency-nos-stall} to the smoothing operator \eqref{eq:f_smooth} yields the following corollary.

\begin{corollary}[Safe smoothing bound]
\label{prop:safe-bound-smoothing}
In the smoothing method \eqref{eq:f_smooth}, if either $\pi_t^{\mathrm{ref}}=\pi_t^{\mathrm{out}}$ or \(\alpha_t \le \frac{\min\{\gamma_t^-,\gamma_t^+\}}{2\,\bar a\,\|\pi_t^{\mathrm{ref}}-\pi_t^{\mathrm{out}}\|_2}\),
% \[
% \alpha_t\ \le\ \frac{\min\{\gamma_t^-,\gamma_t^+\}}{2\,\bar a\,\|\pi_t^{\mathrm{ref}}-\pi_t^{\mathrm{out}}\|_2},
% \]
then 
% every pricing minimizer $q(\pi_t^{\mathrm{sep}})$ belongs to $\mathcal{Q}_-(\pi_t^{\mathrm{out}})$.  Consequently, 
$r_{q(\pi_t^{\mathrm{sep}})}(\pi_t^{\mathrm{out}})<0$, and mispricing is prevented.
\end{corollary}

Unlike smoothing, under penalization the dependence of $\pi_t^{\mathrm{sep}}$ on the penalty parameter is implicit. To relate the penalty magnitude to the induced dual deviation, we use a standard polyhedral error bound. Let $D_R := \{\pi \in \mathbb{R}^m : A_R^\top \pi \le c_R\}$ and let $M_t \subseteq D_R$ denote the set of optimal dual solutions. Then there exists $\lambda_t > 0$ such that
\(
b^\top \pi_t^{\mathrm{out}} - b^\top \pi \ge \lambda_t\,\mathrm{dist}(\pi,M_t), \ \forall \pi \in D_R
\)
\citep[Theorem~3.5]{burke1993weak}. When the dual optimum is unique, this reduces to
\begin{equation}\label{eq:error_bound_unique}
b^\top \pi_t^{\mathrm{out}} - b^\top \pi \ge \lambda_t \|\pi - \pi_t^{\mathrm{out}}\|_2.
\end{equation}
% \(
% b^\top \pi_t^{\mathrm{out}} - b^\top \pi \ge \lambda_t \|\pi - \pi_t^{\mathrm{out}}\|_2.
% \)
Since $\pi_t^{\mathrm{sep}}$ maximizes the penalized RMP, the penalty bounds its optimality gap and hence $\|\pi_t^{\mathrm{sep}} - \pi_t^{\mathrm{out}}\|_2$.
% Because $\pi_t^{\mathrm{sep}}$ maximizes the penalized RMP, the penalty term bounds the objective loss $b^\top \pi_t^{\mathrm{out}} - b^\top \pi_t^{\mathrm{sep}}$; the above error bound then converts this loss into a bound on $\|\pi_t^{\mathrm{sep}} - \pi_t^{\mathrm{out}}\|_2$. 
Combining this with Proposition~\ref{prop:local-consistency-nos-stall} yields the following bounds.

\begin{proposition}[Safe penalization bound]
\label{prop:safe-bound-penalization}
Assume the RMP dual optimum $\pi_t^{\mathrm{out}}$ is unique. 
% For the penalization operator~\eqref{eq:f_pen1}, let $\varepsilon_\infty = \max_i \{\varepsilon_{t,i}^-, \varepsilon_{t,i}^+\}$. 
For the penalization operator~\eqref{eq:f_pen1}, if $\pi_t^{\mathrm{ref}}=\pi_t^{\mathrm{out}}$, then the operator is neutral and mispricing does not occur. Otherwise, let $\varepsilon_\infty = \max_i \{\varepsilon_{t,i}^-, 
\varepsilon_{t,i}^+\}$,
if \(
\varepsilon_\infty \ \le\ \frac{\lambda_t}{\sqrt{m}\,\|\pi_t^{\mathrm{out}}-\pi_t^{\mathrm{ref}}\|_2} 
\cdot \frac{\min\{\gamma_t^-,\gamma_t^+\}}{2\,\bar a},
\)
% \[
% \varepsilon_\infty \ \le\ \frac{\lambda_t}{\sqrt{m}\,\|\pi_t^{\mathrm{out}}-\pi_t^{\mathrm{ref}}\|_2} 
% \cdot \frac{\min\{\gamma_t^-,\gamma_t^+\}}{2\,\bar a},
% \]
then $r_{q(\pi_t^{\mathrm{sep}})}(\pi_t^{\mathrm{out}})<0$ and mispricing is prevented.
\end{proposition}

Although theoretically informative, these bounds are not enforced in our algorithmic framework for two reasons. First, evaluating them is computationally expensive: it requires global reduced-cost information to determine $\gamma_t^-$ and $\gamma_t^+$, and, for penalization, solving an additional RMP to recover $\pi_t^{\mathrm{out}}$ at each iteration. Second, and more importantly, strictly remaining in the safe zone suppresses dual exploration and may exclude columns that accelerate long-run convergence. This intractability motivates the adaptive learning strategy proposed in Section~\ref{sec:method}.
At the same time, this analysis informs the learning-based design by indicating when stabilization becomes overly aggressive. In particular, if pricing repeatedly returns columns with strongly negative reduced costs but yields little objective improvement, this suggests that the stabilized dual has drifted too far from the RMP dual. These observations motivate mispricing-related features in the RL state (Section~\ref{subsubsec:state}).

\subsubsection{Convergence Guarantee}

% \edit{
We now ask: which stabilization properties ensure finite termination under aggressive pricing distortion?
% } % We now turn to the question of convergence: despite allowing for aggressive pricing distortion, what properties of the stabilization mechanism are sufficient to guarantee finite termination with an optimal solution?
Throughout, we make the following standard assumptions: (A1) the master problem and every RMP are feasible and bounded; (A2) the pricing oracle is exact; and (A3) the set of columns that can be generated is finite.

\begin{lemma}[Finite convergence]
\label{lem:convergence}
Under A1--A3, a stabilized column generation algorithm terminates finitely with an optimal solution if it satisfies the following property: whenever the same RMP persists over a sequence of iterations and that RMP is not yet optimal for the full master problem, the stabilization mechanism eventually reduces the pricing distortion enough that the pricing step returns a column separating the current RMP dual solution.
\end{lemma}

Lemma~\ref{lem:convergence} makes explicit the common convergence mechanism behind existing stabilization methods. Global convergence does not strictly require the generation of improving columns at every iteration. Rather, finite termination is guaranteed as long as the algorithm is equipped with a sufficient fallback mechanism that the method eventually decay the dual distortion and be driven back toward the neutral regime characterized in Lemma~\ref{lem:neutral}.

In smoothing, as summarized by \cite[Proposition~2]{pessoa2018automation}, convergence is ensured by keeping $\alpha_t \in [0,1)$ and choosing the reference point from historical dual information so that it remains dual-feasible and yields a valid dual bound no worse than the best one available from previous iterations.
Under these rules, a nonproductive sequence forces the stabilized dual back toward the current RMP dual, so the pricing step eventually recovers a separating column if one exists. 
In penalization, convergence is typically ensured by driving the penalty parameters to zero, either according to a predefined schedule or triggered by successive non-improving iterations. As the penalties vanish, the penalized problem reduces to the standard RMP, so the pricing distortion disappears by Lemma~\ref{lem:neutral}. 
% \edit{We note, however, that \cite{fang2025learning} obtain the stabilized dual directly from a learned policy without forcing the dual deviation to vanish. Thus, in theory, if the learned policy generalizes poorly, it could induce repeated mispricing and prolonged stalling.}
This perspective guides our learning framework, where the policy adapts stabilization parameters while ensuring distortion eventually vanishes, preserving convergence.

\subsection{Reference Point Selection}\label{sec:reference point}

In practice, reference points are usually chosen from historical dual information, e.g., the previous stabilized dual, a past best dual, or a weighted aggregate of past dual iterates.
% \citep{wentges1997weighted,du1999stabilized,neame1999nonsmooth,pessoa2018automation}. 
Such choices are standard in both smoothing and penalization because they anchor stabilization near a region already visited by the algorithm.

Several works also suggest that if \textit{a priori} information about the optimal dual is available, it is a natural candidate for the reference point \citep{du1999stabilized,amor2009choice}. Recent penalization methods \citep{kraul2023machine,shen2024adaptive} follow this idea by using an ML-based prediction $\pi^{\mathrm{pred}}$ as a fixed reference point. Unless the prediction is exact, however, $\pi^{\mathrm{pred}}$ need not be dual-feasible for the full master problem, leading to different behaviors under smoothing and penalization.

For smoothing, if the reference point $\pi^{\mathrm{ref}}$ is fixed and distinct from the current RMP optimum, the stabilization operator does not drive $\pi_t^{\mathrm{sep}}$ toward $\pi_t^{\mathrm{out}}$ unless the smoothing parameter $\alpha_t$ vanishes. Thus, setting $\pi_t^{\mathrm{ref}}=\pi^{\mathrm{pred}}$ alone does not preserve the standard convergence mechanism of \cite[Proposition~2]{pessoa2018automation}. To use learning-based predictions while preserving convergence, we propose a hybrid reference point strategy.
\begin{proposition}\label{prop:smooth-fix-hybrid}
Consider a smoothing method with $\alpha_t\in[0,1)$. Suppose that, for the first $T_0$ iterations, it uses the fixed predicted dual $\pi^{\mathrm{pred}}$ as the reference point, and from iteration $T_0+1$ onward switches to a standard reference selection rule, e.g., $\pi_t^{\mathrm{ref}}=\pi_{t-1}^{\mathrm{sep}}$. Then the algorithm terminates finitely with an optimal solution.
\end{proposition}
By contrast, penalization is robust to the choice of reference point: as the penalty parameters vanish, the penalized problem reduces to the unstabilized RMP. Hence, predicted duals can be used as fixed reference points without affecting convergence.

\section{RL-Guided Stabilization for Column Generation}
\label{sec:method}

% In this section, we introduce the Reinforcement Learning--Guided Stabilization Column Generation (RLSCG) framework, which learns to adaptively select the stabilization parameter at each iteration to accelerate convergence.

%The preceding section shows that the effectiveness of stabilization hinges on a delicate trade-off: aggressive parameter choices can promote exploration but may induce mispricing, whereas conservative choices preserve local progress but may sacrifice opportunities for accelerated convergence. It also shows that the ``safe'' parameter region depends on quantities that are global and prohibitively expensive to evaluate online, making purely analytical tuning impractical within the algorithm. Motivated by this, we cast stabilization parameter selection as a sequential decision-making problem and propose a Reinforcement Learning--Guided Stabilization Column Generation (RLSCG) framework. At each iteration, the RL agent uses the current state of the restricted master problem and the recent algorithmic trajectory to adaptively choose the stabilization parameter, with the goal of improving overall convergence.

The preceding analysis highlights a fundamental trade-off: aggressive parameter choices promote exploration but may induce mispricing, while conservative choices preserve local progress but may limit faster convergence. In this section, we present a Reinforcement Learning–Guided Stabilization Column Generation (RLSCG) framework that applies to both smoothing and penalization, adapting stabilization parameters based on the RMP state and algorithmic trajectory to navigate this trade-off and accelerate convergence.

\subsection{Methodology Framework}
\label{subsec:overview}

The core objective of RLSCG is to learn a policy for selecting stabilization parameters that accelerate convergence for a given method. We formulate the CG process as a \emph{sequential decision-making problem}: at iteration $t$, the agent observes a state $s_t \in \mathcal{S}$ derived from the current RMP and algorithmic progress, and selects a stabilization parameter $\theta_t \in \Theta$ via a policy $\pi_{\mathrm{RL}}$.
$
\theta_t \sim \pi_{\mathrm{RL}}(\cdot \mid s_t).
$
The operator $f_t$ then produces the pricing dual $\pi_t^{\mathrm{sep}}$. In this way, RLSCG replaces fixed or rule-based updates with a learned state-dependent policy. Figure~\ref{fig:methodology} summarizes the framework.

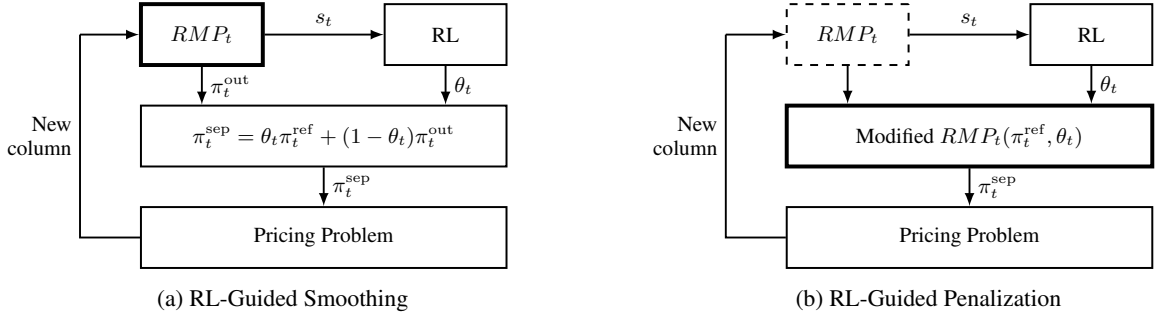
\begin{figure}[ht!]
\centering
\captionsetup[subfigure]{justification=centering,font=small}

\begingroup
\tikzset{
    >=latex, 
    font=\small,
    box/.style={
        draw,
        rectangle,
        align=center,
        inner sep=5pt,
        minimum height=0.90cm,
        line width=0.75pt
    },
    solvedbox/.style={
        draw,
        rectangle,
        align=center,
        inner sep=5pt,
        minimum height=0.90cm,
        line width=1.5pt
    },
    unsolvedbox/.style={
        draw,
        dashed,
        rectangle,
        align=center,
        inner sep=5pt,
        minimum height=0.90cm,
        line width=0.75pt
    },
    flow/.style={->, line width=0.72pt},
    plain/.style={line width=0.72pt}
}

% =========================
% (a) RL-Guided Smoothing
% =========================
\begin{subfigure}[t]{0.48\textwidth}
\centering
% \begin{tikzpicture}[x=1cm,y=1cm]
\begin{tikzpicture}[x=1cm,y=1cm, scale=0.9, transform shape]
    \path[use as bounding box] (-1.3,-0.5) rectangle (5.7,3.5);

    \node[solvedbox, minimum width=1.8cm] (rmp) at (1.0, 3.0) {$RMP_t$};
    \node[box, minimum width=1.8cm] (rl)  at (4.6, 3.0) {RL};

    \node[box, minimum width=5.4cm] (mid) at (2.8, 1.5) {$\pi_t^{\mathrm{sep}} = \theta_t \pi_t^{\mathrm{ref}} + (1-\theta_t)\pi_t^{\mathrm{out}}$};

    \node[box, minimum width=5.4cm] (pricing) at (2.8, 0.0) {Pricing Problem};

    \draw[flow] (rmp.east) -- node[above] {$s_t$} (rl.west);
    \draw[flow] (rmp.south) -- node[right] {$\pi_t^{\mathrm{out}}$} (rmp.south |- mid.north);
    \draw[flow] (rl.south) -- node[right] {$\theta_t$} (rl.south |- mid.north);
    \draw[flow] (mid.south) -- node[right] {$\pi_t^{\mathrm{sep}}$} (pricing.north);

    \coordinate (L1) at (-0.8, 0.0);
    \coordinate (L2) at (-0.8, 3.0);
    \draw[plain] (pricing.west) -- (L1) -- (L2);
    \draw[flow]  (L2) -- (rmp.west);
    \node[align=right, anchor=east] at (-0.8, 1.5) {New\\column};

\end{tikzpicture}
\caption{RL-Guided Smoothing}
\end{subfigure}
\hfill
% =========================
% (b) RL-Guided Penalization
% =========================
\begin{subfigure}[t]{0.48\textwidth}
\centering
% \begin{tikzpicture}[x=1cm,y=1cm]
\begin{tikzpicture}[x=1cm,y=1cm, scale=0.9, transform shape]
    \path[use as bounding box] (-1.3,-0.5) rectangle (5.7,3.5);

    \node[unsolvedbox, minimum width=1.8cm] (rmp) at (1.0, 3.0) {$RMP_t$};
    \node[box, minimum width=1.8cm] (rl)  at (4.6, 3.0) {RL};

    \node[solvedbox, minimum width=5.4cm] (mid) at (2.8, 1.5) {Modified $RMP_t(\pi_t^{\mathrm{ref}}, \theta_t)$};

    \node[box, minimum width=5.4cm] (pricing) at (2.8, 0.0) {Pricing Problem};

    \draw[flow] (rmp.east) -- node[above] {$s_t$} (rl.west);
    \draw[flow] (rmp.south) -- (rmp.south |- mid.north); 
    \draw[flow] (rl.south) -- node[right] {$\theta_t$} (rl.south |- mid.north);
    \draw[flow] (mid.south) -- node[right] {$\pi_t^{\mathrm{sep}}$} (pricing.north);

    \coordinate (L1) at (-0.8, 0.0);
    \coordinate (L2) at (-0.8, 3.0);
    \draw[plain] (pricing.west) -- (L1) -- (L2);
    \draw[flow]  (L2) -- (rmp.west);
    \node[align=right, anchor=east] at (-0.8, 1.5) {New\\column};

\end{tikzpicture}
\caption{RL-Guided Penalization}
\end{subfigure}

\caption{Methodology framework of RLSCG. 
% At iteration $t$, the RL agent observes the state $s_t$ extracted from the current restricted master problem and selects the stabilization parameter $\theta_t$, which determines the stabilized dual vector $\pi_t^{\mathrm{sep}}$ used in pricing. In the smoothing variant (left), the current $RMP_t$ is explicitly solved to obtain $\pi_t^{\mathrm{out}}$, and $\pi_t^{\mathrm{sep}}$ is computed as a convex combination. In the penalization variant (right), the current $RMP_t$ provides only structural information for the RL agent; the modified $RMP_t(\pi_t^{\mathrm{ref}},\theta_t)$ is the formulation explicitly solved to produce $\pi_t^{\mathrm{sep}}$. Thick borders indicate solved RMP formulation at that iteration; dashed borders indicate unsolved ones.
At iteration $t$, the RL agent observes a state $s_t$ from the current $RMP_t$ and selects $\theta_t$. 
% for computing the dual vector $\pi_t^{\mathrm{sep}}$ used in pricing. 
Left: smoothing solves $RMP_t$ to obtain $\pi_t^{\mathrm{out}}$ and then forms $\pi_t^{\mathrm{sep}}$. Right: penalization solves the modified $RMP_t(\pi_t^{\mathrm{ref}},\theta_t)$ to obtain $\pi_t^{\mathrm{sep}}$.
% , while the current $RMP_t$ is used only for state construction. 
Thick borders denote formulations solved at iteration $t$; dashed borders denote unsolved ones.}
\label{fig:methodology}

\endgroup
\end{figure}

%\caption{Methodology framework of the Reinforcement Learning--Guided Stabilization Column Generation (RLSCG). At iteration $t$, the RL agent observes the state $s_t$ extracted from the current restricted master problem and selects the stabilization parameter $\theta_t$, which determines the stabilized dual vector $\pi_t^{\mathrm{sep}}$ used in pricing. The left panel shows the smoothing variant, where the current $RMP_t$ is explicitly solved to obtain $\pi_t^{\mathrm{out}}$, and $\pi_t^{\mathrm{sep}}$ is then computed from $\pi_t^{\mathrm{ref}}$, $\theta_t$, and $\pi_t^{\mathrm{out}}$. The right panel shows the penalization variant, where the current $RMP_t$ is used only to construct the state for the RL agent, while the modified $RMP_t(\pi_t^{\mathrm{ref}},\theta_t)$ is the formulation explicitly solved to produce $\pi_t^{\mathrm{sep}}$. Boxes with thick borders indicate the RMP formulation explicitly solved at that iteration.}
% \label{fig:methodology}

\subsection{RL Formulation}
\label{sec:RL}

The remainder of this section details each component of the RL formulation: the state representation, action space, state transition dynamics, and reward function.

\subsubsection{State Representation}
\label{subsubsec:state}
% The RL agent requires a state representation that captures both the structure of the current RMP and the trajectory of the algorithm. 
The state vector $s_t$ combines two complementary sources of information: a graph-based encoding of the RMP structure and a vector of global features summarizing algorithmic progress. We describe each in turn.

\paragraph{Bipartite graph representation.}
Following recent work that applies Graph Neural Networks to column generation \citep{morabit2021machine,chi2022deep}, we model the RMP at iteration $t$ as an undirected bipartite graph $G_t = (\mathcal{V}_t \cup \mathcal{C}, E_t)$, where $\mathcal{V}_t$ is the set of column nodes and $\mathcal{C}$ is the set of constraint nodes; see Figure~\ref{fig:bipartite_rmp_graph} in Appendix~\ref{appendix:column feature} for an illustration. An edge $(v, c) \in E_t$ exists whenever the coefficient $A_{cv}$ is nonzero. This bipartite structure naturally reflects the column--constraint interactions in the RMP.

Each node carries a feature vector that encodes information relevant to the stabilization decision. Column node features $\mathbf{x}_v$ capture the role of each variable in the current RMP, including its basis status, the number of constraints it participates in, and its reduced cost at generation. Constraint node features $\mathbf{x}_c$ encode corresponding dual information, including the current dual value, the right-hand side coefficient, and a measure of dual variability over recent iterations. The complete feature specification is in Table~\ref{appendix:node_features} of Appendix~\ref{appendix:column feature}.
% \edit{SOME APPENDIX ARE NUMBERS, OTHERS ARE LETTERS.}

We process the bipartite graph with a message-passing GNN consisting of $L$ layers. Let $\mathbf{r}_{v}^{(\ell)}$ and $\mathbf{r}_{c}^{(\ell)}$ denote the embeddings of column node $v$ and constraint node $c$ at layer $\ell$, initialized as $\mathbf{r}_{v}^{(0)} = \mathbf{x}_{v}$ and $\mathbf{r}_{c}^{(0)} = \mathbf{x}_{c}$. Each layer $\ell = 1, \dots, L$ performs two sequential update phases. 
In the first phase, each constraint node aggregates information from its neighboring columns. In the second phase, each column node aggregates information from the newly updated constraint embeddings:
\begin{equation*}
\begin{aligned}
\mu_c^{(\ell)} &= \frac{1}{\sqrt{|\mathcal{N}(c)|}} \sum_{v \in \mathcal{N}(c)} \Phi_c\!\left(\mathbf{r}_c^{(\ell-1)},\,\mathbf{r}_v^{(\ell-1)}\right), \quad
\mathbf{r}_c^{(\ell)} = \Psi_c\!\left(\left[\mathbf{r}_c^{(\ell-1)},\, \mu_c^{(\ell)}\right]\right), \\
\mu_v^{(\ell)} &= \frac{1}{\sqrt{|\mathcal{N}(v)|}} \sum_{c \in \mathcal{N}(v)} \Phi_v\!\left(\mathbf{r}_v^{(\ell-1)},\,\mathbf{r}_c^{(\ell)}\right), \quad
\mathbf{r}_v^{(\ell)} = \Psi_v\!\left(\left[\mathbf{r}_v^{(\ell-1)},\, \mu_v^{(\ell)}\right]\right).
\end{aligned}
\end{equation*}
where $\mathcal{N}(c)$ denotes the neighbors of $c$, and $\Phi_c$, $\Psi_c$ are multi-layer perceptrons. 
This ordering ensures that within each layer, information flows from columns to constraints and then back, enabling bidirectional propagation per layer. After $L$ layers, the local structural representation is obtained by mean-pooling over all final node embeddings:
$$
\mathbf{h}_{\text{local}} = \frac{1}{|\mathcal{V}_t| + |\mathcal{C}|} \left( \sum_{v \in \mathcal{V}_t} \mathbf{r}^{(L)}_v + \sum_{c \in \mathcal{C}} \mathbf{r}^{(L)}_c \right).
$$

\paragraph{Global features.}
While the bipartite graph captures local RMP information, it does not reflect the global progress of the algorithm. We therefore include a global feature vector $\mathbf{g}_t$ that is processed through a separate multi-layer perceptron: \(
\mathbf{h}_{\text{global}} = \text{MLP}_{\text{global}}(\mathbf{g}_t).
\)
The vector $\mathbf{g}_t$ contains problem-specific and algorithmic features, including convergence indicators (iteration count and relative progress), stabilization indicators (the previous parameter $\theta_{t-1}$, the distance $\|\pi_t^{\mathrm{sep}}-\pi_t^{\mathrm{ref}}\|_2$, and the inter-iteration dual movement $\|\pi_t^{\mathrm{sep}}-\pi_{t-1}^{\mathrm{sep}}\|_2$), and mispricing proxies. 

The mispricing proxies indicate when stabilization becomes overly aggressive. For smoothing, mispricing is directly observable, so we track its cumulative count. For penalization, however, it is not directly observable without solving the unstabilized RMP to recover $\pi_t^{\mathrm{out}}$, which is too costly online. We therefore use proxy signals: the average minimum reduced cost and relative progress over a recent window, and the number of consecutive iterations without progress. The full specification is given in Appendix~\ref{appendix:column feature} (Table~\ref{appendix:global_features}).

\paragraph{State vector.}
The final state representation is the concatenation:
$
s_t = \left[\mathbf{h}_{\text{local}} \,;\, \mathbf{h}_{\text{global}}\right].
$

\subsubsection{Actions, Transitions, and Rewards}

Having defined the state space and what the agent observes, we now specify how it acts and what it optimizes.  

\paragraph{Action space.}
At each iteration, the agent selects a stabilization parameter $\theta_t$ from a finite discrete action space $\Theta$. For \textit{smoothing}, $\theta_t$ corresponds to the smoothing weight $\alpha_t$, and $\Theta$ is a discretization of $[0, 1)$. For \textit{penalization}~\eqref{eq:f_pen1}, we restrict attention to symmetric penalties $\varepsilon_t^- = \varepsilon_t^+ =: \varepsilon_t$; the action $\theta_t$ is then the scalar penalty coefficient $\varepsilon_t$, and $\Theta$ is a discretization of $[0, \varepsilon_{\max}]$.

\paragraph{State transition.}
Given a state $s_t$ and a selected parameter $\theta_t$, the stabilization operator computes $\pi_t^{\mathrm{sep}}$, the pricing subproblem returns a new column, and that column is added to the RMP.
% ---corresponding to one full iteration of Algorithm~\ref{alg:stabilized-cg}. 
The updated RMP then defines the next state $s_{t+1}$.

For penalization methods, we additionally employ \emph{dynamic action masking} 
to enforce eventual vanishing of the penalty parameter. Specifically, we maintain an upper bound $\overline{\varepsilon}$ and restrict the effective action space at iteration $t$ to $\{\varepsilon \in \Theta : \varepsilon \le \overline{\varepsilon}\}$. Whenever pricing returns no improving column, i.e., $\sigma(\pi_t^{\mathrm{sep}}) \ge 0$, we update $\overline{\varepsilon} \leftarrow \rho \overline{\varepsilon}$ using a fixed decay factor $\rho = 0.5$. Hence $\varepsilon_t$ eventually vanishes, restoring alignment between $\pi_t^{\mathrm{sep}}$ and $\pi_t^{\mathrm{out}}$.

\paragraph{Reward function.}
The reward signal encourages the agent to minimize the total number of iterations to convergence. Each iteration incurs a fixed cost $\eta > 0$, offset by a term that credits objective improvement:
\begin{equation}
r_t =
\begin{cases}
  \Delta_t - \eta, & \text{if a new column is added}, \\
  R_{\text{conv}}, & \text{if the algorithm terminates optimally},
\end{cases}
\label{eq:reward}
\end{equation}
where $R_{\text{conv}} > 0$ is a terminal reward and $\Delta_t$ normalizes the per-iteration progress by the initial objective value $Z_0$: \(\Delta_t = \beta \frac{\text{Improvement}_t}{|Z_0|},\) where $\beta > 0$ scales the improvement.
% \begin{equation}
% \Delta_t = \beta \frac{\text{Improvement}_t}{|Z_0|},
% \end{equation}
% where $\beta > 0$ controls the relative weight of improvement against the iteration cost $\eta$. 

% The RL policy is trained to maximize the expected cumulative reward over the entire column generation trajectory. In particular, for an episode terminating at iteration $T$, the return from iteration $t$ is $G_t = \sum_{\tau=t}^{T} \gamma^{\,\tau-t} r_\tau,$ where $\gamma \in (0,1]$ is the discount factor. 
% We note that $\eta$, $\beta$, $R_{\text{conv}}$ and $\gamma$ are fixed reward-shaping hyperparameters, described in Appendix~\ref{appendix:hyper}.

The RL policy maximizes the expected cumulative reward over the column generation trajectory. For an episode terminating at iteration $T$, the return from iteration $t$ is $G_t = \sum_{\tau=t}^{T} \gamma^{\,\tau-t} r_\tau$, where $\gamma \in (0,1]$ is the discount factor. Reward-shaping hyperparameters ($\eta$, $\beta$, $R_{\text{conv}}$, $\gamma$) are detailed in Appendix~\ref{appendix:hyper}.

% , set once and shared across all instances; they are not the subject of adaptation. 
% The adaptivity of our framework lies in the stabilization parameter $\theta_t$, which the agent adjusts at each iteration based on the observed state. 
% In our experiments, we adopt them from \cite{chi2022deep} without additional tuning, as detailed in Appendix~\ref{appendix:hyper}.

The definition of $\text{Improvement}_t$ depends on the stabilization method. For smoothing, it is the primal objective reduction:
% the primal objective reduction:
$\text{Improvement}_t = z^{\text{RMP}}_{t-1} - z^{\text{RMP}}_t.$ 
For penalization, we do not use the RMP objective directly, since it includes stabilization terms. 
% In problem classes where the stabilized dual point and the exact pricing value yield a valid lower bound $z_t^{\mathrm{LB}}$, we maintain the incumbent lower bound \(\underline z_t^{\mathrm{LB}}
% =
% \max\left\{\underline z_{t-1}^{\mathrm{LB}},\, z_t^{\mathrm{LB}}\right\},\) and set \(\mathrm{Improvement}_t
% =\underline z_t^{\mathrm{LB}} - \underline z_{t-1}^{\mathrm{LB}} .\)
In problem classes where a valid lower bound $z_t^{\mathrm{LB}}$ exists, we track the incumbent $\underline z_t^{\mathrm{LB}} = \max\left\{\underline z_{t-1}^{\mathrm{LB}},\, z_t^{\mathrm{LB}}\right\}$ and set $\mathrm{Improvement}_t = \underline z_t^{\mathrm{LB}} - \underline z_{t-1}^{\mathrm{LB}}$. 
The availability of such a lower bound, however, is not universal; it depends on the master formulation and on whether the exact pricing value can be converted into a valid bound for the original problem \citep{desrosiers2005primer,amor2006proximal}. 
When such bounds are unavailable, one can resolve the unpenalized RMP during training to compute the objective reduction. Crucially, this additional solve is restricted to training and not required at deployment. For the Cutting Stock Problem studied here, we employ the valid lower bound of \cite{farley1990note}: $z_t^{\mathrm{LB}} = b^\top \pi_t^{\mathrm{sep}} / (1-c_t^*)$ if $c_t^* < 1$, and $0$ otherwise, where $c_t^* := \min_{p\in P}\{1-a_p^\top \pi_t^{\mathrm{sep}}\}$.

% The availability of such a lower bound, however, is not universal; it depends on the master formulation and on whether the exact pricing value can be converted into a valid bound for the original problem \citep{desrosiers2005primer,amor2006proximal}. 
% Therefore, when no such bound is readily available, one can resolve the unpenalized RMP over the current column set at each training iteration, defining $\text{Improvement}_t$ as the resulting objective reduction.
% We emphasize that this additional solve is incurred only during the training phase; at deployment, the trained policy observes only the state features described in Section~\ref{subsubsec:state} and does not require resolving the unpenalized RMP.
% For the Cutting Stock Problem studied in this paper, we employ the valid lower bound of \cite{farley1990note}, computed as $z_t^{\mathrm{LB}} = b^\top \pi_t^{\mathrm{sep}} / (1-c_t^*)$ if $c_t^* < 1$, and $0$ otherwise, where $c_t^* := \min_{p\in P}\{1-a_p^\top \pi_t^{\mathrm{sep}}\}$ denotes the minimum reduced cost yielded by the pricing problem.

The learned policy is not trained to avoid mispricing per se, but to maximize long-run progress via cumulative reward. In this sense, the RL framework learns when aggressive stabilization is worth its short-term cost and when it is not; a trade-off difficult to capture with fixed update rules.

\section{Computational Experiments}\label{sec:experiment}

To assess the effectiveness of the approach, we implement RLSCG with both the smoothing operator (Eq.~\eqref{eq:f_smooth}) and the penalization operator (Eq.~\eqref{eq:f_pen1}), denoted RLSCG(S) and RLSCG(P), respectively.
We first compare RLSCG with traditional CG to assess whether RL-guided stabilization improves convergence. We then compare it with rule-based methods to determine whether the gains come from learning or can be matched by  simpler parameter updates. Next, we evaluate the use of predicted duals as reference points, extending an idea studied only for penalization to smoothing. Finally, we compare RLSCG with learning-based column selection to assess where learning is most effective.

\subsection{Experimental Setup}\label{subsec:setup}

\paragraph{Problem description.}

We use the Cutting Stock Problem (CSP) as our testbed. We formulate it as a set covering model in which each column is a feasible cutting pattern and the pricing subproblem is a knapsack problem. 
% The full formulation is given in Appendix~\ref{appendix:CSP}. 
Although the CSP is an integer program, we focus on CG behavior and therefore solve only the root-node linear relaxation.

\paragraph{Action space.}
The action space has 20 discrete values in both implementations. For smoothing, these are values of $\alpha$ with step size 0.05 over $[0,1)$; for penalization, with symmetric penalties $\varepsilon_t^-=\varepsilon_t^+=\varepsilon_t$, they are equally spaced values of $\varepsilon$ over $[0,1]$.

\paragraph{Instances.}
For training, we generate 500 random instances with roll length $L \sim \mathcal{U}[200,800]$ and number of item types $m \in [20,100]$. For each instance, we sample $lb \sim \mathcal{U}[0.05,0.45]$ and $ub \sim \mathcal{U}[0.5,0.85]$, then draw each item length $\ell_i$ uniformly from $[lb\,L,\; ub\,L]$ and each demand $d_i$ from $[1,20]$.
For evaluation, we generate four groups of synthetic test sets (Gen\_1 to Gen\_4) that vary the ratio of item size to roll length. 
Each group contains Small ($m \in [20,50]$), Medium ($m \in [50,100]$), and Large ($m \in [100,150]$) instances; generation details are given in Table~\ref{tab:data_generation} in Appendix \ref{appendix:CSP}. We also test on BPPLIB benchmark instances \citep{delorme2018bpplib}, including Falkenauer (U and T), Hard28, Schwerin, and Waescher.

\paragraph{Implementation.}
We train the RL agent with Deep Q-Network; details of the GNN architecture, hyperparameters, and computing environment are given in Appendix~\ref{appendix:hyper}. For each stabilization method, a single policy is trained once (about 8 hours) and applied to all test sets. At deployment, the policy adds negligible per-iteration overhead (see Section~\ref{subsec:comparison_rlcg}). The code will be released on GitHub upon acceptance.
% \edit{ 
All runtimes are reported within a common implementation environment, where the RMP is updated incrementally across iterations rather than rebuilt from scratch; iteration counts are invariant to this choice and serve as our primary algorithmic metric.
% }

\subsection{Overall Convergence Improvement}
\label{subsec:overall_performance}

We first compare RLSCG(S) and RLSCG(P) with traditional column generation (TCG), which uses no stabilization and adds, at each iteration, the column with minimum reduced cost under $\pi_t^{\mathrm{out}}$. Table~\ref{tab:overall_results} reports average iterations and computation time over all test sets.

\begin{table}[htbp]
\centering
\tiny
\caption{Performance of TCG, RLSCG(S), and RLSCG(P); RLSCG uses predicted duals. Percentages are relative to TCG; bold denotes the best performance in each row.}
\label{tab:overall_results}
\resizebox{\textwidth}{!}{%
\begin{tabular}{ccccccccc}
\toprule
\multicolumn{2}{c}{\multirow{2}{*}{\textbf{Test Set}}} & \multirow{2}{*}{\textbf{\# Ins.}} & \multicolumn{2}{c}{\textbf{TCG}} & \multicolumn{2}{c}{\textbf{RLSCG(S)}} & \multicolumn{2}{c}{\textbf{RLSCG(P)}} \\
\cmidrule(lr){4-5} \cmidrule(lr){6-7} \cmidrule(lr){8-9}
 & & & \textbf{Iter.} & \textbf{Time(s)} & \textbf{Iter.} & \textbf{Time(s)} & \textbf{Iter.} & \textbf{Time(s)} \\
\midrule
\multirow{3}{*}{Gen\_1} & Small & 50 & 72.4 & 0.98 & \textbf{39.0} (-46\%) & \textbf{0.64} (-35\%) & 48.5 (-33\%) & 0.81 (-17\%) \\
 & Medium & 50 & 243.2 & 3.57 & \textbf{80.0} (-67\%) & \textbf{1.52} (-57\%) & 112.5 (-54\%) & 2.18 (-39\%) \\
 & Large & 50 & 641.3 & 10.25 & \textbf{192.3} (-70\%) & \textbf{4.35} (-58\%) & 217.3 (-66\%) & 4.74 (-54\%) \\
\midrule
\multirow{3}{*}{Gen\_2} & Small & 50 & 20.5 & 0.28 & \textbf{12.4} (-40\%) & \textbf{0.20} (-28\%) & 16.6 (-19\%) & 0.28 (-2\%) \\
 & Medium & 50 & 66.6 & 0.96 & \textbf{22.6} (-66\%) & \textbf{0.38} (-60\%) & 34.0 (-49\%) & 0.60 (-38\%) \\
 & Large & 50 & 231.2 & 3.39 & \textbf{63.6} (-72\%) & \textbf{1.11} (-67\%) & 69.0 (-70\%) & 1.27 (-63\%) \\
\midrule
\multirow{3}{*}{Gen\_3} & Small & 50 & 37.6 & 0.51 & \textbf{17.7} (-53\%) & \textbf{0.29} (-44\%) & 24.6 (-35\%) & 0.40 (-21\%) \\
 & Medium & 50 & 140.8 & 2.01 & \textbf{37.4} (-73\%) & \textbf{0.64} (-68\%) & 53.6 (-62\%) & 0.94 (-53\%) \\
 & Large & 50 & 418.8 & 6.25 & \textbf{89.9} (-79\%) & \textbf{1.62} (-74\%) & 102.5 (-76\%) & 1.89 (-70\%) \\
\midrule
\multirow{3}{*}{Gen\_4} & Small & 50 & 39.1 & \textbf{0.57} & \textbf{37.7} (-4\%) & 0.66 (+16\%) & 42.1 (+8\%) & 0.74 (+30\%) \\
 & Medium & 50 & 87.7 & \textbf{1.35} & \textbf{80.9} (-8\%) & 1.59 (+18\%) & 89.2 (+2\%) & 1.70 (+26\%) \\
 & Large & 50 & 150.5 & \textbf{2.40} & \textbf{138.3} (-8\%) & 2.70 (+12\%) & 151.1 (+1\%) & 3.06 (+27\%) \\
\midrule
\multicolumn{2}{c}{Falkenauer\_U} & 80 & 198.2 & 2.80 & \textbf{140.7} (-29\%) & \textbf{2.44} (-13\%) & 162.4 (-18\%) & 2.92 (+4\%) \\
\multicolumn{2}{c}{Falkenauer\_T} & 80 & 570.8 & 9.31 & \textbf{259.9} (-54\%) & \textbf{5.42} (-42\%) & 269.3 (-53\%) & 5.64 (-39\%) \\
\multicolumn{2}{c}{Schwerin} & 200 & 59.2 & 0.81 & \textbf{45.3} (-23\%) & \textbf{0.75} (-7\%) & 66.4 (+12\%) & 1.12 (+39\%) \\
\multicolumn{2}{c}{Waescher} & 17 & 61.6 & \textbf{1.33} & \textbf{60.5} (-2\%) & 1.41 (+6\%) & 62.2 (+1\%) & 1.49 (+12\%) \\
\multicolumn{2}{c}{Hard\_28} & 28 & 402.4 & \textbf{8.53} & 377.1 (-6\%) & 9.84 (+15\%) & \textbf{364.9} (-9\%) & 9.88 (+16\%) \\
\bottomrule
\end{tabular}%
}
\end{table}

% The results demonstrate that RLSCG outperforms TCG on most synthetic and benchmark instances, particularly as problem complexity increases. Among the two variants, RLSCG(S) achieves the most substantial improvements: on Gen\_1 to Gen\_3, it reduces computation time by approximately 80\% to 90\% compared to TCG. This validates the value of controlled pricing distortion: by dynamically balancing dual exploration and local exploitation, stabilization can substantially accelerate overall convergence.

% \edit{Results show that RLSCG outperforms TCG on many synthetic and benchmark instances, particularly as complexity increases. Between the variants, RLSCG(S) achieves consistent, substantial improvements across the synthetic sets, reducing computation time on Gen\_1 to Gen\_3 by 28\%--74\% compared to TCG. This validates the value of controlled pricing distortion: by dynamically balancing dual exploration and local exploitation, stabilization can substantially accelerate overall convergence.}

% \edit{
Results show that RLSCG outperforms TCG on many synthetic and benchmark instances, particularly as complexity increases. Between the variants, RLSCG(S) achieves consistent, substantial improvements across the synthetic sets, reducing iteration counts on Gen\_1 to Gen\_3 by 40--80\% relative to TCG, with corresponding time reductions of 28--74\%.
% }
%This validates the value of controlled pricing distortion: by dynamically balancing dual exploration and local exploitation, stabilization can substantially accelerate overall convergence.

% \edit{
Gen\_4 is the exception in time but not in iterations: RLSCG(S) still reduces iterations by 4--8\% over TCG, while TCG is slightly faster in wall-clock time.
% } 
These instances feature very small items relative to the roll length, yielding high-cardinality patterns and severe degeneracy. Consequently, many bases share similar objective values, and numerous nonbasic columns have near-zero reduced costs. This shrinks the separation margin $\min\{\gamma_t^-,\gamma_t^+\}$ (Proposition~\ref{prop:local-consistency-nos-stall}), contracting the safe perturbation radius that prevents mispricing. 
% \edit{
The discrepancy of fewer iterations yet longer runtimes reflects RL inference overhead and an inherent stabilization drawback: stabilized duals can make pricing computationally harder by weakening reduced-cost signals \citep{lubbecke2005selected}.
% }

%\edit{An exception occurs in Gen\_4, where TCG outperforms the stabilized methods. These instances feature very small items relative to the roll length, yielding high-cardinality patterns and severe degeneracy. Consequently, many bases share similar objective values, and numerous nonbasic columns have near-zero reduced costs. This shrinks the separation margin $\min\{\gamma_t^-,\gamma_t^+\}$ (Proposition~\ref{prop:local-consistency-nos-stall}), contracting the safe perturbation radius that prevents mispricing. While indicating that stabilization acceleration is problem-dependent, RLSCG remains robust to this degenerate regime in iteration counts. This discrepancy—fewer iterations yet longer runtimes—stems from RL inference overhead and, more importantly, an inherent stabilization drawback: stabilized duals may make the pricing subproblem computationally harder by weakening reduced-cost signals and degrading pruning efficiency \citep{lubbecke2005selected}.}

While RLSCG(P) improves upon TCG across many test sets, it generally lags behind RLSCG(S), this gap stems from their differing stabilization mechanisms: smoothing explicitly interpolates duals using a single bounded parameter $\alpha \in [0, 1)$, whereas penalization modifies the RMP structure via unbounded weight coefficients. This inherently complicates regulating the distance between $\pi_{\text{sep}}$ and $\pi_{\text{out}}$, making pricing distortion and parameter selection harder to control.

\subsection{Learned vs. Rule-Based Parameter Adaptation}
\label{subsec:adaptive_parameter}

To determine whether RLSCG's performance gains stem from the learned policy or could be achieved by simpler rule-based parameter updates, we compare it against established adaptive strategies.
Importantly, the experimental protocols differ between the two methods. For consistency with recent studies \citep{kraul2023machine,shen2024adaptive}, RLSCG(P) uses predicted dual vectors as reference points. However, since predicted references remain unstudied in smoothing, this section evaluates an RLSCG(S) variant that relies solely on the algorithm's dual history for reference points. We defer the analysis of predicted-reference smoothing to Section~\ref{subsec:predicted_dual_smoothing}.

\paragraph{Smoothing.}

We compare RLSCG(S) against two rule-based adaptive smoothing strategies:
\begin{itemize}
    \item \textit{ASCG-1} \citep{wentges1997weighted} uses a progress-based rule:
  $\alpha_t = 1 - \frac{1}{\min\left\{\overline{c},\; \frac{t + N}{2}\right\}},$ where $t$ is the iteration index, $N \ge 1$ tracks the cumulative number of dual bound improvements, and $\overline{c}=10$ caps the growth of $\alpha_t$. 
    \item \textit{ASCG-2} \citep{pessoa2018automation} employs a mispricing-triggered fallback. Starting from $\alpha_0 = 0.5$, each consecutive mispricing event $k$ reduces the weight to $\alpha = \max(0, 1 - k(1 - \alpha_0))$ and retriggers pricing, repeating until mispricing no longer occurs or $\alpha$ reaches zero.
\end{itemize}

As shown in Table~\ref{tab:smoothing_results}, RLSCG(S) substantially outperforms both baselines, especially on large instances. ASCG-2 generally outperforms ASCG-1 by reducing $\alpha$ to correct mispricing and enforce conservative updates. RLSCG(S) improves upon this by strategically tolerating mispricing to balance exploration and exploitation.

% \edit{
Furthermore, Tables~\ref{tab:overall_results} and \ref{tab:smoothing_results} reveal that rule-based smoothing underperforms TCG on these CSP instances, whereas RLSCG(S) consistently improves upon it. This highlights that while stabilization has the potential to accelerate convergence, carefully controlling its magnitude is critical to realizing these benefits.
% }

%\edit{Furthermore, Tables~\ref{tab:overall_results} and \ref{tab:smoothing_results} reveal that ASCG underperforms TCG on CSP instances, whereas RLSCG(S) consistently improves upon it. This highlights that while stabilization have the potential to accelerate convergence, carefully controlling its magnitude is critical to realizing these benefits.}

\begin{table}[htbp]
\tiny
\centering
\caption{Comparison of Smoothing Methods. 
Percentages are relative to the best value within each row.
%For each test set, the minimum values in each scale are highlighted in bold. 
}
\label{tab:smoothing_results}
\resizebox{\textwidth}{!}{%
\begin{tabular}{cccccccc}
\toprule
\multicolumn{2}{c}{\multirow{2}{*}{\textbf{Test set}}}& \multicolumn{2}{c}{\textbf{ASCG-1}} & \multicolumn{2}{c}{\textbf{ASCG-2}} & \multicolumn{2}{c}{\textbf{RLSCG(S)}} \\
\cmidrule(lr){3-4} \cmidrule(lr){5-6} \cmidrule(lr){7-8}
 &  & \multicolumn{1}{c}{\textbf{Iter.}} & \multicolumn{1}{c}{\textbf{Time (s)}} & \multicolumn{1}{c}{\textbf{Iter.}} & \multicolumn{1}{c}{\textbf{Time (s)}} & \multicolumn{1}{c}{\textbf{Iter.}} & \multicolumn{1}{c}{\textbf{Time (s)}} \\
\midrule
Gen\_1 & Small & 141.7 (+206\%) & 1.93 (+159\%) & 70.1 (+52\%) & 0.99 (+32\%) & \textbf{46.3} & \textbf{0.74} \\
 & Medium & 327.5 (+217\%) & 6.84 (+255\%) & 251.7 (+144\%) & 4.70 (+144\%) & \textbf{103.2} & \textbf{1.92} \\
 & Large & 1159.0 (+477\%) & 45.84 (+931\%) & 680.1 (+239\%) & 12.84 (+189\%) & \textbf{200.7} & \textbf{4.45} \\
\midrule
Gen\_2 & Small & 105.9 (+600\%) & 1.43 (+495\%) & 19.0 (+25\%) & 0.32 (+32\%) & \textbf{15.1} & \textbf{0.24} \\
 & Medium & 164.4 (+427\%) & 2.29 (+338\%) & 62.6 (+101\%) & 1.11 (+111\%) & \textbf{31.2} & \textbf{0.52} \\
 & Large & 340.6 (+435\%) & 5.04 (+339\%) & 226.5 (+256\%) & 4.15 (+261\%) & \textbf{63.6} & \textbf{1.15} \\
\midrule
Gen\_3 & Small & 122.2 (+444\%) & 1.65 (+344\%) & 35.8 (+59\%) & 0.60 (+61\%) & \textbf{22.5} & \textbf{0.37} \\
 & Medium & 236.1 (+382\%) & 3.34 (+299\%) & 135.0 (+175\%) & 2.17 (+159\%) & \textbf{49.0} & \textbf{0.84} \\
 & Large & 523.2 (+438\%) & 7.83 (+358\%) & 409.4 (+321\%) & 7.51 (+339\%) & \textbf{97.2} & \textbf{1.71} \\
\midrule
Gen\_4 & Small & 81.9 (+114\%) & 1.25 (+94\%) & 56.5 (+48\%) & 1.05 (+63\%) & \textbf{38.3} & \textbf{0.65} \\
 & Medium & 147.9 (+80\%) & 2.46 (+62\%) & 120.8 (+47\%) & 2.47 (+63\%) & \textbf{82.0} & \textbf{1.52} \\
 & Large & 225.1 (+63\%) & 3.85 (+43\%) & 197.3 (+43\%) & 4.07 (+51\%) & \textbf{138.2} & \textbf{2.70} \\
\midrule
\multicolumn{2}{c}{Falkenauer\_U} & 254.2 (+50\%) & 3.60 (+22\%) & 207.3 (+22\%) & 2.98 (+1\%) & \textbf{169.4} & \textbf{2.96} \\
\multicolumn{2}{c}{Falkenauer\_T} & 725.9 (+165\%) & 16.01 (+186\%) & 575.4 (+110\%) & 11.42 (+104\%) & \textbf{273.6} & \textbf{5.59} \\
\multicolumn{2}{c}{Schwerin} & 85.6 (+58\%) & 1.18 (+40\%) & 60.1 (+11\%) & \textbf{0.84} & \textbf{54.1} & 0.89 (+6\%) \\
\multicolumn{2}{c}{Waescher} & 98.4 (+62\%) & 4.96 (+255\%) & 89.8 (+48\%) & 3.96 (+184\%) & \textbf{60.6} & \textbf{1.40} \\
\multicolumn{2}{c}{Hard28} & 503.2 (+34\%)& 31.13 (+217\%)& 507.0 (+35\%) & 21.88 (+123\%) & \textbf{375.4} & \textbf{9.81} \\
\bottomrule
\end{tabular}%
}
\end{table}

\paragraph{Penalization.}
We compare RLSCG(P) against rule-based adaptive penalization strategies:
\begin{itemize}
    \item \textit{APCG-1} \citep{kraul2023machine} uses predicted duals as a fixed reference with an initial penalty $\varepsilon = 0.1$, halving $\varepsilon$ upon mispricing.

    \item \textit{APCG-2} \citep{shen2024adaptive} also uses predicted duals as a fixed reference but dynamically updates $\varepsilon = \frac{\mathrm{rc}}{\mathrm{rc}-1}$ if the minimum reduced cost $\mathrm{rc} < 0$, and $\varepsilon = 0$ otherwise.
\end{itemize}

% \edit{Table~\ref{tab:penalization_results} shows that RLSCG(P) outperforms both rule-based penalization baselines on most instances, especially on harder instances. On smaller problems, the margin narrows and APCG is occasionally competitive. This suggests heuristic updates can be adequate on some instances. Also, since penalization baselines already exploit predicted duals as reference points, it elevates baseline efficiency and narrows the margin for improvement. 
% This raises a natural question: can the smoothing method also benefit from predicted duals?  We investigate this in the next subsection.}

% \edit{
Table~\ref{tab:penalization_results} shows that RLSCG(P) outperforms both rule-based penalization baselines on most instances, with the largest gains on harder problems. On smaller instances the margin narrows and APCG is occasionally competitive, indicating that heuristic updates can be adequate when the search landscape is less challenging. The penalization baselines already exploit predicted duals as reference points, which raises their baseline efficiency and leaves a narrower margin for further improvement. This raises a natural question: can the smoothing method also benefit from predicted duals? We investigate this in the next subsection.
% }

\begin{table}[htbp]
\tiny
\centering
\caption{Comparison of Penalization Methods. 
For each test set, the minimum values in each metric are highlighted in bold. 
Percentages indicate the relative increase (+) compared to the best value within each row and metric; ``-'' indicates that the algorithm failed to converge for some instances within the 3,600-second time limit.}
\label{tab:penalization_results}
\resizebox{\textwidth}{!}{%
\begin{tabular}{cccccccc}
\toprule
\multicolumn{2}{c}{\multirow{2}{*}{\textbf{Test Set}}} & \multicolumn{2}{c}{\textbf{APCG-1}} & \multicolumn{2}{c}{\textbf{APCG-2}} & \multicolumn{2}{c}{\textbf{RLSCG(P)}} \\
\cmidrule(lr){3-4} \cmidrule(lr){5-6} \cmidrule(lr){7-8}
 & & \textbf{Iter.} & \textbf{Time(s)} & \textbf{Iter.} & \textbf{Time(s)} & \textbf{Iter.} & \textbf{Time(s)} \\
\midrule
\multirow{3}{*}{Gen\_1} & Small & 59.2 (+22\%) & 0.83 (+9\%) & 52.8 (+9\%) & \textbf{0.76} & \textbf{48.5} & 0.81 (+7\%) \\
 & Medium & 142.7 (+27\%) & 3.11 (+43\%) & 126.2 (+12\%) & 2.70 (+24\%) & \textbf{112.5} & \textbf{2.18} \\
 & Large & 264.2 (+22\%) & 6.28 (+33\%) & 239.8 (+10\%) & 5.81 (+23\%) & \textbf{217.3} & \textbf{4.74} \\
\midrule
\multirow{3}{*}{Gen\_2} & Small & 19.7 (+27\%) & 0.30 (+24\%) & \textbf{15.5} & \textbf{0.24} & 16.6 (+7\%) & 0.28 (+14\%) \\
 & Medium & 40.4 (+19\%) & 0.68 (+15\%) & 34.2 (+1\%) & \textbf{0.59} & \textbf{34.0} & 0.60 (+1\%) \\
 & Large & 85.1 (+23\%) & 1.58 (+25\%) & 71.5 (+4\%) & 1.38 (+9\%) & \textbf{69.0} & \textbf{1.27} \\
\midrule
\multirow{3}{*}{Gen\_3} & Small & 29.9 (+23\%) & 0.43 (+22\%) & \textbf{24.3} & \textbf{0.35} & 24.6 (+1\%) & 0.40 (+14\%) \\
 & Medium & 66.7 (+24\%) & 1.06 (+16\%) & 56.0 (+4\%) & \textbf{0.91} & \textbf{53.6} & 0.94 (+3\%) \\
 & Large & 130.6 (+27\%) & 2.27 (+20\%) & 112.1 (+9\%) & 2.00 (+6\%) & \textbf{102.5} & \textbf{1.89} \\
\midrule
\multirow{3}{*}{Gen\_4} & Small & 62.6 (+49\%) & 0.99 (+33\%) & 59.3 (+41\%) & 0.94 (+27\%) & \textbf{42.1} & \textbf{0.74} \\
 & Medium & 128.7 (+44\%) & 2.23 (+32\%) & 125.4 (+41\%) & 2.16 (+27\%) & \textbf{89.2} & \textbf{1.70} \\
 & Large & 208.5 (+38\%) & 3.88 (+27\%) & 205.0 (+36\%) & 3.83 (+25\%) & \textbf{151.1} & \textbf{3.06} \\
\midrule
\multicolumn{2}{c}{Falkenauer\_U} & 181.7 (+12\%) & 2.74 (+5\%) & 174.0 (+7\%) & \textbf{2.62} & \textbf{162.4} & 2.92 (+12\%) \\
\multicolumn{2}{c}{Falkenauer\_T} & 354.2 (+32\%) & 7.63 (+35\%) & 338.2 (+26\%) & 7.31 (+30\%) & \textbf{269.3} & \textbf{5.64} \\
\multicolumn{2}{c}{Schwerin} & 77.1 (+16\%) & 1.10 (+5\%) & 73.3 (+11\%) & \textbf{1.05} & \textbf{66.4} & 1.12 (+7\%) \\
\multicolumn{2}{c}{Waescher} & 101.5 (+63\%) & 3.92 (+163\%) & 96.2 (+55\%) & 3.49 (+134\%) & \textbf{62.2} & \textbf{1.49} \\
\multicolumn{2}{c}{Hard28} & 535.6 (+47\%) & 23.89 (+142\%) & 513.9 (+41\%) & 19.55 (+98\%) & \textbf{364.9} & \textbf{9.88} \\
\bottomrule
\end{tabular}%
}
\end{table}
% \begin{remark}
%     One could ask whether the learned policy reveals an interpretable parameter schedule that could itself serve as a rule. We examined the parameter values chosen across instances and found no recurring pattern. The policy's decisions depend on the evolving algorithmic state in ways that do not reduce to a simple heuristic, which aligns with our motivation for adopting a learning-based approach.
% \end{remark}
%; thus, we omit detailed parameter trajectory plots

\paragraph{Convergence dynamics.}
Figure~\ref{fig:convergence_plots} shows the normalized RMP objective trajectory for four representative sets. We report only smoothing variants, as the penalized RMP objective does not yield the true primal bound. RLSCG(S) shows a steeper initial descent than ASCG-1 and ASCG-2, indicating faster identification of high-quality columns, and maintains a lower optimality gap throughout, yielding better solutions under early termination.

This raises a related question: can the smoothing method also benefit from predicted duals?  We investigate this in the next subsection.

\begin{figure}[htbp]
    \centering
    \captionsetup[subfigure]{justification=centering,font=scriptsize}

    \begin{subfigure}[t]{0.315\textwidth}
        \centering
        \includegraphics[width=\linewidth]{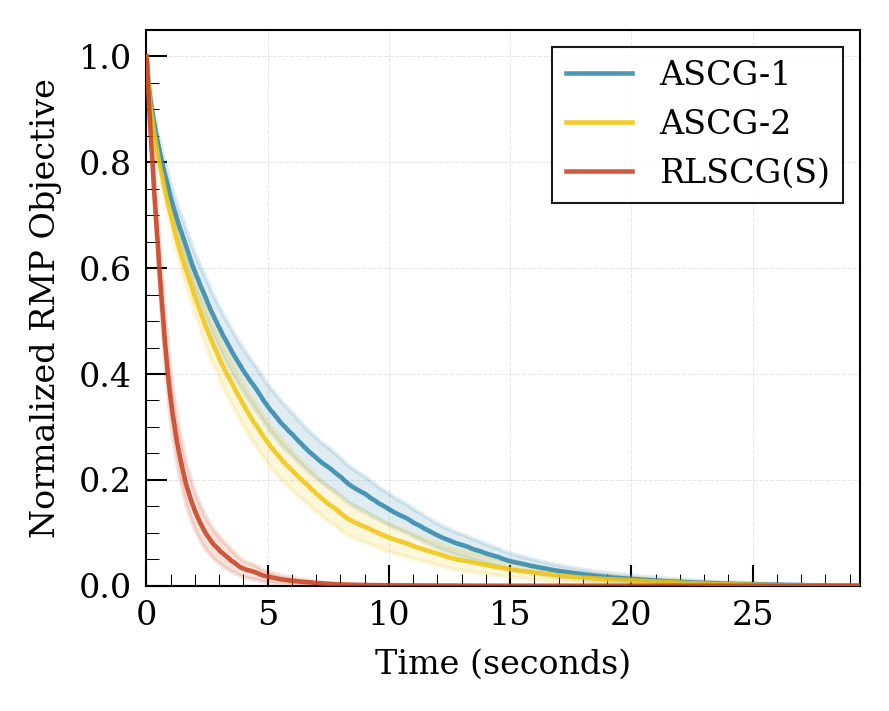}
        \caption{Gen\_1 Large}
        \label{fig:conv_gen1}
    \end{subfigure}
    \hspace{0.015\textwidth}
    \begin{subfigure}[t]{0.315\textwidth}
        \centering
        \includegraphics[width=\linewidth]{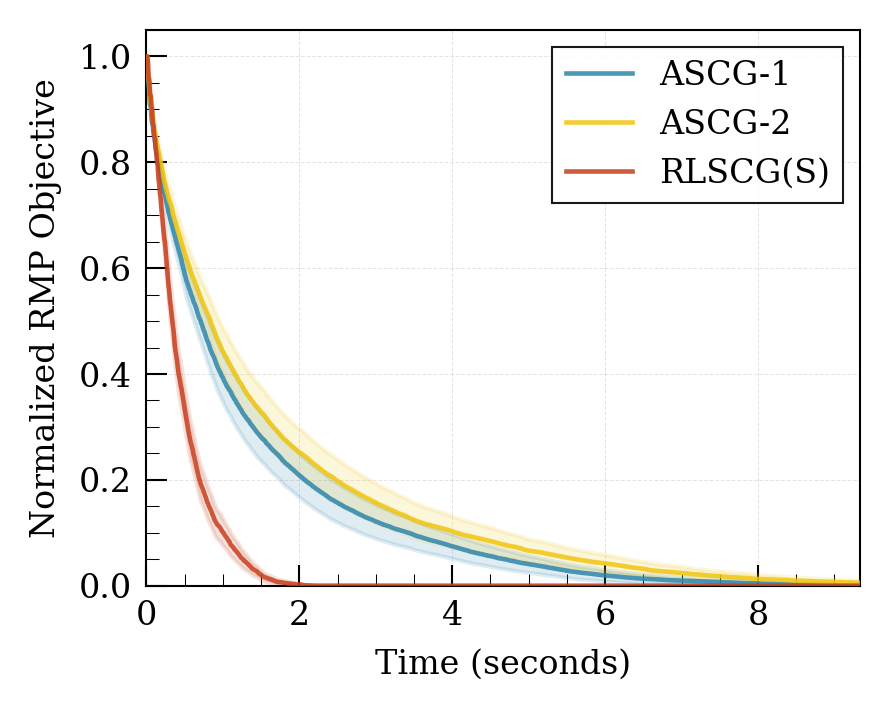}
        \caption{Gen\_2 Large}
        \label{fig:conv_gen2}
    \end{subfigure}

    \vspace{0.3em}

    \begin{subfigure}[t]{0.315\textwidth}
        \centering
        \includegraphics[width=\linewidth]{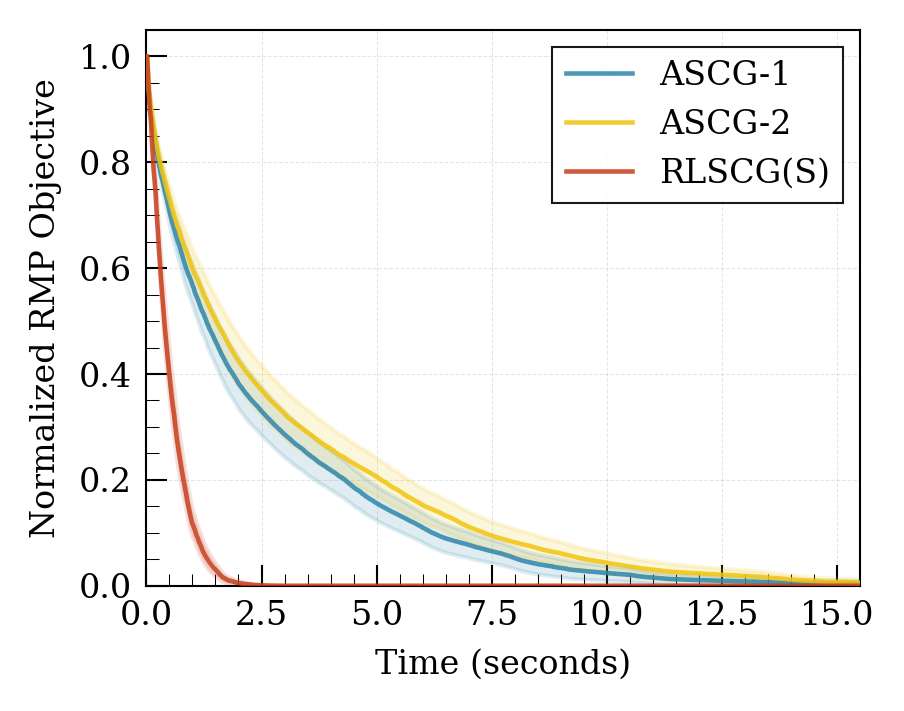}
        \caption{Gen\_3 Large}
        \label{fig:conv_gen3}
    \end{subfigure}
    \hspace{0.015\textwidth}
    \begin{subfigure}[t]{0.315\textwidth}
        \centering
        \includegraphics[width=\linewidth]{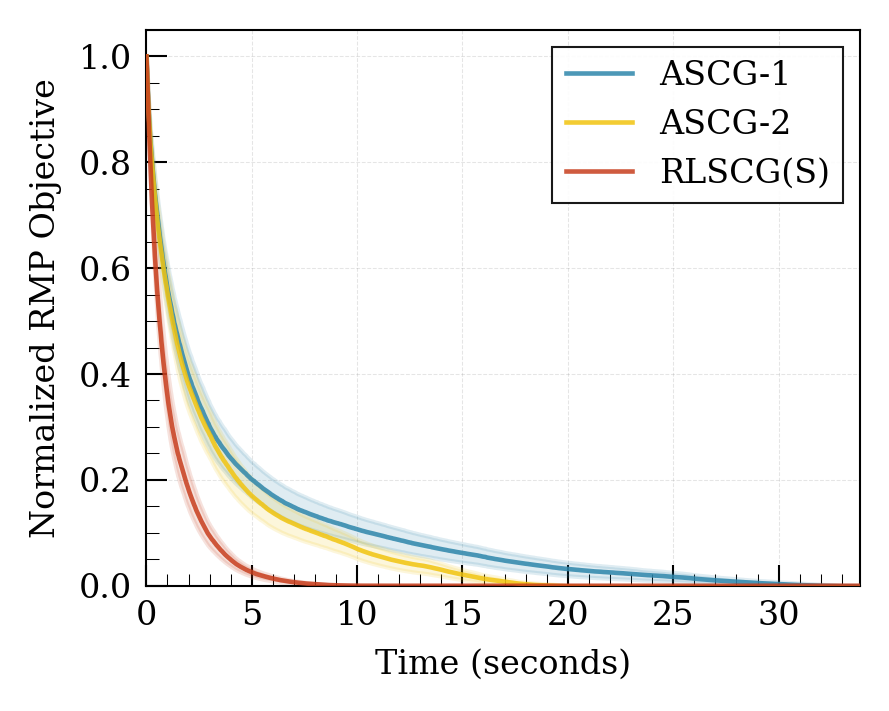}
        \caption{Falkenauer\_T}
        \label{fig:conv_gen4}
    \end{subfigure}

    \caption{(Color Online) Convergence plots for ASCG-1, ASCG-2, and RLSCG(S). The solid curves show mean normalized objective values, and the shaded areas show $\pm$1 standard deviation.}
    \label{fig:convergence_plots}
\end{figure}

\subsection{Impact of Predicted Duals as Reference Points}
\label{subsec:predicted_dual_smoothing}

Using fixed predicted duals may compromise convergence in smoothing. We therefore adopt the hybrid strategy of Proposition~\ref{prop:smooth-fix-hybrid}: the algorithm initially uses the predicted dual as the reference, then reverts to standard updates (e.g., the incumbent best dual). As dual prediction is not our focus, we refer to \cite{kraul2023machine,shen2024adaptive} for details.

We augment ASCG-1, ASCG-2, and RLSCG(S) with two hybrid reference variants:
\begin{itemize}
    \item \textit{Static Predicted Reference (SPR)}: The predicted dual serves as the reference solely at the first iteration before the algorithm switches to standard updates.
    \item \textit{Adaptive Predicted Reference (APR)}: The predicted dual is retained until stagnation (defined as a relative objective improvement $<10^{-3}$ for 5 consecutive iterations), after which standard updates resume.
\end{itemize}

% \edit{Table~\ref{tab:impact_time} reports computation times for large instances; detailed boxplots are in Appendix~\ref{appendix: Predicted Duals in Smoothing}. We highlight three findings. First, predicted duals can substantially improve rule-based methods, reducing computation times by up to 78\% (e.g., ASCG-2 on Gen\_3), though these gains are not universal across all test sets. 
% Second, the impact on RLSCG(S) is mixed: while it benefits on some instances (e.g., Gen\_1 and Gen\_3), its performance slightly degrades on Gen\_2. This suggests that the learned policy is already effective, leaving little room for further improvement by leveraging predicted duals.
% Third, no predicted-reference strategy dominates uniformly, implying that the optimal switching point is highly instance-dependent.}

% \edit{
Table~\ref{tab:impact_time} reports computation times for large instances; detailed boxplots are in Appendix~\ref{appendix: Predicted Duals in Smoothing}. We highlight three findings. First, predicted duals can substantially improve rule-based methods, reducing computation times by up to 78\% (e.g., ASCG-2 on Gen\_3), though these gains are not universal across all test sets. Second, the impact on RLSCG(S) is mixed: predicted references help on Gen\_1 and Gen\_3 (up to -19\%) but yield little to no improvement on Gen\_2 and Gen\_4. This suggests that the learned policy is already effective, leaving little room for further improvement by leveraging predicted duals. Third, no predicted-reference strategy dominates uniformly, implying that the optimal switching point is highly instance-dependent.
% }

\begin{table}[htbp]
\centering
\scriptsize
\caption{Impact of predicted-dual reference strategies on computation time (s) for large-scale instances. 
Percentages indicate the relative change in computation time compared with the corresponding \textit{Standard} variant.}
\label{tab:impact_time}
\begin{tabular}{ccc c c}
\toprule
\textbf{Test Set}& \textbf{Method} & \textbf{Standard}& \textbf{Static Predicted Reference (SPR)}& \textbf{Adaptive Predicted Reference (APR)}\\
\midrule
% --- Gen 1 Large ---
\multirow{3}{*}{\shortstack[l]{Gen\_1\\Large}} & ASCG-1& 45.84 & 45.83 (0\%) & \textbf{11.23} \textbf{(-76\%)} \\
 & ASCG-2& 12.84 & \textbf{5.52} \textbf{(-57\%)} & 5.71 (-56\%) \\
 & RLSCG(S)& 4.45 & 4.35 (-2\%) & \textbf{3.59} \textbf{(-19\%)} \\
\midrule
% --- Gen 2 Large ---
\multirow{3}{*}{\shortstack[l]{Gen\_2\\Large}} & ASCG-1& 5.04 & 5.03 (0\%) & \textbf{5.00} \textbf{(-1\%)} \\
 & ASCG-2& 4.15 & 4.15 (0\%) & \textbf{3.09} \textbf{(-26\%)} \\
 & RLSCG(S)& 1.15 & \textbf{1.11} \textbf{(-3\%)} & 1.40 (+22\%) \\
\midrule
% --- Gen 3 Large ---
\multirow{3}{*}{\shortstack[l]{Gen\_3\\Large}} & ASCG-1& 7.83 & 7.87 (+1\%) & \textbf{2.76} \textbf{(-65\%)} \\
 & ASCG-2& 7.51 & 1.82 (-76\%) & \textbf{1.67} \textbf{(-78\%)} \\
 & RLSCG(S)& 1.71 & 1.62 (-5\%) & \textbf{1.38} \textbf{(-19\%)} \\
\midrule
% --- Gen 4 Large ---
\multirow{3}{*}{\shortstack[l]{Gen\_4\\Large}} & ASCG-1& 3.85 & 3.84 (0\%) & \textbf{3.78} \textbf{(-2\%)} \\
 & ASCG-2& 4.07 & \textbf{3.72} \textbf{(-9\%)} & 3.83 (-6\%) \\
 & RLSCG(S)& \textbf{2.70} & \textbf{2.70} \textbf{(0\%)} & 2.77 (+3\%) \\
\bottomrule
\end{tabular}
\end{table}

%\edit{A natural question is whether a good reference point alone closes the gap with the learned policy. On Gen\_1 Large (Figure~\ref{fig:ws_impact_gen1}), it substantially narrows it: predicted ASCG-2 (5.52s) becomes competitive with standard RLSCG(S) (4.71s). However, on Gen\_2 Large (Figure~\ref{fig:ws_impact_gen2}), the gap remains pronounced: RLSCG(S) (1.19s) is nearly 3.5 times faster than predicted ASCG-2 (4.15s). This highlights that while predicted references improve rule-based methods, the learned policy provides more consistent gains across diverse instances.}

% \edit{
A natural question is whether a good reference point alone closes the gap with the learned policy. On Gen\_1 Large (Figure~\ref{fig:ws_impact_gen1}), predicted ASCG-2 narrows the gap to within ${\sim}24\%$ of standard RLSCG(S). However, on Gen\_2 Large (Figure~\ref{fig:ws_impact_gen2}), the gap remains pronounced: RLSCG(S) is nearly 3.6 times faster than predicted ASCG-2. This highlights that while predicted references improve rule-based methods, the learned policy provides more consistent gains across diverse instances.
% }

\begin{figure}[htbp]

    \centering
   
    \begin{subfigure}[b]{0.265\textwidth}
        \centering

        \includegraphics[width=\textwidth]{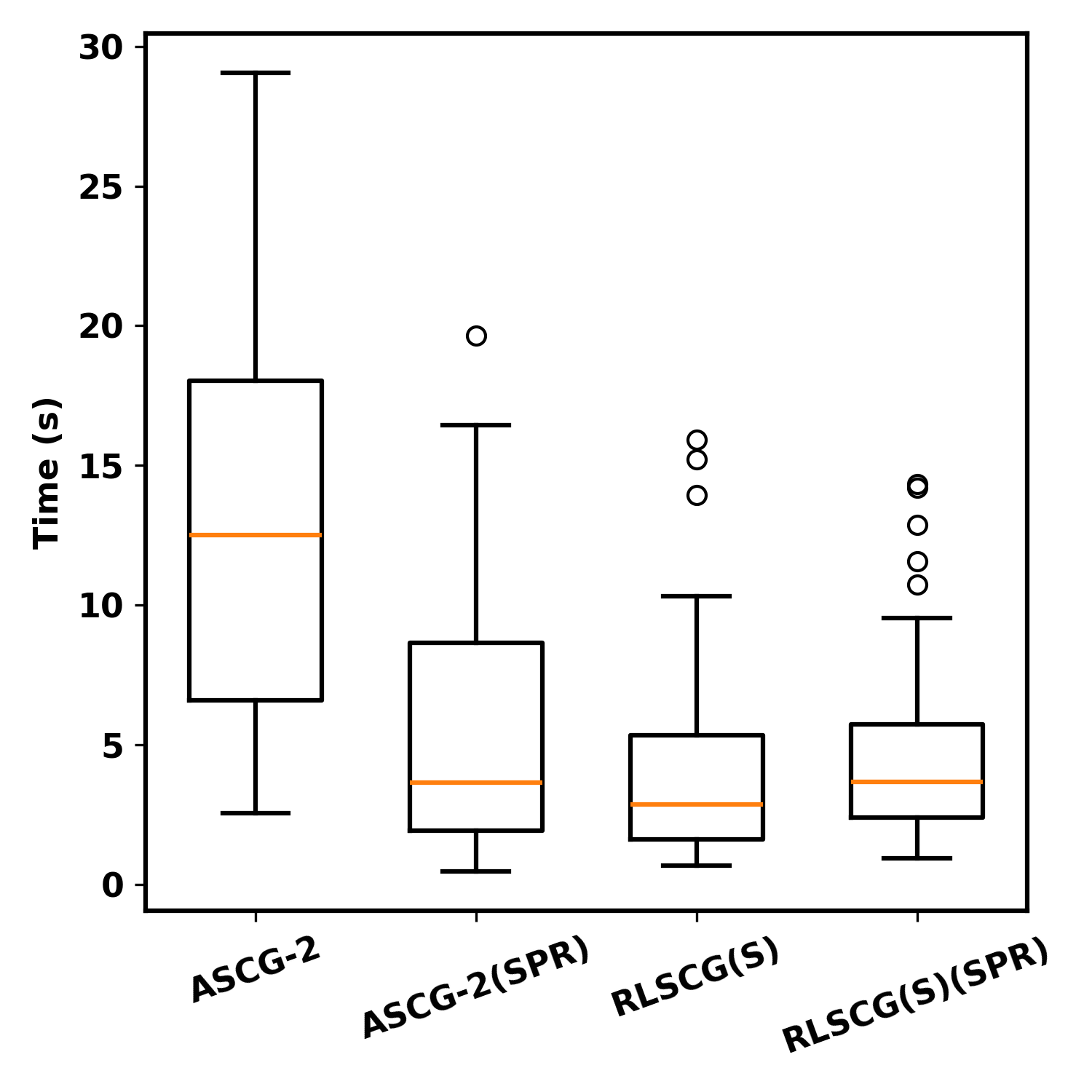}
        \caption{Gen\_1 Large}
        \label{fig:ws_impact_gen1}
    \end{subfigure}
    % \hfill 
    \begin{subfigure}[b]{0.265\textwidth}
        \centering
        \includegraphics[width=\textwidth]{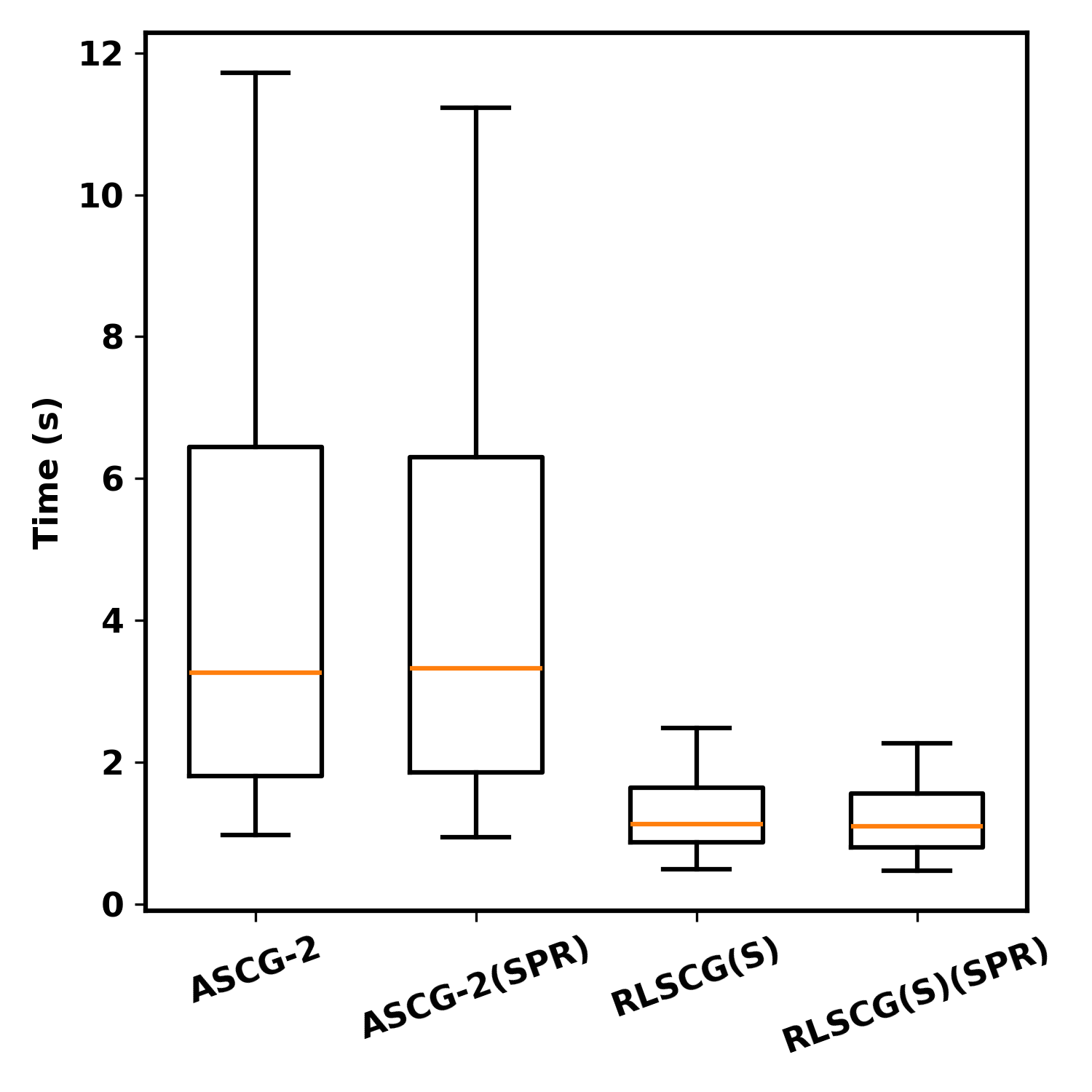}
        \caption{Gen\_2 Large}
        \label{fig:ws_impact_gen2}
    \end{subfigure}
    \caption{Comparison of computation time distributions between ASCG-2 and RLSCG(S) under Standard and Static Predicted Reference strategies. 
    % (a) In \textit{Gen\_1 Large}, enhancing ASCG with static predicted reference yields performance comparable to standard RLSCG(S). (b) In \textit{Gen\_2 Large}, RLSCG(S) maintains a significant advantage regardless of using predicted dual.
    }
    \label{fig:ws_impact}
\end{figure}

\subsection{Comparison with Learning-Based Column Selection}
\label{subsec:comparison_rlcg}

Finally, we benchmark RLSCG against \textit{RLCG} \citep{chi2022deep}, an RL-based column selection policy. We use RLSCG(S) as our representative variant given its stronger performance across the test sets. 
As shown in Table~\ref{tab:comparison_rlcg}, while RLCG sometimes requires fewer or comparable iterations, RLSCG(S) achieves lower computation time.

\begin{table}[htbp]
\centering
\tiny
\caption{Comparison between RLCG and RLSCG(S). Percentages indicate the relative change of RLSCG(S) compared to RLCG. The best performance in each metric is highlighted in bold.}
\label{tab:comparison_rlcg}
\resizebox{\textwidth}{!}{%
\begin{tabular}{ll cc cc cc cc}
\toprule
\multicolumn{2}{c}{\multirow{2}{*}{\textbf{Test Set}}}& \multicolumn{2}{c}{\textbf{Iterations}} & \multicolumn{2}{c}{\textbf{Total Time (s)}} & \multicolumn{2}{c}{\textbf{Pricing Time (s)}} & \multicolumn{2}{c}{\textbf{RMP Time (s)}} \\
\cmidrule(lr){3-4} \cmidrule(lr){5-6} \cmidrule(lr){7-8} \cmidrule(lr){9-10}
 & & \textbf{RLCG} & \textbf{RLSCG(S)}& \textbf{RLCG} & \textbf{RLSCG(S)}& \textbf{RLCG} & \textbf{RLSCG(S)}& \textbf{RLCG} & \textbf{RLSCG(S)}\\
\midrule
\multirow{3}{*}{Gen\_1} & Small & 39.1 & \textbf{39.0} (-1\%) & 7.44 & \textbf{0.64} (-91\%) & 6.78 & \textbf{0.27} & 0.46 & \textbf{0.26} \\
 & Medium & 96.1 & \textbf{80.0} (-17\%) & 21.96 & \textbf{1.52} (-93\%) & 20.02 & \textbf{0.73} & 1.19 & \textbf{0.56} \\
 & Large & 198.3 & \textbf{192.3} (-3\%) & 56.88 & \textbf{4.35} (-92\%) & 51.12 & \textbf{2.23} & 2.61 & \textbf{1.45} \\
\midrule
\multirow{3}{*}{Gen\_2} & Small & \textbf{8.3} & 12.4 (+49\%) & 1.41 & \textbf{0.20} (-86\%) & 1.29 & \textbf{0.09} & \textbf{0.09} & \textbf{0.09} \\
 & Medium & \textbf{20.6} & 22.6 (+10\%) & 3.44 & \textbf{0.38} (-89\%) & 3.09 & \textbf{0.16} & 0.22 & \textbf{0.16} \\
 & Large & \textbf{55.0} & 63.6 (+16\%) & 9.97 & \textbf{1.11} (-89\%) & 8.72 & \textbf{0.47} & 0.61 & \textbf{0.44} \\
\midrule
\multirow{3}{*}{Gen\_3} & Small & \textbf{15.3} & 17.7 (+16\%) & 2.51 & \textbf{0.29} (-89\%) & 2.30 & \textbf{0.12} & 0.15 & \textbf{0.12} \\
 & Medium & 39.5 & \textbf{37.4} (-5\%) & 6.96 & \textbf{0.64} (-91\%) & 6.33 & \textbf{0.27} & 0.39 & \textbf{0.26} \\
 & Large & \textbf{89.6} & 89.9 (+1\%) & 17.24 & \textbf{1.62} (-91\%) & 15.23 & \textbf{0.69} & 0.96 & \textbf{0.65} \\
\midrule
\multirow{3}{*}{Gen\_4} & Small & 38.8 & \textbf{37.7} (-3\%) & 9.83 & \textbf{0.66} (-93\%) & 9.17 & \textbf{0.31} & 0.47 & \textbf{0.25} \\
 & Medium & \textbf{80.7} & 80.9 (+1\%) & 28.35 & \textbf{1.59} (-94\%) & 26.65 & \textbf{0.75} & 1.06 & \textbf{0.59} \\
 & Large & 139.4 & \textbf{138.3} (-1\%) & 78.47 & \textbf{2.70} (-97\%) & 74.45 & \textbf{1.24} & 2.08 & \textbf{1.00} \\
\midrule
\multicolumn{2}{c}{Falkenauer\_U}& 156.5 & \textbf{140.7} (-10\%) & 40.40 & \textbf{2.44} (-94\%) & 37.36 & \textbf{1.02} & 1.80 & \textbf{0.99} \\
\multicolumn{2}{c}{Falkenauer\_T}& \textbf{237.1} & 259.9 (+10\%) & 65.06 & \textbf{5.42} (-92\%) & 57.37 & \textbf{2.44} & 2.99 & \textbf{2.00} \\
\multicolumn{2}{c}{Schwerin}& \textbf{44.4} & 45.3 (+2\%) & 153.36 & \textbf{0.75} (-99\%) & 152.25 & \textbf{0.32} & 0.82 & \textbf{0.31} \\
\multicolumn{2}{c}{Waescher}& 60.7 & \textbf{60.5} (-1\%) & 16.58 & \textbf{1.41} (-91\%) & 15.50 & \textbf{0.83} & 0.73 & \textbf{0.41} \\
\multicolumn{2}{c}{Hard28}& 384.1 & \textbf{377.1} (-2\%) & 146.22 & \textbf{9.84} (-93\%) & 130.86 & \textbf{5.03} & 5.44 & \textbf{3.25} \\
\bottomrule
\end{tabular}
}
\end{table}

% Results are reported in Table~\ref{tab:comparison_rlcg}.
% While RLCG generally requires fewer or comparable iterations, RLSCG(S) achieves lower total computation time on the majority of instances. 
The iteration advantage of RLCG is expected: its policy selects the most beneficial column from a candidate pool. However, generating this pool incurs pricing overhead. As shown in Table~\ref{tab:comparison_rlcg}, the pricing time for RLCG is consistently orders of magnitude higher than that of RLSCG(S), driving its substantially longer total runtimes. 
Fundamentally, the distinction lies in how learning induces pricing distortion: RLCG learns post-pricing column selection, whereas RLSCG learns pre-pricing dual deviation. Our results suggest that for large-scale CSP instances, the latter is more efficient; whether this advantage persists for harder pricing oracles present in other problem classes remains an open question.

\section{Conclusion}\label{sec:conclusion}

In this paper, we study stabilization in column generation. Our first contribution is a unified view of smoothing and penalization through \emph{pricing distortion}, the deviation between the stabilized dual and the RMP dual induced by reference points and stabilization parameters. Within this framework, we formalize mispricing, derive parameter bounds that prevent it, and show that these bounds depend on global quantities expensive to evaluate online. This, in turn, motivates our second contribution, RLSCG, a reinforcement learning-guided framework that adaptively selects stabilization parameters from structural and algorithmic information encoded by a graph neural network. 
On the Cutting Stock Problem, RLSCG substantially reduces iteration count and computation time on most instances relative to traditional CG, rule-based stabilization, and learning-based column selection, with the largest gains on large-scale instances. 
% \edit{
Notably, rule-based stabilization can in fact slow convergence on these instances, highlighting the importance of learned parameter control.
% } 
The comparison with RLCG suggests that learning stabilization outperforms column selection for this problem class.

Future work includes testing RLSCG on problems with harder pricing subproblems (e.g., vehicle routing, crew scheduling) and extending it to multi-column addition by jointly controlling stabilization strength and the number of columns per iteration.

% \section*{Acknowledgments}
% This was was supported in part by......

%Bibliography
% \bibliographystyle{unsrt}  
% \bibliography{references}  

% \begingroup
% \setstretch{1.5} % 
% \small % 
% \bibliographystyle{unsrt}  
% \bibliography{references}
% \endgroup

\bibliographystyle{unsrt}  
\bibliography{references}

\newpage

% Acknowledgments here
% \ACKNOWLEDGMENT{%
% Enter the text of acknowledgments here
% }% Leave this (end of acknowledgment)

% Appendix here
% Options are (1) APPENDIX (with or without general title) or 
%             (2) APPENDICES (if it has more than one unrelated sections)
% Outcomment the appropriate case if necessary
%
% \begin{APPENDIX}{<Title of the Appendix>}
% \end{APPENDIX}
%
%   or 
%
% \begin{APPENDICES}

% \newpage
\appendix

\section{Supplementary Analysis}

\subsection{Penalization Formulations}\label{appendix:alternative_penalization}

In Section \ref{subsec:prelim} and \ref{subsec:stabilization framework}, we detail the single-point polyhedral penalization. Here, we present two other common penalization formulations and their corresponding operators within our generalized stabilization framework.

\textit{Interval Polyhedral Penalty.}
Instead of a single reference point, one may define a trust region box.

\begin{itemize}
    \item Parameters: $\theta_t = \{\varepsilon_t^-, \varepsilon_t^+,\xi\}$ with $\varepsilon_t^-, \varepsilon_t^+, \xi \ge 0$, and $\delta_t^- = \pi^{\mathrm{ref}}_t-\xi$, $\delta_t^+ = \pi^{\mathrm{ref}}_t+\xi$, where $\xi$ represents the box size.
    \item Operator: Let $(\pi_t^{\mathrm{sep}}, w_t^-, w_t^+)$ be an optimal solution to the following problem

% \begin{equation}\label{eq:f_pen2}
% \begin{aligned}
% \max_{\pi\in\mathbb{R}^m,\, w^-\in\mathbb{R}^m_+,\, w^+\in\mathbb{R}^m_+}\quad
% & b^\top\pi - (\varepsilon_t^-)^\top w^- - (\varepsilon_t^+)^\top w^+ \\
% \text{s.t.}\quad
% & A_R^\top\pi \le c_R, \\
% & \pi - w^+ \le \pi_t^{\mathrm{ref}}+\xi, \\
% & -\pi - w^- \le -(\pi_t^{\mathrm{ref}}-\xi) .
% \end{aligned}
% \end{equation}
\begin{equation}\label{eq:f_pen2}
\begin{aligned}
\max_{\pi\in\mathbb{R}^m,\; w^-\in\mathbb{R}^m_+,\; w^+\in\mathbb{R}^m_+}\quad
& b^\top\pi - (\varepsilon_t^-)^\top w^- - (\varepsilon_t^+)^\top w^+ \\
\text{s.t.}\quad
& A_R^\top\pi \le c_R,\;
\pi - w^+ \le \pi_t^{\mathrm{ref}}+\xi,\;
-\pi - w^- \le -(\pi_t^{\mathrm{ref}}-\xi).
\end{aligned}
\end{equation}
The stabilized dual vector is given by $f_t^{\mathrm{pen}}(\pi_t^{\mathrm{ref}}, \theta_t; \mathcal{I}_t) = \pi_t^{\mathrm{sep}}$.
\end{itemize}
This creates a three-piece linear penalization function: deviations within $[\pi^{\mathrm{ref}}_t-\xi, \pi^{\mathrm{ref}}_t+\xi]$ incur no penalty, while deviations outside are penalized linearly.

\textit{Euclidean–Proximal Penalty.}
This formulation uses a quadratic penalty term to penalize the Euclidean distance between the dual solution and the reference center $\pi^{\mathrm{center}}_t$.
\begin{itemize}
    \item Parameters: $\theta_t = \{\varepsilon\}$ with $\varepsilon > 0$, and $\pi_t^{\mathrm{ref}}$ is the box center.
    \item Operator: 
Let $\pi_t^{\mathrm{sep}}$ be the unique optimal solution to:
\begin{equation}\label{eq:f_pen3}
\begin{aligned}
\max_{\pi\in\mathbb{R}^m}\quad
& b^\top \pi - \varepsilon\|\pi - \pi_t^{\mathrm{ref}}\|^2_2 \\
\text{s.t.}\quad
& A_R^\top\pi \le c_R.
\end{aligned}
\end{equation}
The stabilized dual vector is given by $f_t^{\mathrm{pen}}(\pi_t^{\mathrm{ref}}, \theta_t; \mathcal{I}_t) = \pi_t^{\mathrm{sep}}$.
\end{itemize}

\subsection{Detailed Calculations for the Toy Example} \label{app:toy_details}

We outline the column generation steps for the illustrative example in Section~\ref{sec:pricing_distortion}. At iteration $t$, the pricing problem is $\min_{q \in \{A,B,C\}} r_q(\pi)$, where $r_q(\pi) = c_q - a_q^\top \pi$. 
Note that for the weak and strong distortion regimes, the separation points $\pi_t^{\mathrm{sep}}$ are intentionally selected to demonstrate the distinct algorithmic behaviors under different pricing strategies.

\paragraph{Standard Pricing (Figure~\ref{fig:toy_pricing_steps}a).}
Standard pricing always selects the inequality with the most negative reduced cost at $\pi_t^{\mathrm{out}}$. 
% This is locally optimal for the current restricted problem, but it may introduce intermediate inequalities that are not needed to define the final optimal vertex.
% Pricing is performed at the RMP optimum, $\pi_t^{\mathrm{out}}$.
\begin{itemize}
    \item \textit{Iteration $t=0$:} $\pi_0^{\mathrm{out}} = (5,5)$. Pricing yields Cut $A$.
    \item \textit{Iteration $t=1$:} $\pi_1^{\mathrm{out}} = (5, 1)$. Pricing yields Cut $B$.
    \item \textit{Iteration $t=2$:} $\pi_2^{\mathrm{out}} = (11/3, 7/3)$. Pricing yields Cut $C$.
    \item \textit{Iteration $t=3$:} The RMP optimum reaches $\pi^\star = (3.5, 2.5)$. All remaining reduced costs are non-negative. Convergence.
\end{itemize}

\paragraph{Weak Pricing Distortion (Figure~\ref{fig:toy_pricing_steps}b).}
Weak distortion prices at a stabilized dual rather than at $\pi_t^{\mathrm{out}}$. The selected cut is no longer the greedy one for the current RMP, but it still separates $\pi_t^{\mathrm{out}}$.
% Pricing is performed at a stabilized dual $\pi_t^{\mathrm{sep}}$.
\begin{itemize}
    \item \textit{Iteration $t=0$:} Using $\pi_0^{\mathrm{sep}} = \pi_0^{\mathrm{out}} = (5,5)$, pricing yields Cut $A$.
    \item \textit{Iteration $t=1$:} $\pi_1^{\mathrm{out}} = (5,1)$. Pricing at the smoothed dual $\pi_1^{\mathrm{sep}} = 0.5\pi_0^{\mathrm{sep}} + 0.5\pi_1^{\mathrm{out}} = (5, 3)$ yields Cut $C$. Since $r_C(\pi_1^{\mathrm{out}}) < 0$, this cut separates $\pi_1^{\mathrm{out}}$.
    \item \textit{Iteration $t=2$:} The RMP optimum reaches $\pi^\star = (3.5, 2.5)$. Convergence.
\end{itemize}

\paragraph{Strong Pricing Distortion (Figure~\ref{fig:toy_pricing_steps}c).}
Strong distortion generate a cut that is not separating for $\pi_t^{\mathrm{out}}$, producing a mispricing event and no immediate objective improvement. Nevertheless, the added inequality may reshape the restricted feasible region in a way that improves the subsequent search path.
% Aggressive exploration is induced by selecting a fixed center initially.
\begin{itemize}
    \item \textit{Iteration $t=0$:} $\pi_0^{\mathrm{out}} = (5,5)$. Pricing at the selected point $\pi_0^{\mathrm{sep}} = (3.6, 2.4)$ yields Cut $C$. However, $r_C(\pi_0^{\mathrm{out}}) > 0$, meaning Cut $C$ does not separate $\pi_0^{\mathrm{out}}$. A \textit{mispricing} event occurs, as illustrated in Panel $t=0$ of Figure~\ref{fig:toy_pricing_steps}(c).
    \item \textit{Iteration $t=1$:} $\pi_1^{\mathrm{out}} = (5,5)$. Pricing at the selected point $\pi_1^{\mathrm{sep}} = 0.5\pi_0^{\mathrm{sep}} + 0.5\pi_1^{\mathrm{out}} = (4.3, 3.7)$. Pricing at $\pi_1^{\mathrm{sep}}$ yields Cut $A$. This is a valid separating cut since $r_A(\pi_1^{\mathrm{out}}) < 0$.
    \item \textit{Iteration $t=2$:} The RMP optimum reaches $\pi^\star = (3.5, 2.5)$. Convergence.
\end{itemize}

\subsection{Restricted Equivalence of Smoothing and Euclidean-Proximal Penalization}
\label{sec:eq-smooth-prox}

We establish a condition under which the Euclidean-proximal penalization operator \eqref{eq:f_pen3} is equivalent to the smoothing operator \eqref{eq:f_smooth}.

% Let $D_R=\{\pi\in\mathbb{R}^m: A_R^\top \pi\le c_R\}$ be the nonempty, closed, and convex dual feasible set of the RMP. Fix $\pi_t^{\mathrm{ref}}\in D_R$, and let $\pi_t^{\mathrm{out}}\in\arg\max_{\pi\in D_R} b^\top\pi$ be an optimal dual solution to the RMP.
% % \footnote{If the RMP dual solution is not unique, $\pi_t^{\mathrm{out}}$ can be any selection from the set of optimal solutions.}
% Define the direction vector $d=\pi_t^{\mathrm{out}}-\pi_t^{\mathrm{ref}}$. We assume $d\neq 0$ and $b^\top d > 0$, which implies that $\pi_t^{\mathrm{out}}$ has a strictly better RMP objective value than $\pi_t^{\mathrm{ref}}$.
Let
$
D_R=\{\pi\in\mathbb{R}^m: A_R^\top \pi\le c_R\}
$
be the nonempty, closed, convex dual feasible set of the RMP. Fix $\pi_t^{\mathrm{ref}}\in D_R$, and let
$
\pi_t^{\mathrm{out}}\in\arg\max_{\pi\in D_R} b^\top\pi
$
be an optimal RMP dual solution. Define
$
d=\pi_t^{\mathrm{out}}-\pi_t^{\mathrm{ref}},
$
and assume $d\neq 0$ and $b^\top d>0$.
We further impose the geometric condition that $b$ is positively collinear with $d$, i.e.,
\begin{equation}
    b = \kappa\, d \quad \text{for some scalar } \kappa > 0.
    \label{eq:geom-condition}
\end{equation}
% This condition implies that the RMP objective $b^\top \pi$ increases most steeply exactly in the direction from the reference point $\pi_t^{\mathrm{ref}}$ to the RMP optimum $\pi_t^{\mathrm{out}}$. 
Geometrically, as shown in Figure~\ref{fig:prox-geom-equiv}, when $b$ collines positively with $d$, the objective $b^\top \pi$ and the penalty $-\varepsilon\|\pi - \pi_t^{\mathrm{ref}}\|_2^2$ create pulls along the same axis. Consequently, the entire path of Euclidean-proximal solutions $\{\pi_t^{\mathrm{pen}}(\varepsilon)\}_{\varepsilon>0}$ (generated by varying $\varepsilon$) lies exactly on the smoothing segment $[\pi_t^{\mathrm{ref}}, \pi_t^{\mathrm{out}}]$.  The resulting equivalence is stated below.
% The task of maximizing the strongly concave function $b^\top \pi - \varepsilon\|\pi - \pi_t^{\mathrm{ref}}\|_2^2$ over the entire convex set $D_R$ thus reduces to a one-dimensional optimization problem along the line segment $[\pi_t^{\mathrm{ref}}, \pi_t^{\mathrm{out}}]$.

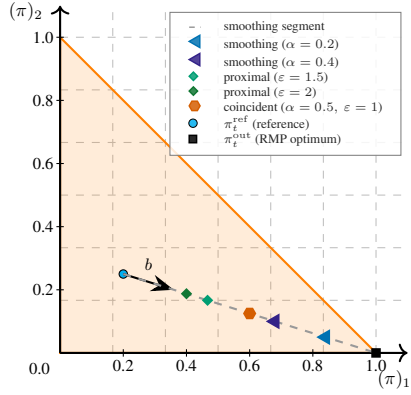
\begin{figure}[htbp]
    \centering
    % ==== Color-blind friendly palette ====
    \definecolor{colorSmoothA}{HTML}{0072B2}   % Blue   (alpha=0.2)
    \definecolor{colorSmoothB}{HTML}{332288}   % Indigo (alpha=0.4)
    \definecolor{colorProxA}{HTML}{009E73}     % Green  (eps=1.5)
    \definecolor{colorProxB}{HTML}{117733}     % Darker Green (eps=2)
    \definecolor{colorEquiv}{HTML}{D55E00}     % Vermillion (coincident)

    % ==== Styles ====
    \tikzstyle{gridstyle}  = [dashed, color=gray!50, thin]
    % smoothing points: triangles
    \tikzstyle{smoothnode} = [draw, regular polygon, regular polygon sides=3, inner sep=1.5pt, thick, scale=0.9, rotate=90]
    % proximal points: diamonds
    \tikzstyle{proxnode}   = [draw, diamond, inner sep=1.5pt, thick, scale=0.8]
    % coincident point: hexagon
    \tikzstyle{equivnode}  = [draw, regular polygon, regular polygon sides=6, inner sep=1.8pt, thick, scale=1.0]

    % \begin{tikzpicture}[x=6cm, y=6cm]
    \begin{tikzpicture}[x=6cm, y=6cm, scale=0.7, transform shape]
        % 1) grid
        \draw[gridstyle] (0,0) grid (1.1, 1.1);

        % 2) feasible set D_R (triangle)
        \fill[color=orange, opacity=0.15] (0,0) -- (1,0) -- (0,1) -- cycle;
        \draw[color=orange, thick] (0,0) -- (1,0) -- (0,1) -- cycle;

        % 3) axes
        \draw[->, thick] (0,0) -- (1.1, 0) node[below left, yshift=-2mm, xshift=2mm] {$(\pi)_1$};
        \draw[->, thick] (0,0) -- (0, 1.1) node[below left, xshift=-2mm, yshift=2mm] {$(\pi)_2$};
        \foreach \i in {0.2, 0.4, 0.6, 0.8, 1.0} {
            \draw (\i, -2pt) -- (\i, 2pt) node[below, font=\small, yshift=-2pt] {\i};
            \draw (-2pt, \i) -- (2pt, \i) node[left,  font=\small, xshift=-2pt] {\i};
        }
        \node[below left, font=\small, xshift=-2pt, yshift=-2pt] at (0,0) {0.0};

        % 4) key points
        \coordinate (piref) at (0.20, 0.25);
        \coordinate (piout) at (1.00, 0.00);

        % 5) reference and b (with b || d)
        \node[draw, circle, fill=cyan, inner sep=1.6pt] at (piref) {};
        % choose an endpoint to avoid overlap
        \draw[->, thick, -{Stealth[length=3mm]}] (piref) -- node[above, font=\small, yshift=1mm] {$b$} (0.36, 0.20);

        % smoothing segment
        \draw[dashed, thick, color=gray!80] (piref) -- (piout);
        % RMP optimum
        \node[draw, rectangle, inner sep=2pt, thick, fill=black] at (piout) {};

        % ===== Points on the segment:  pi(t) = (0.2 + 0.8 t , 0.25 - 0.25 t) =====
        % Proximal eps=1.5 -> t=1/(2*1.5)=1/3 -> pi=(0.4667, 0.1667)
        \coordinate (p_proxA) at (0.4667, 0.1667);
        \node[proxnode, fill=colorProxA, draw=colorProxA] at (p_proxA) {};
        % \node[font=\small, anchor=south west] at (0.472, 0.168) ;

        % Proximal eps=2 -> t=1/4 -> pi=(0.4, 0.1875)
        \coordinate (p_proxB) at (0.40, 0.1875);
        \node[proxnode, fill=colorProxB, draw=colorProxB] at (p_proxB) {};
        % \node[font=\small, anchor=south west] at (0.405, 0.189) {$\varepsilon=2,\ t=\tfrac{1}{4}$};

        % Coincident: alpha=0.5 <-> eps=1 -> t=0.5 -> pi=(0.6, 0.125)
        \coordinate (p_coinc) at (0.60, 0.125);
        \node[equivnode, fill=colorEquiv, draw=colorEquiv] at (p_coinc) {};
        % \node[font=\small, anchor=south west] at (0.605, 0.126) {$\alpha=0.5,\ \varepsilon=1,\ t=0.5$};

        % Smoothing alpha=0.4 -> t=0.6 -> pi=(0.68, 0.10)
        \coordinate (p_smoothB) at (0.68, 0.10);
        \node[smoothnode, fill=colorSmoothB, draw=colorSmoothB] at (p_smoothB) {};
        % \node[font=\small, anchor=south west] at (0.685, 0.101) {$\alpha=0.4,\ t=0.6$};

        % Smoothing alpha=0.2 -> t=0.8 -> pi=(0.84, 0.05)
        \coordinate (p_smoothA) at (0.84, 0.05);
        \node[smoothnode, fill=colorSmoothA, draw=colorSmoothA] at (p_smoothA) {};
        % \node[font=\small, anchor=south west] at (0.845, 0.051) {$\alpha=0.2,\ t=0.8$};

        % Legend (top-right)
        \node[anchor=north east, fill=white, fill opacity=0.85, draw=gray, inner sep=4pt] at (1.08, 1.08) {%
            % \small
            \scriptsize
            \begin{tabular}{ll}
                \tikz[baseline=-0.5ex]{\draw[dashed, thick, color=gray!80] (0,0) -- (0.22cm, 0);} & smoothing segment \\
                \tikz{\node[smoothnode, fill=colorSmoothA, draw=colorSmoothA] {};} & smoothing ($\alpha=0.2$) \\
                \tikz{\node[smoothnode, fill=colorSmoothB, draw=colorSmoothB] {};} & smoothing ($\alpha=0.4$) \\
                \tikz{\node[proxnode,   fill=colorProxA,   draw=colorProxA]   {};} & proximal ($\varepsilon=1.5$) \\
                \tikz{\node[proxnode,   fill=colorProxB,   draw=colorProxB]   {};} & proximal ($\varepsilon=2$) \\
                \tikz{\node[equivnode,  fill=colorEquiv,   draw=colorEquiv]   {};} & coincident ($\alpha=0.5,\ \varepsilon=1$) \\
                \tikz{\node[draw, circle, fill=cyan, inner sep=1.6pt] {};}     & $\pi_t^{\mathrm{ref}}$ (reference) \\
                \tikz{\node[draw, rectangle, inner sep=2pt, thick, fill=black] {};} & $\pi_t^{\mathrm{out}}$ (RMP optimum)
            \end{tabular}
        };
    \end{tikzpicture}

    \caption{Geometric illustration of the equivalence under the collinearity condition ($b \parallel d=\pi_t^{\mathrm{out}}-\pi_t^{\mathrm{ref}}$). The shaded triangle is the dual feasible set $D_R=\{(x,y): x\ge0,\ y\ge0,\ x+y\le1\}$. We use $\pi_t^{\mathrm{ref}}=(0.20,0.25)$, $\pi_t^{\mathrm{out}}=(1,0)$, and set $b=d$ to satisfy Proposition~\ref{prop:prox-smooth-equiv}. The dashed segment is the smoothing path $\{\alpha\,\pi_t^{\mathrm{ref}}+(1-\alpha)\pi_t^{\mathrm{out}}:\alpha\in[0,1]\}$. As the proximal strength $\varepsilon$ \emph{decreases}, the proximal solution $\pi_t^{\mathrm{pen}}(\varepsilon)$ (diamonds) moves \emph{from} $\pi_t^{\mathrm{ref}}$ \emph{towards} $\pi_t^{\mathrm{out}}$ along this segment and saturates at $\pi_t^{\mathrm{out}}$; as $\varepsilon$ \emph{increases}, it moves back towards $\pi_t^{\mathrm{ref}}$. Symmetrically, as $\alpha$ \emph{decreases}, the smoothing point (triangles) travels from $\pi_t^{\mathrm{ref}}$ to $\pi_t^{\mathrm{out}}$, and as $\alpha$ \emph{increases} it returns to $\pi_t^{\mathrm{ref}}$. The hexagon marks a coincident point where the two methods agree exactly ($\alpha=0.5,\ \varepsilon=1,\ t=0.5$).}
    \label{fig:prox-geom-equiv}
\end{figure}

% This equivalence is formalized in the following proposition.

\begin{proposition}
\label{prop:prox-smooth-equiv}
Let $\pi_t^{\mathrm{pen}} $ be the unique solution to the Euclidean-proximal penalization problem \eqref{eq:f_pen3} with $\theta_t = \{\varepsilon\}$.
Under \eqref{eq:geom-condition}, $\pi_t^{\mathrm{pen}}$ lies on the line segment $[\pi_t^{\mathrm{ref}}, \pi_t^{\mathrm{out}}]$ and satisfies 
\begin{equation}
\label{eq:prox-equals-smooth}
\pi_t^{\mathrm{pen}}
=\alpha_t\,\pi_t^{\mathrm{ref}}+(1-\alpha_t)\,\pi_t^{\mathrm{out}}
\end{equation}
where the smoothing parameter $\alpha_t$ is determined by the penalization parameter $\varepsilon$ as
\begin{equation}
\label{eq:eps-to-alpha-mapping}
\alpha_t \;=\; \max\left\{0, \ 1-\frac{b^\top d}{2\,\varepsilon\,\|d\|_2^2}\right\}.
\end{equation}
Conversely, for any target smoothing parameter $\alpha_t\in[0,1)$, choosing the penalization parameter $\varepsilon$ such that
\begin{equation}
\label{eq:alpha-to-eps-mapping}
\varepsilon \;=\; \frac{b^\top d}{2(1-\alpha_t)\,\|d\|_2^2}.
\end{equation}
% ensures that the penalization operator \eqref{eq:f_pen3} and the smoothing operator \eqref{eq:f_smooth} coincide.
\end{proposition}

\section{Proofs}\label{sec:proof}

\begin{proof}{Proof of Lemma \ref{lem:neutral}}

First, we show that standard column generation is recovered when the stabilization strength is zero. 
For smoothing, it follows directly from \eqref{eq:f_smooth}.
For penalization, let $D_R:=\{\pi\in\mathbb{R}^m:\ A_R^\top\pi\le c_R\}$ be the RMP dual feasible region. Because the dual objective is decreasing in the slack variables, an optimal solution always attains the minimal feasible values $w_{+,i} = (\pi_i-\pi_{t,i}^{\mathrm{ref}})_+$ and $w_{-,i} = (\pi_{t,i}^{\mathrm{ref}}-\pi_i)_+$, where $(x)_+:=\max\{x,0\}$. Thus, the penalization operator is equivalent to the maximization:
$$\pi_t^{\mathrm{sep}} \in \arg\max_{\pi\in D_R}\Big\{ b^\top\pi -\langle \varepsilon_t^-,(\pi_t^{\mathrm{ref}}-\pi)_+\rangle -\langle \varepsilon_t^+,(\pi-\pi_t^{\mathrm{ref}})_+\rangle \Big\}.$$
When $\varepsilon_t^- = \varepsilon_t^+ = 0$, the penalized objective reduces to $b^\top\pi$. This perfectly recovers the standard RMP dual, meaning any maximizer is an optimal RMP dual solution.

Beyond this common case:
\begin{enumerate}
    \item For smoothing, if $\pi_t^{\mathrm{ref}} = \pi_t^{\mathrm{out}}$, $f_t^{\mathrm{smooth}}(\pi_t^{\mathrm{out}}, \alpha_t; \mathrm{RMP}_t) = \pi_t^{\mathrm{out}}$ follows directly from \eqref{eq:f_smooth}.
    \item For penalization, assume $\pi_t^{\mathrm{ref}}$ is an optimal RMP dual solution. For any $\pi\in D_R$, the penalty terms in the equivalent maximization are nonnegative, yielding the upper bound:
    $$b^\top\pi -\langle \varepsilon_t^-,(\pi_t^{\mathrm{ref}}-\pi)_+\rangle -\langle \varepsilon_t^+,(\pi-\pi_t^{\mathrm{ref}})_+\rangle \le b^\top\pi.$$
    Since $\pi_t^{\mathrm{ref}}$ is optimal for $\max_{\pi\in D_R} b^\top\pi$, we have $b^\top\pi \le b^\top\pi_t^{\mathrm{ref}}$ for all $\pi\in D_R$. At $\pi_t^{\mathrm{ref}}$, the penalty terms evaluate to zero, so the penalized objective equals $b^\top\pi_t^{\mathrm{ref}}$. Consequently, $\pi_t^{\mathrm{ref}}$ attains the global maximum of the penalized objective, ensuring that any maximizer returned by $f_t^{\mathrm{pen}}$ is an optimal dual solution of the RMP.

% 2. For polyhedral penalization, write the stabilized dual objective as 
% \[
% \max_{A_R^\top\pi\le c_R}\ b^\top\pi 
% -\big\langle \varepsilon_t^-,(\delta_t^--\pi)_+\big\rangle
% -\big\langle \varepsilon_t^+,(\pi-\delta_t^+)_+\big\rangle .
% \]
% On the compact feasible set, as $\varepsilon_t^\pm\to 0$ the objective converges uniformly to $b^\top\pi$, 
% so the maximizer converges to the RMP dual optimizer $\pi_t^{\mathrm{out}}$.
% If $\pi_t^{\mathrm{out}}$ satisfies $\delta_t^-\le \pi_t^{\mathrm{out}}\le \delta_t^+$, the penalty vanishes 
% at $\pi_t^{\mathrm{out}}$ and the stabilized solution equals $\pi_t^{\mathrm{out}}$.
% 3. For Euclidean–proximal penalization \eqref{eq:f_pen3}, strong concavity ensures a unique optimizer; 
% as $\varepsilon\downarrow 0$, the objective uniformly converges to $b^\top\pi$, whence 
% $f_t^{\mathrm{pen}}(\pi_t^{\mathrm{ref}},\theta_t)\to \pi_t^{\mathrm{out}}$. 
% If $\pi_t^{\mathrm{ref}}=\pi_t^{\mathrm{out}}$, the penalty is minimized at $\pi_t^{\mathrm{out}}$, 
% and the stabilized solution is $\pi_t^{\mathrm{out}}$. 

\end{enumerate}
\end{proof}

\begin{proof}{Proof of Proposition \ref{prop:local-consistency-nos-stall}}

For any column $q\in Q$, the reduced cost is 
$r_q(\pi)=c_q-A_q^\top\pi$. Hence, for any $\pi,\pi'\in\mathbb{R}^m$,

\begin{equation}
    |r_q(\pi) - r_q(\pi')| = |(\pi' - \pi)^\top A_q| \le \|A_q\|_2 \|\pi - \pi'\|_2 \le \bar{a} \|\pi - \pi'\|_2.
\end{equation}
Let $\Delta = \|\pi_t^{\mathrm{sep}} - \pi_t^{\mathrm{out}}\|_2$. By assumption, 
\[
\Delta \le \frac{\min\{\gamma_t^-, \gamma_t^+\}}{2\bar{a}}.
\]
Pick a most-negative-reduced-cost column under $\pi_t^{\mathrm{out}}$:
\[
q_t^\star \in \arg\min_{q\in Q} r_q(\pi_t^{\mathrm{out}}),
\qquad
r_{q_t^\star}(\pi_t^{\mathrm{out}})=\sigma(\pi_t^{\mathrm{out}})=-\gamma_t^-.
\]
Evaluating it at $\pi_t^{\mathrm{sep}}$ yields

\[
r_{q_t^\star}(\pi_t^{\mathrm{sep}})
\le r_{q_t^\star}(\pi_t^{\mathrm{out}})+\bar a\,\Delta
= -\gamma_t^- + \bar a\,\Delta
\le -\gamma_t^- + \bar a\cdot\frac{\gamma_t^-}{2\bar a}
= -\frac{\gamma_t^-}{2}<0.
\]
So there exists at least one column with negative reduced cost at $\pi_t^{\mathrm{sep}}$.

If $\mathcal Q_+(\pi_t^{\mathrm{out}})=\varnothing$, then every column belongs to 
$\mathcal Q_-(\pi_t^{\mathrm{out}})$ and the claim is immediate. 
Henceforth assume $\mathcal Q_+(\pi_t^{\mathrm{out}})\neq\varnothing$.
Now take any non-improving column $q_+\in\mathcal{Q}_+(\pi_t^{\mathrm{out}})$.
By definition, $r_{q_+}(\pi_t^{\mathrm{out}})\ge \gamma_t^+$, hence
\[
r_{q_+}(\pi_t^{\mathrm{sep}})
\ge r_{q_+}(\pi_t^{\mathrm{out}})-\bar a\,\Delta
\ge \gamma_t^+ - \bar a\,\Delta
\ge \gamma_t^+ - \bar a\cdot\frac{\gamma_t^+}{2\bar a}
= \frac{\gamma_t^+}{2}\ge 0.
\]
Therefore, every column in $\mathcal{Q}_+(\pi_t^{\mathrm{out}})$ has nonnegative reduced cost at
$\pi_t^{\mathrm{sep}}$, while at least one column in $Q$ has negative reduced cost at $\pi_t^{\mathrm{sep}}$.
Since $Q$ is partitioned as
$Q=\mathcal{Q}_-(\pi_t^{\mathrm{out}})\cup\mathcal{Q}_+(\pi_t^{\mathrm{out}})$, any minimizer
$q(\pi_t^{\mathrm{sep}})\in\arg\min_{q\in Q} r_q(\pi_t^{\mathrm{sep}})$ must belong to
$\mathcal{Q}_-(\pi_t^{\mathrm{out}})$.

Consequently, $r_{q(\pi_t^{\mathrm{sep}})}(\pi_t^{\mathrm{out}})<0$ and mispricing does not occur.
\end{proof}

\begin{proof}{Proof of Corollary \ref{prop:safe-bound-smoothing}}

Recall the smoothing operator defined in \eqref{eq:f_smooth}: $\pi_t^{\mathrm{sep}} = \alpha_t \pi_t^{\mathrm{ref}} + (1-\alpha_t) \pi_t^{\mathrm{out}}$.
The Euclidean distance between the stabilized point and the RMP solution is:
\[
\|\pi_t^{\mathrm{sep}} - \pi_t^{\mathrm{out}}\|_2 = \|\alpha_t (\pi_t^{\mathrm{ref}} - \pi_t^{\mathrm{out}})\|_2 = \alpha_t \|\pi_t^{\mathrm{ref}} - \pi_t^{\mathrm{out}}\|_2.
\]
To prevent mispricing, Proposition \ref{prop:local-consistency-nos-stall} requires $\|\pi_t^{\mathrm{sep}} - \pi_t^{\mathrm{out}}\|_2 \le \frac{\min\{\gamma_t^-, \gamma_t^+\}}{2\bar{a}}$.
Substituting the distance expression:
\[
\alpha_t \|\pi_t^{\mathrm{ref}} - \pi_t^{\mathrm{out}}\|_2 \le \frac{\min\{\gamma_t^-, \gamma_t^+\}}{2\bar{a}}.
\]
Solving for $\alpha_t$ yields the condition stated in the corollary.
\end{proof}

\begin{proof}{Proof of Proposition \ref{prop:safe-bound-penalization}}

If $\pi_t^{\mathrm{ref}} = \pi_t^{\mathrm{out}}$, the stabilization penalty vanishes at the optimum, yielding $\pi_t^{\mathrm{sep}} = \pi_t^{\mathrm{out}}$ and preventing mispricing. 

Assume $\pi_t^{\mathrm{ref}} \neq \pi_t^{\mathrm{out}}$. The stabilized dual $\pi_t^{\mathrm{sep}}$ maximizes the penalized objective $b^\top\pi - P_t(\pi)$ over $D_R$, where the penalty term is defined as $P_t(\pi) := \langle \varepsilon_t^-,(\pi_t^{\mathrm{ref}}-\pi)_+\rangle + \langle \varepsilon_t^+,(\pi-\pi_t^{\mathrm{ref}})_+\rangle \ge 0$. 
By the optimality of $\pi_t^{\mathrm{sep}}$ evaluated against the true dual optimum $\pi_t^{\mathrm{out}}$, we have \(b^\top \pi_t^{\mathrm{sep}} - P_t(\pi_t^{\mathrm{sep}}) \ge b^\top \pi_t^{\mathrm{out}} - P_t(\pi_t^{\mathrm{out}}).\)
Since $P_t(\pi_t^{\mathrm{sep}}) \ge 0$, rearranging gives a bound on the optimality gap:
$$b^\top \pi_t^{\mathrm{out}} - b^\top \pi_t^{\mathrm{sep}} \le P_t(\pi_t^{\mathrm{out}}).$$
% Invoking \eqref{eq:error_bound_unique} for the unique optimum $\pi_t^{\mathrm{out}}$, we have $b^\top \pi_t^{\mathrm{out}} - b^\top \pi_t^{\mathrm{sep}} \ge \lambda_t\|\pi_t^{\mathrm{sep}} - \pi_t^{\mathrm{out}}\|_2$. 
Combining this with  \eqref{eq:error_bound_unique} yields:
$$\|\pi_t^{\mathrm{sep}} - \pi_t^{\mathrm{out}}\|_2 \le \frac{P_t(\pi_t^{\mathrm{out}})}{\lambda_t}.$$
Next, we bound $P_t(\pi_t^{\mathrm{out}})$ using $\varepsilon_\infty = \max_i \{\varepsilon_{t,i}^-, \varepsilon_{t,i}^+\}$. Since $w_i^- = (\pi_{t,i}^{\mathrm{ref}} - \pi_i)_+$ and $w_i^+ = (\pi_i - \pi_{t,i}^{\mathrm{ref}})_+$, we note that $w_i^- + w_i^+ = |\pi_i - \pi_{t,i}^{\mathrm{ref}}|$. Thus:
$$P_t(\pi_t^{\mathrm{out}}) \le \sum_{i=1}^m \varepsilon_\infty |\pi_{t,i}^{\mathrm{out}} - \pi_{t,i}^{\mathrm{ref}}| = \varepsilon_\infty \|\pi_t^{\mathrm{out}} - \pi_t^{\mathrm{ref}}\|_1 \le \varepsilon_\infty \sqrt{m}\,\|\pi_t^{\mathrm{out}} - \pi_t^{\mathrm{ref}}\|_2.$$
Substituting this into our distance bound gives:
$$\|\pi_t^{\mathrm{sep}} - \pi_t^{\mathrm{out}}\|_2 \le \frac{\varepsilon_\infty \sqrt{m}\,\|\pi_t^{\mathrm{out}} - \pi_t^{\mathrm{ref}}\|_2}{\lambda_t}.$$
By Proposition~\ref{prop:local-consistency-nos-stall}, if the dual deviation $\|\pi_t^{\mathrm{sep}} - \pi_t^{\mathrm{out}}\|_2$ is bounded by the safe radius $\frac{\min\{\gamma_t^-,\gamma_t^+\}}{2\bar{a}}$, then $r_{q(\pi_t^{\mathrm{sep}})}(\pi_t^{\mathrm{out}}) < 0$ and mispricing is prevented. Enforcing this upper bound on the derived distance yields the stated condition on $\varepsilon_\infty$, completing the proof.

 \end{proof}

\begin{proof}{Proof of Proposition \ref{prop:smooth-fix-hybrid}}

The finite iteration \(t\le T_0\) does not affect convergence. 
After the switch, the method coincides with the standard reference points update scheme.
By \cite[Proposition~2]{pessoa2018automation}, on any stalled sequence, this scheme eventually realigns $\pi_t^{\mathrm{sep}}$ with $\pi_t^{\mathrm{out}}$ enough for pricing to return a separating column. Thus, it finally satisfies Lemma~\ref{lem:convergence}, which implies finite termination with an optimal solution.

\end{proof}

\begin{proof}{Proof of Proposition \ref{prop:prox-smooth-equiv}}

Let $D_R:=\{\pi\in\mathbb{R}^m:\ A_R^\top\pi\le c_R\}$, which is nonempty, closed, and convex. Fix $\pi_t^{\mathrm{ref}}\in D_R$, and let \(\pi_t^{\mathrm{out}}\in\arg\max_{\pi\in D_R} \ b^\top \pi\). By definition, $d=\pi_t^{\mathrm{out}}-\pi_t^{\mathrm{ref}}$. If $d=0$, then $\pi_t^{\mathrm{ref}}=\pi_t^{\mathrm{out}}$ and the claim is immediate. 
Henceforth assume $d\neq 0$, and suppose that $b$ is aligned with $d$, i.e., $b=\kappa d$ for some $\kappa>0$.

Consider the penalized dual problem
\begin{equation}\label{eq:prox-pen}
\max_{\pi\in D_R}\ \Big\{ \phi(\pi):= b^\top\pi - \varepsilon\|\pi-\pi_t^{\mathrm{ref}}\|_2^2 \Big\},
\qquad \varepsilon>0.
\end{equation}
Because $\phi$ is concave and $D_R$ is closed and convex, \eqref{eq:prox-pen} admits a unique maximizer, denoted $\pi_t^{\mathrm{pen}}$.
Take any $\pi\in D_R$ and decompose
\[
\pi-\pi_t^{\mathrm{ref}}=\lambda d + u,
\qquad
\lambda:=\frac{d^\top(\pi-\pi_t^{\mathrm{ref}})}{\|d\|_2^2}\in\mathbb{R},
\quad
u\perp d.
\]
Using $b=\kappa d$ and $u\perp d$, we have $b^\top u=\kappa d^\top u=0$ and therefore
\begin{equation}\label{eq:phi-decomp-OR}
\phi(\pi)
=
b^\top \pi_t^{\mathrm{ref}}
+
\lambda\, b^\top d
-
\varepsilon\Big(\lambda^2\|d\|_2^2 + \|u\|_2^2\Big).
\end{equation}

We first show that any maximizer must satisfy $\lambda\in[0,1]$.
If $\lambda<0$, then $\lambda(b^\top d)\le 0$ and $\|\pi-\pi_t^{\mathrm{ref}}\|_2^2>0$, so
\[
\phi(\pi)-\phi(\pi_t^{\mathrm{ref}})
=
\lambda(b^\top d)-\varepsilon\|\pi-\pi_t^{\mathrm{ref}}\|_2^2
<0,
\]
Since $\pi_t^{\mathrm{ref}}\in D_R$ is feasible, the optimal value satisfies
$\phi(\pi_t^{\mathrm{pen}})=\max_{\pi\in D_R}\phi(\pi)\ge \phi(\pi_t^{\mathrm{ref}})$, and hence any feasible point with $\lambda<0$ cannot be optimal. 
Therefore, $\lambda\ge 0$ at $\pi_t^{\mathrm{pen}}$.
Moreover, if a feasible $\pi\in D_R$ had $\lambda>1$, then by $b^\top u=0$ we would obtain
\[
b^\top \pi
=
b^\top \pi_t^{\mathrm{ref}}+\lambda\, b^\top d
>
b^\top \pi_t^{\mathrm{ref}}+ b^\top d
=
b^\top \pi_t^{\mathrm{out}},
\]
contradicting the definition of $\pi_t^{\mathrm{out}}$ as a maximizer of $b^\top\pi$ over $D_R$. 
Hence $\lambda\le 1$ for any feasible $\pi$, and in particular for $\pi_t^{\mathrm{pen}}$.

Now fix $\lambda\in[0,1]$ and define the projection onto the line through $\pi_t^{\mathrm{ref}}$ and $\pi_t^{\mathrm{out}}$:
\[
\pi(\lambda):=\pi_t^{\mathrm{ref}}+\lambda d=(1-\lambda)\pi_t^{\mathrm{ref}}+\lambda \pi_t^{\mathrm{out}}.
\]
By convexity of $D_R$, $\pi(\lambda)\in D_R$. Comparing \eqref{eq:phi-decomp-OR} for $\pi$ and $\pi(\lambda)$ yields
\[
\phi(\pi(\lambda))-\phi(\pi)=\varepsilon\|u\|_2^2\ge 0,
\]
with strict inequality if $u\neq 0$. Hence the unique maximizer must satisfy $u=0$, i.e.,
\[
\pi_t^{\mathrm{pen}}=\pi_t^{\mathrm{ref}}+\lambda d
\quad \text{for some } \lambda\in[0,1].
\]

Substituting $\pi=\pi_t^{\mathrm{ref}}+\lambda d$ into \eqref{eq:prox-pen} gives
\[
\max_{\lambda\in[0,1]}
\ \Big\{
b^\top(\pi_t^{\mathrm{ref}}+\lambda d)-\varepsilon\lambda^2\|d\|_2^2
\Big\}
=
\text{const}+
\max_{\lambda\in[0,1]}
\ \Big\{
\lambda(b^\top d)-\varepsilon\lambda^2\|d\|_2^2
\Big\}.
\]
The unconstrained maximizer of this concave quadratic is $\lambda^\star=\frac{b^\top d}{2\varepsilon\|d\|_2^2}.$
% \[
% \lambda^\star=\frac{b^\top d}{2\varepsilon\|d\|_2^2}.
% \]
Therefore,
\[
\lambda_{\mathrm{opt}}=\min\{1,\lambda^\star\},
\qquad
\pi_t^{\mathrm{pen}}=\pi_t^{\mathrm{ref}}+\lambda_{\mathrm{opt}}(\pi_t^{\mathrm{out}}-\pi_t^{\mathrm{ref}})
=(1-\lambda_{\mathrm{opt}})\pi_t^{\mathrm{ref}}+\lambda_{\mathrm{opt}}\pi_t^{\mathrm{out}}.
\]
Setting $1-\alpha_t=\lambda_{\mathrm{opt}}$ gives \eqref{eq:prox-equals-smooth} and 
\eqref{eq:eps-to-alpha-mapping}. Solving for $\varepsilon$ yields \eqref{eq:alpha-to-eps-mapping}. 
\end{proof}

\section{State Features}\label{appendix:column feature}

\begin{figure}[htbp]
\centering
\begingroup
\usetikzlibrary{positioning}
\begin{tikzpicture}[
    x=1cm,y=1cm,
    scale=0.65,
    transform shape,
    >=latex,
    font=\small,
    vnode/.style={draw, circle, minimum size=0.85cm, line width=1pt, inner sep=0pt},
    cnode/.style={draw, circle, minimum size=0.85cm, line width=1pt, inner sep=0pt},
    edge/.style={line width=0.8pt},
    lab/.style={font=\normalsize}
]

% Titles
\node[lab] at (1.9,6.2) {Column set};
\node[lab] at (6.7,6.2) {Constraint set};

% Left nodes
\node[vnode] (v1)  at (1.9,5.1) {$v_1$};
\node[vnode] (v2)  at (1.9,4.0) {$v_2$};
\node[vnode] (v3)  at (1.9,2.9) {$v_3$};
\node            at (1.9,1.8) {$\vdots$};
\node[vnode] (vmt) at (1.9,0.7) {$v_{m_t}$};

% Right nodes
\node[cnode] (c1) at (6.7,4.4) {$c_1$};
\node[cnode] (c2) at (6.7,3.1) {$c_2$};
\node            at (6.7,1.95) {$\vdots$};
\node[cnode] (cn) at (6.7,0.7) {$c_n$};

% Edges
\draw[edge] (v1) -- (c1);
\draw[edge] (v1) -- (c2);

\draw[edge] (v2) -- (c2);
\draw[edge] (v2) -- (cn);

\draw[edge] (v3) -- (c1);
\draw[edge] (v3) -- (cn);

\draw[edge] (vmt) -- (c2);
\draw[edge] (vmt) -- (cn);

\end{tikzpicture}
\endgroup
% \caption{Illustration of the bipartite graph representation of the restricted master problem at iteration $t$. The left partition contains column nodes corresponding to the current RMP variables, denoted by $\mathcal{V}_t=\{v_1,\dots,v_{m_t}\}$, where $m_t$ is the number of columns currently present in the RMP at iteration $t$. The right partition contains constraint nodes corresponding to the master constraints, denoted by $\mathcal{C}=\{c_1,\dots,c_n\}$. An edge $(v,c)$ is present if and only if the coefficient of column $v$ in constraint $c$ is nonzero. This graph serves as the structural input to the GNN encoder in the RL state representation.}
\caption{Bipartite graph representation of the restricted master problem at iteration $t$. Column nodes $\mathcal{V}_t=\{v_1,\dots,v_{m_t}\}$ represent the current RMP variables, and constraint nodes $\mathcal{C}=\{c_1,\dots,c_n\}$ represent the master constraints. An edge $(v,c)$ is present iff the coefficient of column $v$ in constraint $c$ is nonzero. The graph serves as the structural input to the GNN encoder.}
\label{fig:bipartite_rmp_graph}
\end{figure}
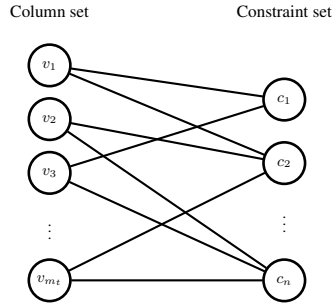

Figure~\ref{fig:bipartite_rmp_graph} illustrates the bipartite graph representation used to encode the RMP in Section~\ref{subsubsec:state}. 
% Column nodes correspond to variables currently present in the RMP, constraint nodes correspond to master constraints, and an edge is included whenever the corresponding coefficient in the RMP matrix is nonzero.
% Below, we provide the complete specification of features used in the state representation $s_t$.
The node features for the bipartite graph representation $G_t$ are detailed in Table~\ref{appendix:node_features}, while the global features are presented in Table~\ref{appendix:global_features}.

\begin{table}[h!]
\scriptsize
\centering
\caption{Node features for the bipartite graph representation $G_t$.}
\label{appendix:node_features}
\resizebox{\textwidth}{!}{%
\begin{tabular}{l l p{8.5cm}}
\toprule
\textbf{Component} & \textbf{Feature} & \textbf{Rationale / Interpretation} \\
\midrule
\multicolumn{3}{l}{\textbf{Column Node Features ($\mathbf{x}_v$)}} \\
\textit{Basic Properties} & Waste& Unused capacity of the cutting pattern.\\
& Degree& Number of constraints covered by the column.\\
& Basis status & Binary indicator (1 if $v$ is in the RMP basis, 0 otherwise). \\
& $r_v(\pi^{\mathrm{sep}})$& Reduced cost under $\pi^{\mathrm{sep}}$ at the iteration of addition.\\
& $r_v(\pi^{\mathrm{out}})$& Reduced cost under $\pi^{\mathrm{out}}$ at introduction (computed for smoothing only).\\
\addlinespace
\textit{Historical Information}& Age& Number of iterations since the column was generated.\\
& Basis transitions & Count of times the column entered or left the basis.\\
\midrule
\multicolumn{3}{l}{\textbf{Constraint Node Features ($\mathbf{x}_c$)}} \\
\textit{Properties} & Right-hand side ($b_c$) & \\
& Degree& Number of columns covering the constraint.\\
\addlinespace
\textit{Dual Information}& Current dual ($\pi_{c,t}^{\mathrm{sep}}$) & Component $c$ of the stabilized dual vector $\pi_t^{\mathrm{sep}}$. \\
& Historical duals (5 iters) & $\pi_{c,t-1}^{\mathrm{sep}}, \dots, \pi_{c,t-5}^{\mathrm{sep}}$.\\
& Dual value variance & Variance of the last 5 historical dual values.\\
\bottomrule
\end{tabular}
}
\end{table}

\begin{table}[h!]
\centering
\scriptsize
\caption{Global features for the state representation ($\mathbf{g}_t$).}
\label{appendix:global_features}
\resizebox{\textwidth}{!}{%
\begin{tabular}{l l p{8.5cm}}
\toprule
\textbf{Component} & \textbf{Feature} & \textbf{Rationale / Interpretation} \\
\midrule
\multicolumn{3}{l}{\textbf{Problem-Specific Features}} \\
\textit{Problem Scale} & Number of item types &  \\
& Raw material roll length & \\
\addlinespace
\multicolumn{3}{l}{\textbf{Algorithmic State Features}} \\
\textit{Progress} & Iteration count ($t$) & Tracks overall algorithm progress. \\
& Relative progress & Normalized improvement in the progress metric. For smoothing, this metric is the RMP objective value; for penalization, it is the objective value of the stabilized solution, excluding the penalty term.\\
\addlinespace
\textit{Stabilization State} & Previous stabilization parameter ($\theta_{t-1}$) & The last action chosen by the agent (e.g., $\alpha_{t-1}$ or $\varepsilon_{t-1}$). \\
& $\|\pi_t^{\mathrm{sep}} - \pi_t^{\mathrm{ref}}\|_2$ & Euclidean distance; measures the influence of the reference point. \\
& $\|\pi_t^{\mathrm{sep}} - \pi_{t-1}^{\mathrm{sep}}\|_2$ & Euclidean distance; measures dual solution stability/oscillation. \\
\addlinespace
\textit{Mispricing Proxy} & Mispricing count & Cumulative number of mispricing iterations; computed for smoothing only. \\
& Average minimum reduced cost (5 iters) & Average $\sigma(\pi_{t'}^{\mathrm{sep}})$ for $t'=t-5,\ldots,t-1$. \\
& Average relative progress (5 iters) & Over the last 5 iterations, computed from the same progress measure used above. \\
& Iterations without progress & Count of consecutive iterations without measurable progress. \\
\bottomrule
\end{tabular}
}
\end{table}

% In the column node features, $r_v(\pi^{\mathrm{out}})$ refers to the reduced cost of the column generated using $\pi^{\mathrm{sep}}$ evaluated with respect to the \emph{true} RMP dual solution $\pi^{\mathrm{out}}$. 
% By Definition~\ref{def:mispricing}, if this value is nonnegative, then the added column fails to separate $\pi^{\mathrm{out}}$ at this iteration and a mispricing occurs. 
% Accordingly, the global feature \textit{Mispricing count} records the cumulative number of such iterations.
% We note that this feature is only computed for the smoothing method because for penalization, the algorithm solves a stabilized RMP and $\pi^{\mathrm{out}}$ is not naturally available without additional computation of the true RMP.
Among the column-node features, $r_v(\pi^{\mathrm{out}})$ denotes the reduced cost of a column generated under $\pi^{\mathrm{sep}}$, evaluated at the \emph{true} RMP dual solution $\pi^{\mathrm{out}}$. By Definition~\ref{def:mispricing}, a nonnegative value means that the added column fails to separate $\pi^{\mathrm{out}}$ and hence corresponds to mispricing. 
The global feature \textit{Mispricing count} records the cumulative number of such iterations. This feature is used only for smoothing, since under penalization $\pi^{\mathrm{out}}$ is not directly available without additionally solving the true RMP.

% Several features are designed to provide the RL agent with a view of the stabilization dynamics, including: \textit{dual value variance} to capture dual oscillation patterns, and \textit{Average minimum reduced cost} monitors pricing behavior under the stabilized dual. \textit{Relative progress} and \textit{Consecutive iterations without progress} are used to detect stalling. For smoothing, progress is measured by the change in the RMP objective value. For penalization, since the true unpenalized RMP objective is not directly available, we use the original objective component of the penalized solution, excluding the realized penalty contribution, as a computationally efficient proxy for current progress. Finally, \textit{Previous stabilization parameter} and $\|\pi_t^{\mathrm{sep}} - \pi_t^{\mathrm{ref}}\|_2$ help the agent reason about its recent action and its immediate effect on dual deviation.

Several features are included to summarize stabilization dynamics. \textit{Dual value variance} captures dual oscillation, and \textit{Average minimum reduced cost} summarizes pricing behavior under $\pi^{\mathrm{sep}}$. \textit{Relative progress} and \textit{Consecutive iterations without progress} indicate stalling. For smoothing, progress is measured by the change in the RMP objective; for penalization, where the true unpenalized RMP objective is unavailable, we use the original objective component of the penalized solution, excluding the realized penalty term, as a proxy. Finally, \textit{Previous stabilization parameter} and $\|\pi_t^{\mathrm{sep}}-\pi_t^{\mathrm{ref}}\|_2$ help the agent assess its recent action and its immediate effect on dual deviation.

% Additionally, the \textit{Number of item types} and \textit{Raw material roll length} are problem-dependent features specifically designed for the cutting stock problem. These can be replaced with other problem-specific features when applying the RLSCG framework to other problems.
% All features are normalized to the range $[-1, 1]$ before being input to the graph neural network to ensure stable training.
The \textit{Number of item types} and \textit{Raw material roll length} are CSP-specific features and can be replaced by other problem-specific quantities in different applications. All features are normalized to $[-1,1]$ before being passed to the graph neural network.

% \section{Cutting Stock Problem}\label{appendix:CSP}
\section{Synthetic Test Instance Generation}\label{appendix:CSP}

Table~\ref{tab:data_generation} summarizes the parameter ranges used to generate the synthetic test sets in Section~\ref{subsec:setup}.

\begin{table}[htbp]
\centering
\scriptsize
\caption{Parameters for synthetic test instance generation.}
\label{tab:data_generation}
\renewcommand{\arraystretch}{1.2}
\begin{tabular}{lcc}

\toprule
\textbf{Group} & \textbf{Lower Bound ($lb$)} & \textbf{Upper Bound ($ub$)} \\
Gen\_1 & $\mathcal{U}[0.05, 0.45]$ & $\mathcal{U}[0.50, 0.85]$ \\
Gen\_2 & $\mathcal{U}[0.35, 0.45]$ & $\mathcal{U}[0.85, 0.95]$ \\
Gen\_3 & $\mathcal{U}[0.25, 0.45]$ & $\mathcal{U}[0.75, 0.95]$ \\
Gen\_4 & $\mathcal{U}[0.05, 0.15]$& $\mathcal{U}[0.15, 0.30]$ \\
\bottomrule
\multicolumn{3}{l}{\footnotesize \textit{Note:} For all groups, roll length $L \sim \mathcal{U}[800, 1000]$ and demand $d_i \sim \mathcal{U}[1, 20]$.} \\
\multicolumn{3}{l}{\footnotesize Size categories: Small ($m \in [20,50]$), Medium ($m \in [50,100]$), Large ($m \in [100,150]$).}
\end{tabular}
\end{table}

\section{GNN Structure and Hyperparameter}\label{appendix:hyper}

% Here, we detail the Graph Neural Network architecture used for state representation and the hyperparameters used for training the Deep Q-Network agent.

% The state $s_t$ is represented as a bipartite graph $G_t = (\mathcal{V}_t \cup \mathcal{C}, E_t)$, as defined in Section \ref{subsubsec:state}. We employ a message-passing GNN to encode this graph structure into a fixed-size embedding.
% The input feature vectors for constraint nodes ($\mathbf{x}_c$) and column nodes ($\mathbf{x}_v$) are constructed based on the specifications provided in Appendix \ref{appendix:column feature}. The input dimension for constraint nodes is $d_c=9$. For column nodes, the input dimension depends on the stabilization method: $d_v=7$ for the smoothing method and $d_v=6$ for penalization methods.
% The complete architectural specifications are summarized in Table \ref{tab:gnn-arch}.

Here we summarize the GNN architecture and training hyperparameters. The state $s_t$ is represented by the bipartite graph $G_t=(\mathcal{V}_t\cup\mathcal{C},E_t)$ from Section~\ref{subsubsec:state}.
We employ a message-passing GNN to encode this graph structure into a fixed-size embedding.
Node features are defined in Appendix~\ref{appendix:column feature}, and the full architecture is reported in Table~\ref{tab:gnn-arch}.

\begin{table}[h!]
\centering
\scriptsize
\caption{GNN Architecture Specifications}
\label{tab:gnn-arch}
\begin{tabular}{ll}
\toprule
\textbf{Component} & \textbf{Specification} \\
\midrule
Input Dimension ($\mathbf{x}_c$) & 9 \\
Input Dimension ($\mathbf{x}_v$) & 7 (Smoothing) / 6 (Penalization) \\
Global Feature Dimension & 11(Smoothing) / 10 (Penalization)\\
Hidden Dimension & 32 \\
Message Passing Layers & 3 \\
Message Functions & 2-layer MLP with ReLU \\
Aggregation & Degree-normalized message aggregation followed by mean-pooling over all nodes\\
State Vector Dimension & 64\\
Output Actions & 20 \\
\bottomrule
\end{tabular}
\end{table}

The RLSCG agent is trained using the Deep Q-Network (DQN) algorithm \citep{mnih2015human} equipped with experience replay and a target network to ensure training stability. Optimization is performed using the Adam optimizer with a batch size of 64.
% To ensure a fair comparison with state-of-the-art learning-based column generation methods, we adopt the hyperparameter configuration reported by \cite{chi2022deep}. We do not perform separate hyperparameter tuning for our model; instead, we directly utilize the best-performing settings from the literature to isolate the contribution of our stabilization framework.
% The specific settings are:
% \begin{itemize}
%     \item \textit{Reward Scaling:} The improvement reward coefficient is set to $\beta = 300$.
%     \item \textit{Iteration Penalty:} The per-iteration penalty is set to $\eta = 1$.
%     \item \textit{Terminal Reward:} The terminal reward is set to $R_{\text{conv}} = 10$.
%     \item \textit{Exploration:} We use an $\epsilon$-greedy strategy where $\epsilon$ decays linearly to a final value of $0.05$.
%     \item \textit{Discount Factor:} The discount factor for future rewards is $\gamma = 0.9$.
%     \item \textit{Learning Rate:} The learning rate is set to $lr = 0.001$.
% \end{itemize}
To ensure a fair comparison with prior learning-based column generation methods, we adopt the hyperparameter settings of \cite{chi2022deep} without further tuning. Specifically, we set $\beta=300$, $\eta=1$, $R_{\mathrm{conv}}=10$, $\gamma=0.9$, learning rate $lr=0.001$, and use an $\epsilon$-greedy policy with linear decay to $\epsilon=0.05$.

The reinforcement learning framework was implemented in Python 3.11.7 using PyTorch 2.6.0. The Gurobi Optimizer 10.0.1 was employed to solve both the restricted master problem and the pricing subproblems. All computational experiments were conducted on a Linux cluster node equipped with an AMD EPYC 7763 Processor. The training phase required approximately 12 hours. For the evaluation of smoothing and penalization methods, we trained a single model on the training set and evaluated its performance on separate testing sets to assess generalization capabilities.

\section{Detailed Experimental Results}
\label{app:detailed_results}

\subsection{Impact of Predicted Duals in Smoothing}
\label{appendix: Predicted Duals in Smoothing}

We provide comprehensive performance distributions for the reference point strategies discussed in Section~\ref{subsec:predicted_dual_smoothing}. While the main text reports average performance metrics for large-scale instances, the boxplots presented in Figure \ref{fig:scg_time}-\ref{fig:rlscg_time} show the variance and robustness of each strategy (Standard, Static Predicted Reference, and Adaptive Predicted Reference) across all test instances.

The visual results support three key observations in the main text.
For ASCG-1 and ASCG-2, the Standard strategy often exhibits large interquartile ranges and extreme outliers. 
% S-WS and A-WS significantly compresses these distributions, particularly in the difficult instances.
% Also, while both S-WS and A-WS improve upon the Standard strategy, A-WS displays larger variances in several instances. 
% \edit{
Incorporating predicted duals improves upon the Standard strategy across many test sets by compressing these distributions. 
% }
However, no clear superiority is observed between them, which stems from the sensitivity of the algorithm to the specific stagnation threshold used to trigger the switch. 
% The RLSCG method maintains compact distributions even under the Standard strategy, demonstrating that the marginal visual improvement from predicted duals is less pronounced than in ASCG-1 and ASCG-2. 
% \edit{
Finally, the RLSCG method maintains compact distributions even under the Standard strategy, and does not appear to gain much from the inclusion of predicted duals.
% }

% version 3: all in one figrue.

% \begin{figure}[htbp]
%     \centering
%     \captionsetup[subfigure]{justification=centering}

%     \begin{subfigure}[b]{0.75\textwidth}
%         \centering
%         \includegraphics[width=\textwidth]{figures/SCG_time.png}
%         \caption{ASCG-1}
%         \label{fig:scg_time}
%     \end{subfigure}

%     \vspace{0.8em}

%     \begin{subfigure}[b]{0.75\textwidth}
%         \centering
%         \includegraphics[width=\textwidth]{figures/ASCG_time.png}
%         \caption{ASCG-2}
%         \label{fig:ascg_time}
%     \end{subfigure}

%     \vspace{0.8em}

%     \begin{subfigure}[b]{0.75\textwidth}
%         \centering
%         \includegraphics[width=\textwidth]{figures/RLSCG_time.png}
%         \caption{RLSCG(S)}
%         \label{fig:rlscg_time}
%     \end{subfigure}

%     \caption{Computation time distributions for the smoothing methods under different reference-point strategies. Panels (a)--(c) report results for ASCG-1, ASCG-2, and RLSCG(S), respectively. In each panel, the x-axis represents the datasets, and the three boxes per dataset correspond to Standard, Static Predicted Reference, and Adaptive Predicted Reference, respectively.}
%     \label{fig:time_distributions_all}
% \end{figure}

% version 2: just time (seperate figures)

\begin{figure}[htbp]
    \centering
    \includegraphics[width=1\textwidth]{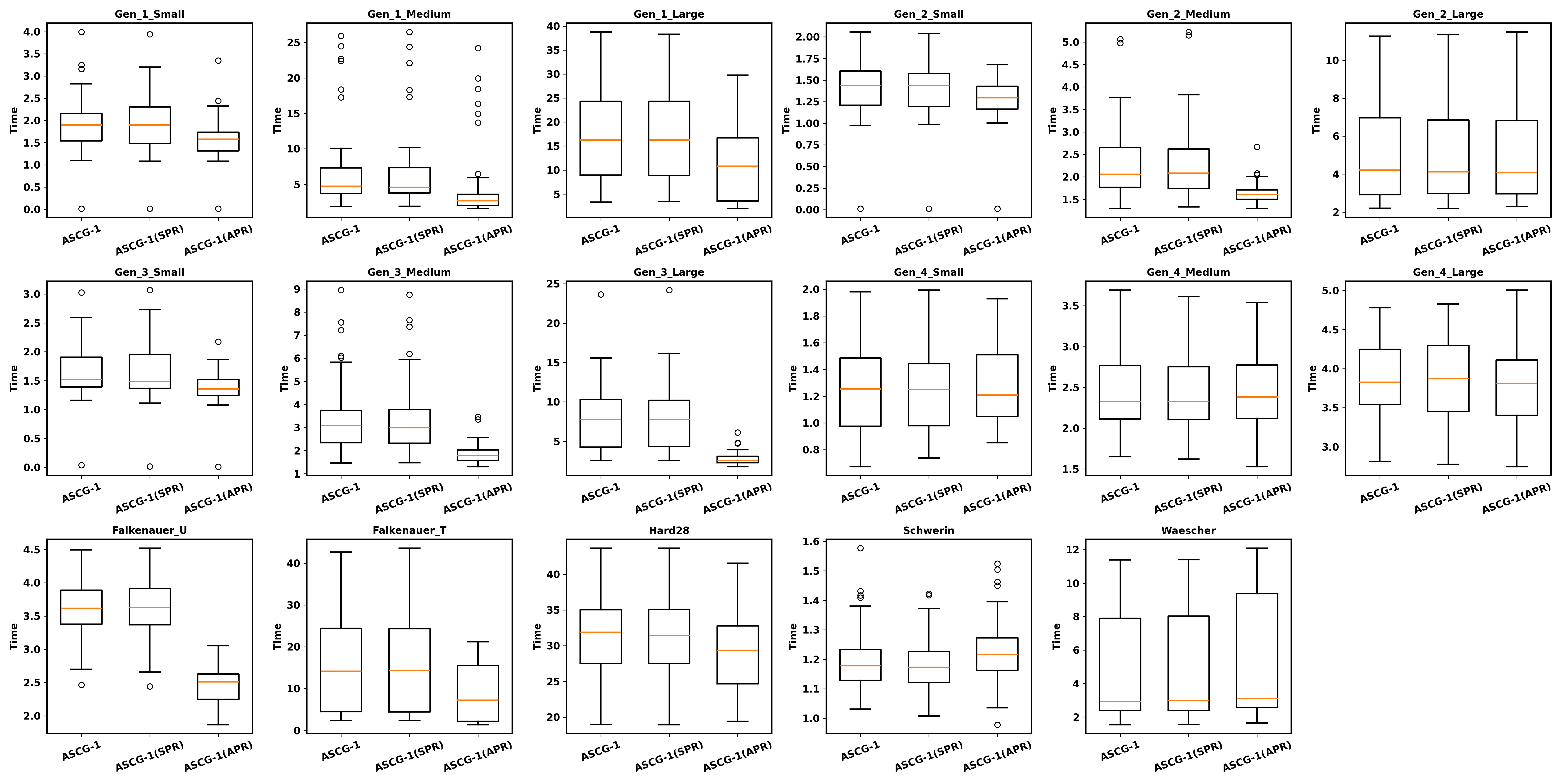}
    \caption{Computation time distributions for ASCG-1. The x-axis represents the datasets, and the three boxes per dataset correspond to Standard, Static Predicted Reference, and Adaptive Predicted Reference, respectively.}
    \label{fig:scg_time}
\end{figure}

\begin{figure}[htbp]
    \centering
    \includegraphics[width=1\textwidth]{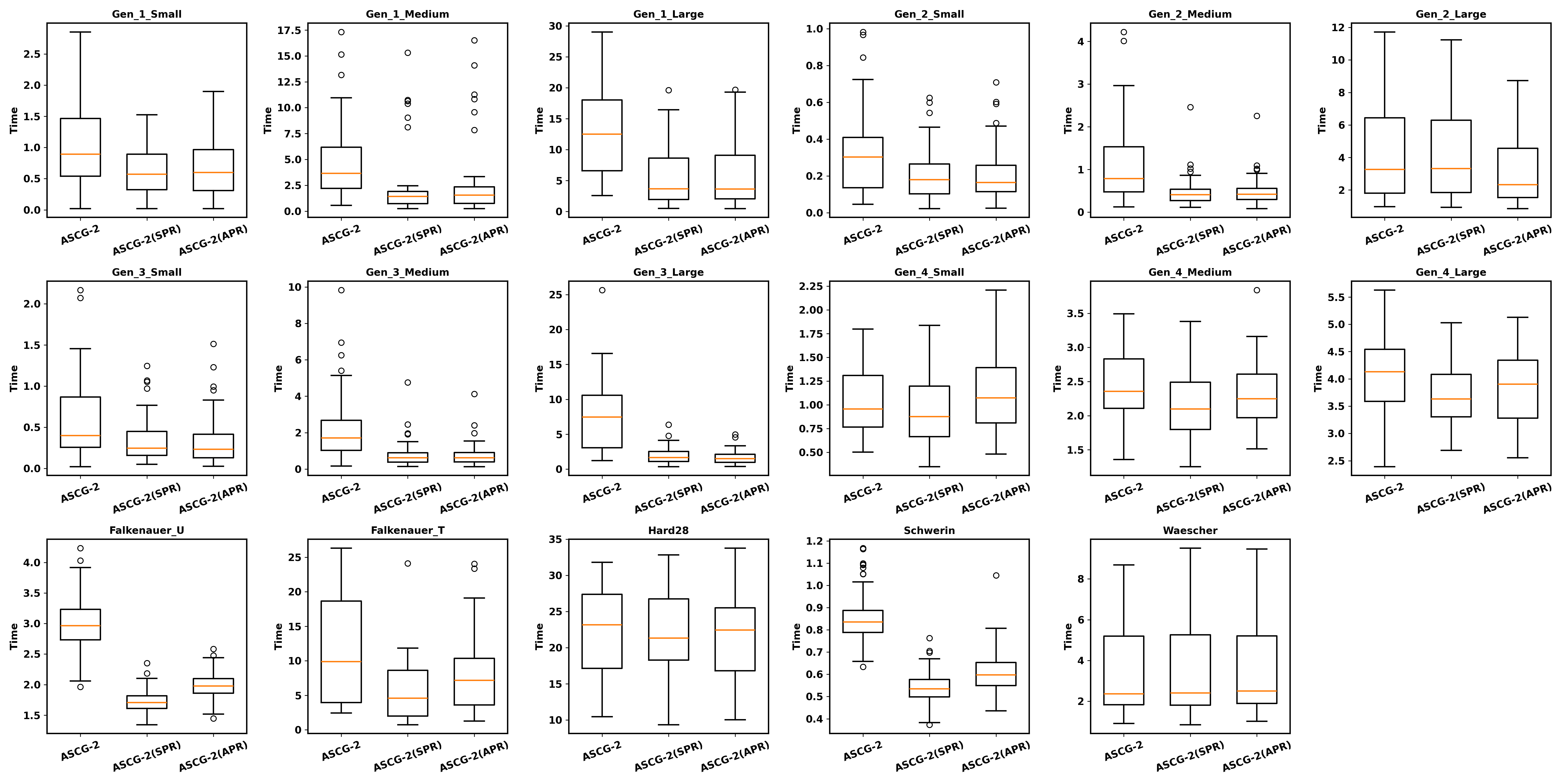}
    \caption{Computation time distributions for ASCG-2. The x-axis represents the datasets, and the three boxes per dataset correspond to Standard, Static Predicted Reference, and Adaptive Predicted Reference, respectively.}
    \label{fig:ascg_time}
\end{figure}

\begin{figure}[htbp]
    \centering
    \includegraphics[width=1\textwidth]{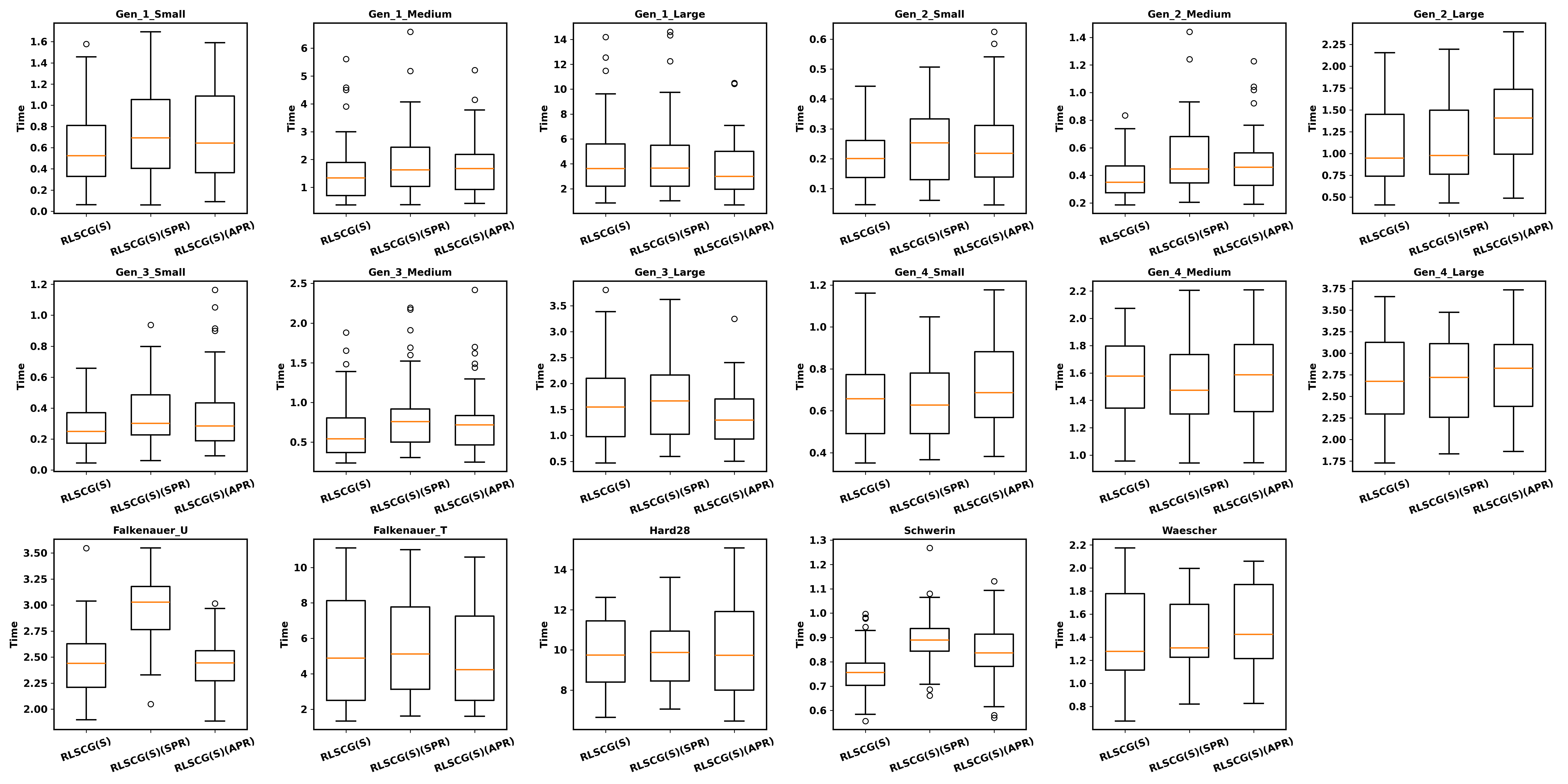}
    \caption{Computation time distributions for RLSCG(S). The x-axis represents the datasets, and the three boxes per dataset correspond to Standard, Static Predicted Reference, and Adaptive Predicted Reference, respectively.}
    \label{fig:rlscg_time}
\end{figure}

\end{document}